\documentclass[11pt,a4paper,oneside]{amsart}
\usepackage[a4paper,margin=1in]{geometry}
\usepackage[utf8]{inputenc} % allow utf-8 input
\usepackage[T1]{fontenc}    % use 8-bit T1 fonts
\usepackage{hyperref}       % hyperlinks
\usepackage{url}            % simple URL typesetting
\usepackage{booktabs}       % professional-quality tables
\usepackage{amsfonts}       % blackboard math symbols
\usepackage{nicefrac}       % compact symbols for 1/2, etc.
\usepackage{microtype}      % microtypography
\usepackage{xcolor}         % colors
\usepackage{graphicx}
\usepackage{float}
\usepackage{multirow}
\usepackage{subcaption}
\usepackage{enumitem}

\usepackage{amsmath, amssymb, amsthm}
\numberwithin{equation}{section}

\theoremstyle{plain}
\newtheorem{theorem}{Theorem}[section]
\newtheorem{lemma}[theorem]{Lemma}
\newtheorem{proposition}[theorem]{Proposition}

\theoremstyle{definition}

\newtheorem{assumption}[theorem]{Assumption}
\newtheorem{example}[theorem]{Example}

\theoremstyle{remark}
\newtheorem{remark}[theorem]{Remark}

\newcommand{\para}[1]{%
  \par\addvspace{0.8\baselineskip}%
  \noindent\textbf{#1.}\enspace\ignorespaces%
}

\title[Entropy Regularization Improves Robustness in Continuous-Time RL]{Entropy Regularization Improves Policy Robustness in Continuous-Time Reinforcement Learning
}

\author{Jialun Cao$^{1,*}$}
\author{Fernando Acero$^{2,*}$}
\author{David \v{S}i\v{s}ka$^1$}
\author{Yufei Zhang$^3$}

\thanks{$^{*}$Equal contribution.}
\thanks{$^{1}$School of Mathematics, University of Edinburgh, United Kingdom. \texttt{\{galen.cao\}\{d.siska\}@ed.ac.uk}}
\thanks{$^{2}$J.P. Morgan AI Research, United Kingdom. \texttt{fernando.acero@jpmorgan.com}}
\thanks{$^{3}$Department of Mathematics, Imperial College London, United Kingdom. \texttt{yufei.zhang@imperial.ac.uk}}

\begin{document}

\begin{abstract}
Entropy regularization is widely used in continuous-time reinforcement learning (RL) to reduce sensitivity to environmental perturbations, yet its robustness benefits lack a rigorous theoretical foundation. This paper establishes the first robustness guarantees for entropy-regularized continuous-time Markov decision processes. We show that maximizing an entropy-regularized objective yields a lower bound on a worst-case robust RL problem with joint reward and transition perturbations. We analytically characterize the induced robust sets and prove that they expand monotonically with the regularization strength, justifying the empirical observation that stronger entropy improves robustness. In contrast to prior discrete-time analyses, our results remove the intractable state-distribution entropy term and provide guarantees invariant to action frequency. Experiments on queueing network control and market making confirm our theory, showing that entropy-regularized policies outperform greedy and $\epsilon$-greedy baselines under dynamics perturbations.

\end{abstract}

\keywords{Entropy regularization; policy robustness; robustness in reinforcement learning; continuous-time reinforcement learning; continuous-time Markov decision processes}

\maketitle
% ======================================================================
\section{Introduction}\label{sec:intro}
% ======================================================================

Policy robustness is a central requirement for deploying reinforcement learning (RL) in real-world systems. Policies trained in simulation or under a nominal model often face environments at deployment that differ from those seen during training, due to model misspecification, unmodeled dynamics, non-stationarity, human biases, or even adversarial disturbances~\cite{eysenbach2022maximum}. Such sim-to-real and train-to-deploy gaps can severely degrade performance if the learned policy is brittle~\cite{tobin2017domain, peng2018sim, aljalbout2025reality}. Ensuring that policies maintain reliable performance under perturbations of the environment is therefore a key challenge in modern RL.

One approach is robust RL, which models uncertainty via a set of plausible transition kernels or reward functions and seeks policies that perform well in the worst case over this set  ~\cite{morimoto2005robust,osogami2012robustness,lim2013reinforcement, pinto2017robust,tessler2019action}. 
This leads to a two-player game between a controller and an adversary that perturbs the dynamics and rewards. While this framework yields explicit worst-case guarantees, it requires specifying a suitable uncertainty set and solving a more challenging minimax problem, often with specialized algorithms and increased sample and computational cost~\cite{nilim2005robust,wiesemann2013robust,chen2024robust,liang2022efficient,bukharin2023robust}. 

In parallel, entropy-regularized RL has emerged as a  widely used alternative that often exhibits robust behavior without explicitly modeling an adversary or solving a minimax problem. By augmenting the reward objective with an entropy term, entropy-regularized methods encourage stochastic policies that explore more broadly, improve optimization stability, and empirically tend to be less sensitive to perturbations in the environment. Such methods underlie many state-of-the-art algorithms in  large-scale RL~\cite{schulman2015trust, schulman2017proximal}, where entropy regularization is primarily viewed as a tool for   numerical stability, yet is widely observed to improve robustness in practice.

Recent theoretical work in discrete-time Markov decision processes (MDPs) has begun to formalize this phenomenon. Several authors have shown that entropy-regularized RL can be interpreted as optimizing robust objectives under suitable perturbation models for the reward and transition kernel~\cite{eysenbach2022maximum, brekelmans2022your,derman2023robustness,ashlag2025state}. These results suggest that robustness may arise implicitly from entropy regularization, without the algorithmic and modeling overhead of explicit robust RL.

However, most of this theory is developed in discrete-time settings, whereas many applications of RL, especially in physical systems, robotics, finance, and inventory management, are more naturally modeled in continuous time. In such domains, the environment state can change at any time, control acts continuously, and performance is evaluated via continuous-time functionals. Continuous-time RL provides a natural framework for these problems, avoiding the need to artificially discretize time and approximate the underlying dynamics, as well as the potential performance degradation that can arise when applying algorithms designed for discrete-time MDPs~\cite{jia2022policy_eval, jia2022policy_grad,
jia2023q, zhao2023policy,
cheng2026deterministic}. In these settings, entropy regularization is often used explicitly in algorithm design, yet its robustness benefits have not been rigorously analyzed. This raises the following question:
\begin{quote}
 \emph{To what extent does entropy regularization in continuous-time RL provide   robustness guarantees, and against which classes of environment perturbations?}
\end{quote}
Answering this question requires techniques beyond those used in the discrete-time setting, as the structure of continuous-time dynamics can fundamentally change how entropy regularization interacts with model uncertainty. In fact, as we shall show, existing robust guarantees for discrete-time RL can degenerate as the number of time steps increases (Example \ref{example:robust set}).

\para{Our Contribution}
In this paper, we take an initial step toward addressing this gap by establishing robustness guarantees  of entropy-regularized   continuous-time MDPs, where  the state dynamics is given by a controlled continuous-time Markov chain. Such problems arise naturally in various domains, such as network queuing control \cite{dai2022queueing}, market making for high-frequency trading \cite{cartea2015algorithmic}, or supply-chain management \cite{puterman1994mdp}. 
Our   main contributions are summarized as follows:
\begin{itemize}[leftmargin=10mm]
    \item We  show that the entropy-regularized objective  with a   log-transformed reward  provides a lower bound for a robust RL problem over a set  of joint reward and state perturbations (Theorems \ref{thm:ct_global_sa_robust} and \ref{thm:maxent_robust_ctmdp}). 
    In contrast to   existing robustness results for discrete-time MDPs~\cite{eysenbach2022maximum, ashlag2025state}, our  reward  does not involve a difficult-to-compute entropy term of the state distribution, making it applicable to general RL tasks (Remark \ref{rmk:pessimistic_reward}). In the   case   with reward-only perturbations, we show  that entropy-regularized RL is exactly equivalent to a robust RL problem (Theorem \ref{thm:exact_reward_only_robustness}).
\item We   provide two complementary characterizations of the corresponding robust sets: one in terms of induced state-occupancy measures, and another in terms of \emph{the local entropy rate} of the state transitions. 
Both   sets  expand as the entropy regularization parameter increases, thereby justifying empirical observations that stronger entropy regularization improves policy robustness (Propositions \ref{prop:monotonicity_tau_global} and 
\ref{prop:monotonicity_tau}).
We further prove that robust-set constructions proposed in \cite{eysenbach2022maximum} for discrete-time RL degenerate to an empty set as the number of time steps increases and
 thus fail in the continuous-time limit  (Example \ref{example:robust set}).
% \item Our numerical experiments on queueing network control confirm that entropy regularization yields more robust policies than both greedy policies for a nominal environment model and $\epsilon$-greedy policies (Section \ref{sec:experiments}). 
\item Numerical experiments on queueing network control and market making confirm 
our theory:
entropy regularization yields more robust policies than both greedy policies for a nominal environment model and $\epsilon$-greedy policies (Section \ref{sec:experiments} and Appendix~\ref{app:robustness}) and validate the tightness of the theoretical certificates (Appendix~\ref{app:certificate}).
\end{itemize}

% =================================================================
% \textbf{Most Related Works.}%\label{sec:related}
\subsection{Closely Related Works}
% =================================================================
Given the extensive research literature on RL, we focus on the specific subfield that is most relevant to our work.

\para{Robust RL}
In RL, robustness has traditionally been studied via robust MDPs, where policies optimize worst-case performance over uncertainty sets for rewards or transition kernels \cite{nilim2005robust,iyengar2005robust,wiesemann2013robust}. While these methods provide  theoretical guarantees against model misspecification, they typically require inner adversarial optimization and are computationally expensive. Recent work proposes more scalable approaches for particular robustness sets \cite{grandclement2021scalable, wang2023policy,kumar2023policy,zhou2023natural,li2026policy}, but results largely remain for discrete-time MDPs. Continuous-time robust RL is far less developed and has primarily been studied through $H_\infty$
  control and differential-game formulations for nonlinear controlled ODEs \cite{morimoto2005robust,perrusquia2021continuous,lutter2021robust}. 
  
To the best of our knowledge, there is no prior work on continuous-time robust RL for stochastic dynamics. Moreover, as in discrete time, these formulations are computationally challenging, often requiring time discretization and inner minimax updates \cite{perrusquia2021continuous}.

\para{Entropy-regularized RL and robustness}
Entropy-regularization is  widely used to  improve exploration, probabilistic inference, and optimization stability~\cite{todorov2006linearly,rawlik2013stochastic,geist2019theory,haarnoja2018soft,lan2023mirror,kerimkulov2025fisher} (see Appendix \ref{app:extended_related} for   applications in continuous-time RL). Recent work connects robustness and entropy regularization, but primarily in discrete-time MDPs. Under reward uncertainty, policy regularization admits an exact robust interpretation  against worst-case reward perturbations~\cite{husain2021regularized,brekelmans2022your}. Under transition uncertainty, however, ordinary policy-entropy regularization typically provides only a lower bound on a robust control objective rather than an exact equivalence~\cite{eysenbach2022maximum,derman2023robustness}, and robust formulations may require value-dependent regularization (e.g., twice-regularized MDPs~\cite{derman2021twice}). State-entropy regularization further strengthens this view by yielding a nontrivial, generally unimprovable lower bound under kernel uncertainty~\cite{ashlag2025state}.

However, these robustness bounds often involve an intractable state-distribution entropy term in the reward, limiting applicability to general RL settings. Moreover, these discrete-time   guarantees can degrade as the step size shrinks, preventing direct extension to continuous-time dynamics.

\section{Problem Formulation} \label{sec: preliminaries}

\para{Entropy-regularized CTMDP}
Consider an infinite-horizon continuous-time Markov decision process (CTMDP)  
defined by the tuple
$
\big(\mathcal S, \mathcal A, (\mathcal A_s)_{s\in\mathcal S}, \lambda, R\big),
$ 
as in \cite{guo2009continuous}:    $\mathcal S$  and  $\mathcal A$  are finite state and action spaces, respectively. For each state $ s\in\mathcal S$,    $\mathcal A_s\subseteq\mathcal A $ is a  nonempty set, representing the    available actions in state  $s$. 
For each   state--action pair  $(s,a)\in \mathcal S\times\mathcal A_s$,   $ \lambda^a_{ss'}$  is  the transition rate  from state  $s$  to state $ s'$   under action $a$, which satisfies 
$
 \lambda^a_{ss'} 
 \ge 0$ for all $s'\neq s$, and  
$\lambda^a_{ss} = -\sum_{s'\neq s}\lambda^a_{ss'}
$.    
$ 
R:\mathcal S\times\mathcal A \to (0, \infty)
$ 
is the instantaneous reward  function.

We consider optimizing a discounted objective over stationary randomized Markov policies. 
 Let $\mathcal P(\mathcal A)$ be the space of probability measures on $\mathcal A$,
 and   define the admissible policy class by 
$$\Pi
    :=
    \big\{
        \pi : \mathcal S \to \mathcal P(\mathcal A) 
        \,|\,
        \sum_{a\in A_s}\pi(a  | s)=1,
        \ \forall s \in \mathcal S
    \big\} \, .$$ 
Let $\rho \in \mathcal P(\mathcal S)$  be the  initial state distribution. 
For each 
$\pi\in \Pi$,
let 
$\mathbb P^\pi$  
denote the law of the controlled continuous-time Markov chain  (CTMC) $(S_t)_{t \ge 0}$  initialized according to $\rho$, 
with the generator matrix 
$\lambda^\pi =  (\lambda^\pi_{ss'})_{s,s' \in \mathcal S}$
induced by $\pi$,  
where 
$\lambda^\pi_{ss'} := \sum_{a \in \mathcal A_s} \pi(a | s)\lambda^a_{ss'}$ for all $s,s' \in \mathcal S$.
 Let  $r>0$ be the  discount rate and $\tau\ge 0$ be the  regularization parameter.
We seek a policy $\pi\in\Pi$ that maximizes the entropy-regularized objective
\begin{equation}
\label{eq:def_J_tau}
\begin{aligned}
    J_{\tau}(\pi)
    &:=
    \mathbb E_\rho^{\mathbb P^\pi}
    \bigg[
        \int_0^\infty e^{-rt}
        \bigg(
            \sum_{a \in \mathcal A_{S_t}}
            \pi(a | S_t)\, R(S_t,a)
            -
            \tau\,\mathrm{KL}(\pi \,\|\, \mu)(S_t)
        \bigg)
        \mathrm dt
    \bigg]\,,
\end{aligned}
\end{equation}
where
  $\mu: \mathcal S \to \mathcal P(\mathcal A) $ is  a reference policy
  and 
$\mathrm{KL}(\pi\|\mu)$ is 
the Kullback-Leibler (KL) divergence between $\pi$ and $\mu$ defined by  $$\mathrm{KL}(\pi\|\mu)(s) :=  \mathrm{KL}\big( \pi (\cdot | s)\|\mu (\cdot | s) \big) =\sum_{a\in\mathcal A_s} \pi(a| s)\log\tfrac{\pi(a| s)}{\mu(a| s)} \, ,$$
with the convention $0\log(0/x)=0$. 
For simplicity, we assume   $\operatorname{supp}\mu(\cdot | s) = \mathcal A_s$ for all $s \in \mathcal S$, 
  and 
  for each $\pi \in \Pi$,
  define the normalized   discounted state occupancy measure
  $\bar{\mathrm d}_{\rho}^\pi\in \mathcal P(\mathcal S)$
  by 
 $$ 
    \bar{\mathrm d}_{\rho}^\pi(s)
    :=
    r \mathbb E_\rho^{\mathbb P^\pi}
    \Big[
        \int_0^\infty e^{-rt}\mathbf 1_{\{S_t=s\}}\,\mathrm dt
    \Big], \quad \forall s \in \mathcal S \, .
    $$

The objective $ J_{\tau}(\pi)$ has a natural event-driven implementation, in which an agent interacts with a controlled CTMC at Poisson-distributed decision epochs and samples actions from 
$\pi$ based on the observed states;  see Appendix  \ref{appx: sampling methods difference} and  \cite[Ch. 11]{puterman1994mdp}
for details.

\para{Policy Robustness in  CTMDP} 
The main goal of this paper is to show that maximizing the entropy-regularized objective \eqref{eq:def_J_tau} yields policies with robustness guarantees. We measure  the robustness of a policy $\pi\in \Pi$
by  its performance under perturbed instantaneous reward functions and perturbed CTMCs with modified transition rates.

 Specifically, 
 we consider a parameterized family of transition rates and reward functions indexed by shared parameters 
 $\theta\in \Theta$.
 This allows us to model structured perturbations that simultaneously affect both the dynamics and the rewards, as is common in many CTMDP applications (see Section \ref{sec:worked_example}).
 We denote dependence on $\theta\in \Theta$ by a subscript,
  for example in the transition rates
 $\lambda^a_\theta = (\lambda^a_{ss',\theta})_{s,s'\in\mathcal S}$, 
  the reward function $R_\theta=(R_\theta(s,a))_{(s,a)\in \mathcal S\times \mathcal A}$, and  
  the  law of the CTMC $\mathbb{P}^\pi_{\theta}$.
   Moreover, we require that for any $\theta \in \Theta$, the parameterized components maintain the standard CTMDP structure defined previously; in particular, the reward function satisfies $R_\theta(s,a) > 0$. 
For any given $\pi\in \Pi$, 
we define the   robust   objective
in terms of a  robust set 
$\mathcal C \subset 
 \Theta$:
\begin{equation}
\label{eq:def_robust_objective}
\begin{aligned}
    \inf_{\tilde\theta\in \mathcal C} \,
    \mathbb E_\rho^{\mathbb P^\pi_{\tilde\theta}}
    \bigg[
        \int_0^\infty e^{-rt}
        \sum_{a \in \mathcal A_{S_t}}
        \pi(a | S_t)\,R_{\tilde\theta}(S_t,a)\,
        \mathrm dt
    \bigg]\,.
\end{aligned}
\end{equation}
Finding an optimal policy that maximizes \eqref{eq:def_robust_objective} can be computationally expensive, as it requires solving a zero-sum game.
 In the sequel, we show that the entropy-regularized objective \eqref{eq:def_J_tau} for
a  baseline model provides a lower bound on the robust objective \eqref{eq:def_robust_objective}, and   further characterize the corresponding robust set $\mathcal C$
 analytically in terms of the baseline model, the policy, and the entropy-regularization parameter. This implies that a policy maximizing the entropy-regularized objective \eqref{eq:def_J_tau} implicitly guarantees robustness over a family of adversarial models.

\section{Main Theoretical Results}
\label{sec: main_results}

This section presents robustness guarantees for entropy-regularized CTMDPs.
Our main results show that the entropy-regularized objective with a simple pessimistic reward provides a lower bound for a robust RL problem with respect to a family of joint reward and state perturbations. Unlike existing robustness results, our pessimistic reward does not involve a difficult-to-compute state-entropy. We provide two   characterizations of the robust sets: one in terms of induced state-occupancy measures, and another in terms of the local entropy rate of the state transitions. We further show that robust-set constructions from discrete-time RL degenerate as the number of time steps increases and do not   extend to the continuous-time setting.
In the sequel,  let   $\theta\in \Theta$ be a fixed   baseline (nominal) model.

\subsection{Robustness under Reward-only Perturbations}

We first establish a strong duality between the entropy-regularized objective and a robust RL problem in the setting where the perturbation affects only the reward function. This result will then be used to prove the robustness theorem under joint perturbations of both the dynamics and the reward functions.
Since the perturbations $\tilde\theta\in\Theta$ only affect the reward, we have $\mathbb P_{\tilde\theta}^{\pi} = \mathbb P_{\theta}^{\pi}$.

\begin{theorem}
\label{thm:exact_reward_only_robustness}
Let  $\pi\in\Pi$, $\tau>0$ and $\varepsilon \ge 0$. 
Define the reward-only robust set by
\begin{equation}
\label{eq:reward_only_exact_set}
\begin{aligned}
    \widehat{\mathcal C}_{\tau,\varepsilon}^\pi(\theta)
    :=
    \Big\{
        \tilde\theta \in \Theta
        \,\Big|\,
        \tau \log
        \bigg[
        \sum_{s \in \mathcal S}
        \bar{\mathrm d}_{\rho,\theta}^\pi(s)
        \sum_{a \in \mathcal A_s}
        \mu(a | s)
        \exp\bigg(
            \tfrac{
                R_\theta(s,a)-R_{\tilde\theta}(s,a)
            }{\tau}
        \bigg)
        \bigg]
        \le \varepsilon
    \Big\}\,.
\end{aligned}
\end{equation}
Let $J_{\tau}(\pi;\theta)$
be the entropy-regularized objective 
\eqref{eq:def_J_tau}
under the baseline model $\theta$.
Then
\begin{equation}
\label{eq:reward_only_exact_lower_bound}
\begin{aligned}
    \inf_{\tilde\theta \in  \widehat{\mathcal C}_{\tau,\varepsilon}^\pi(\theta)}
    \mathbb E_\rho^{\mathbb P^\pi_{\theta}}
    \bigg[
        \int_0^\infty e^{-rt}
        \sum_{a \in \mathcal A_{S_t}}
        \pi(a | S_t)\,R_{\tilde\theta}(S_t,a)\,
        \mathrm dt
    \bigg]
    \geq
    J_{\tau}(\pi;\theta)
    -
    \frac{\varepsilon}{r}.
\end{aligned}
\end{equation}
where an equality is achieved  
if 
there exists $\tilde \theta^* \in \Theta$ such that 
$ 
    R_{\tilde \theta^*}(s,a)
    =
    R_{\theta}(s,a)
    -
    \tau \log \frac{\pi(a | s)}{\mu(a | s)}
    -
    \varepsilon \, .
$ 
In this case, 
$\tilde \theta^*$ is a worst-case model. 
\end{theorem}

Theorem
\ref{thm:exact_reward_only_robustness} shows that the optimal policy for the entropy-regularized objective \eqref{eq:def_J_tau}
is also optimal in the worst-case scenario with reward uncertainty    given in 
\eqref{eq:reward_only_exact_set}, which 
extends similar results for discrete-time maximum-entropy MDPs in 
\cite{eysenbach2022maximum, ashlag2025state} to the continuous-time setting. The proof is given in Appendix~\ref{appendix: proof of the reward -only}.

By Theorem \ref{thm:exact_reward_only_robustness},
the policy obtained from an entropy-regularized objective is robust to reward perturbations, as long as the new reward does not fall too far below the original one, quantified by  smoothed maximum   over the state–action space.
It is easy to show 
$ \widehat{\mathcal C}_{\tau,\varepsilon}^\pi(\theta)$  expands as $\tau$ increases,   indicating that entropy regularization improves policy robustness. The parameter $\varepsilon$ 
 represents the adversary’s budget for allowable perturbations: as 
$\varepsilon$ increases, the set of admissible perturbations grows, but the corresponding worst-case lower bound decreases.

\subsection{
Robustness under Joint Perturbations of Dynamics and Rewards
}

We now turn to robustness guarantees for entropy-regularized CTMDPs when the perturbation also affects the state dynamics.
As shown in \cite[Theorem 4.1]{eysenbach2022maximum}, 
the exact equivalence to a robust RL formulation, as in Theorem  \ref{thm:exact_reward_only_robustness}, does not hold in this setting. Instead, we will show that the entropy-regularized objective provides a lower bound on the worst-case value over a family of environments.
 
\para{Robustness Guarantee with State Occupancy Measure}
The first theorem quantifies robustness in terms of perturbations to the induced state-occupancy distribution.

\begin{theorem}
\label{thm:ct_global_sa_robust}
For all  $\pi \in \Pi$, $\tau>0$, and $\varepsilon\ge 0$, 
\small
\begin{equation}
\label{eq:robust_occup_measure}
\begin{aligned}
    &\inf_{\tilde\theta\in\widetilde {\mathcal C}^\pi_{\tau,\varepsilon}(\theta)}
    \mathbb E_{\rho}^{\mathbb P^\pi_{\tilde\theta}}
    \bigg[
        \int_0^\infty e^{-rt} \sum_{a \in \mathcal A_{S_t}}
        \pi(a | S_t)\, R_{\tilde\theta}(S_t, a)\,\mathrm dt
    \bigg] \\
    &\ge
    \exp\bigg(
        r \, \mathbb E_\rho^{\mathbb P^\pi_\theta}
        \bigg[
            \int_0^\infty e^{-rt}
            \Big(\sum_{a \in \mathcal A_{S_t}}
        \pi(a | S_t)\, 
                \log R_\theta(S_t, a)
             -
               \tau\,\mathrm{KL}(\pi\|\mu)(S_t)
            \Big)\,\mathrm dt
        \bigg]
        -
        \varepsilon - \log r
    \bigg)\,,
\end{aligned}
\end{equation}
\normalsize
where
 the robust set
 $ \widetilde {\mathcal C}^\pi_{\tau,\varepsilon}(\theta)$ is defined as
 \small
\begin{equation}
\label{eq:robust_set_entropy}
    \widetilde{\mathcal C}_{\tau,\varepsilon}^\pi(\theta)
    :=
    \Big\{
        \tilde\theta\in\Theta
        \,\Big|\,
        \tau\log
        \bigg[
        \sum_{s\in\mathcal S}
        \bar{\mathrm d}_{\rho,\theta}^\pi(s)
        \sum_{a\in\mathcal A_s}
        \mu(a| s)
        \exp\bigg(
            \frac{
                f_{\theta,\tilde\theta}(s,a)
            }{\tau}
        \bigg)
        \bigg]
        \le \varepsilon
    \Big\}\,,
\end{equation}
\normalsize
where 
$f_{\theta,\tilde\theta}(s,a):=
 \log
                \tfrac{
                    \bar{\mathrm d}_{\rho,\theta}^\pi(s)
                }{
                    \bar{\mathrm d}_{\rho,\tilde\theta}^\pi(s)
                }
                -
                \log
                \frac{
                    R_{\tilde\theta}(s,a)
                }{
                    R_\theta(s,a)
                }
$.
\end{theorem}

Theorem~\ref{thm:ct_global_sa_robust}
provides the first theoretical robustness guarantee for entropy regularization in continuous-time RL,
whose   proof is given in Appendix~\ref{appx: proof of proposition ct_global_sa}.
It shows that  a policy that performs well across a range of dynamics can be obtained by   solving an entropy-regularized problem with a pessimistic log-transformed reward, and, conversely, that solving an entropy-regularized problem is implicitly  solving a robust RL problem under a modified reward function.

\begin{remark}[\textbf{Pessimistic reward without state entropy}]
\label{rmk:pessimistic_reward}
    The key feature is that 
a  pessimistic reward   \eqref{eq:robust_occup_measure} is  a simple log-transformation of the original reward.
This is achieved by constructing 
the 
robust set $\widetilde{\mathcal C}^{\pi}_{\tau,\varepsilon}(\theta)$ 
using a smoothed maximum under the baseline occupancy measure
$\bar{\mathrm d}_{\rho,\theta}^\pi $ over the state space and the reference policy $\mu$ over the action space.
This not only 
yields a nonuniform perturbation budget over the  spaces, allowing larger perturbations at less frequently visited states and actions,
but   also allows for showing a robust bound with the simple log-transformed reward.

In contrast, existing robustness results for discrete-time MDPs  
\cite{eysenbach2022maximum,ashlag2025state} 
  construct   robust sets using a uniform distribution over the state space, which introduces an additional      entropy term of the state dynamics  in the pessimistic objective (see Appendix \ref{appendix:uniform}
  for an analogous result in the present setting). This entropy term is difficult to compute for complex or unknown state dynamics and thus limits the applicability of those results in general RL settings.
\end{remark}

 The following proposition shows that 
 the robust set $\widetilde{\mathcal C}^{\pi}_{\tau,\varepsilon}(\theta)$  expands as the entropy regularization parameter $\tau$ increases, confirming that 
stronger entropy regularization leads to more robust policies.  
The proof is provided in Appendix~\ref{appx: proof of prop:monotonicity_tau_global}. 

\begin{proposition}
\label{prop:monotonicity_tau_global}
Let   $\pi \in \Pi$ and $\varepsilon \geq 0$. 
The set-valued map $\tau\mapsto \widetilde{\mathcal C}_{\tau, \varepsilon}^\pi(\theta)$ is nondecreasing, i.e.,  
$ 
  \widetilde{\mathcal C}_{\tau_1, \varepsilon}^\pi(\theta)
    \subseteq
  \widetilde{\mathcal C}_{\tau_2, \varepsilon}^\pi(\theta)
$ for all $0<\tau_1\le\tau_2$. 
 Moreover,
\begin{equation*}
    \widetilde{\mathcal C}_{\infty, \varepsilon}^\pi(\theta) : =\lim_{\tau\to\infty}\widetilde{\mathcal C}_{\tau, \varepsilon}^\pi(\theta)
    =
    \bigcup_{\tau>0} \widetilde{\mathcal C}_{\tau, \varepsilon}^\pi(\theta)
    =
    \Big\{
        \tilde\theta\in\Theta
        \,\Big|\,
        \sum_{s\in\mathcal S}
        \bar{\mathrm d}_{\rho,\theta}^\pi(s)
        \sum_{a\in\mathcal A_s}
        \mu(a| s)\,f_{\theta,\tilde\theta}(s,a) 
        \le
        \varepsilon
    \Big\}\,,
\end{equation*}
and
\begin{equation*}
    \widetilde{\mathcal C}_{0,\varepsilon}^\pi(\theta): = \lim_{\tau\rightarrow0}   \widetilde{\mathcal C}_{\tau, \varepsilon}^\pi(\theta)
    =
    \bigcap_{\tau>0} \widetilde{\mathcal C}_{\tau, \varepsilon}^\pi(\theta)
    =
    \Big\{
        \tilde\theta\in\Theta
        \,\Big|\,
       f_{\theta,\tilde\theta}(s,a) 
        \le
        \varepsilon\,,
   \; 
    \textnormal{$\forall a\in \mathcal A_s$,
    $\bar{\mathrm d}_{\rho,\theta}^\pi$-a.s.~$s\in \mathcal S$}
     \Big\}\, \,,
\end{equation*}
where $f_{\theta,\tilde\theta}(s,a)$ is given in Theorem~\ref{thm:ct_global_sa_robust}.
\end{proposition}

\para{Robustness Guarantee with State Transition Intensity}
A limitation of  Theorem  
\ref{thm:ct_global_sa_robust} is that   the robust set $ \widetilde {\mathcal C}^\pi_{\varepsilon,\tau}(\theta)$ in 
\eqref{eq:robust_set_entropy} is defined directly in terms of ratios of occupancy measures, making it difficult to determine whether a given perturbed model belongs to this set.
To improve the interpretability, we provide a robustness result in which the set is characterized in terms of the discrepancy between the transition rates of the baseline and perturbed state dynamics.

To this end, 
for any perturbed model $\tilde{\theta}\in \Theta$, 
we define the following   the \emph{local relative entropy rate} 
of the two CTMCs: 
\begin{equation} \label{eq:action-wise local relative entropy} \ell(\lambda^a_\theta,\lambda^a_{\tilde\theta})(s)
    :=
    \sum_{s' \neq s}
    \left[
        \lambda^a_{ss',\theta}
        \log\frac{\lambda^a_{ss',\theta}}{\lambda^a_{ss',\tilde\theta}}
        -
        (\lambda^a_{ss',\theta}
        -
        \lambda^a_{ss',\tilde\theta})
    \right]\ge 0\,, \quad \forall s \in \mathcal S, a \in \mathcal A_s\,,
\end{equation}
which  measures the instantaneous relative entropy per unit time between the path measures of the baseline model $\mathbb P^\pi_{\theta}$
  and the perturbed model 
  $\mathbb P^\pi_{\tilde\theta}$ (see also~\cite{opper2007variational}). It is the continuous-time analogue of the KL divergence between discrete-time one-step transition kernels.
We further define the local state-reward perturbation cost rate of the model $\tilde \theta$ by
\begin{equation}
\label{eq:perturbation_cost}
c_{\theta,\tilde\theta}(s,a):= 
\ell(\lambda^a_\theta,\lambda^a_{\tilde\theta})(s)
                    -
                    r \log \frac{R_{\tilde\theta}(s,a)}{R_\theta(s,a)}\,,
                    \quad \forall s\in \mathcal S, a\in \mathcal A_s\,,
   \end{equation}
which
quantifies the magnitude of  perturbations that jointly affect the state dynamics and instantaneous rewards.
The cost $c_{\theta,\tilde\theta}$ increases as the perturbed intensity $\lambda^a_{ss',\tilde\theta}$ deviates from the baseline 
$\lambda^a_{ss',\theta}$, or as the perturbed model assigns a lower reward to a state--action pair than the baseline reward in the training environment.

We assume that the perturbed dynamics and the baseline model share the same jump structure: they admit the same possible transitions, but may assign different rates to those transitions.

\begin{assumption}
\label{assump:support_match}
For all  $\tilde\theta\in\Theta$,  $s,s' \in \mathcal S$ and $a \in \mathcal A_s$, $ \lambda^a_{ss',\theta} > 0 $ if and only if $\lambda^a_{ss',\tilde\theta} > 0$.
\end{assumption}

We now state a robustness result in which the robust set is explicitly characterized in terms of the perturbation cost  
$c_{\theta,\tilde\theta}$.

 \begin{theorem}
\label{thm:maxent_robust_ctmdp}
Suppose  that Assumption~\ref{assump:support_match} holds. For all  $\tau>0$, $\pi \in \Pi$, and   $\varepsilon \geq 0$, 
\small
\begin{equation}
\label{eq:main_theorem_expanded}
\begin{aligned}
  &  \inf_{\tilde\theta \in \mathcal C_{\tau,\varepsilon}^\pi(\theta)}
    \mathbb E_\rho^{\mathbb P^\pi_{\tilde\theta}}
    \bigg[
        \int_0^\infty e^{-rt}
        \sum_{a \in \mathcal A_{S_t}}
        \pi(a | S_t)\,
        R_{\tilde\theta}(S_t,a)\,
        \mathrm dt
    \bigg]
    \\
    &\ge
    \exp\bigg(
        \mathbb E_\rho^{\mathbb P^\pi_\theta}
        \bigg[
            \int_0^\infty e^{-rt}
            \bigg(
                \sum_{a \in \mathcal A_{S_t}}
                \pi(a | S_t)\, \big( r \log R_\theta(S_t,a) \big)
                -
                \tau\,\mathrm{KL}(\pi \,\|\, \mu)(S_t)
            \bigg)
            \mathrm dt 
        \bigg] - \frac{\varepsilon}{r} - \log r
    \bigg)\,,
\end{aligned}
\end{equation}
where the robust set 
$\mathcal C_{\tau, \varepsilon}^\pi(\theta)$
is given by
\small
\begin{equation}
\label{eq:robust_set}
\begin{aligned}
    \mathcal C_{\tau, \varepsilon}^\pi(\theta)
    :=
    \left\{
        \tilde\theta \in \Theta
        \,\Big|\,
        \sum_{s \in \mathcal S}
        \bar{\mathrm d}_{\rho,\theta}^\pi(s)
       \tau  \log \left[
            \sum_{a \in \mathcal A_s}
            \mu(a | s)
            \exp\left(
                \tfrac{
                  c_{\theta,\tilde\theta}(s,a)
                }{\tau}
            \right)
        \right]
        \le
        \varepsilon
    \right\}\,,
\end{aligned}
\end{equation}
with the perturbation cost $c_{\theta,\tilde\theta}(s,a)$ defined in \eqref{eq:perturbation_cost}.
\end{theorem}

Compared with Theorem  
\ref{thm:ct_global_sa_robust}, 
 Theorem 
\ref{thm:maxent_robust_ctmdp}
constructs the robust set more explicitly by directly aggregating the local perturbation cost 
$c_{\theta,\tilde\theta}$. 
The proof  starts by  lower bounding the performance of $\pi$ in a perturbed model by 
its performance in the baseline model  and the likelihood-ratio $\xi_t=\log \frac{\mathrm d \mathbb P^\pi_{ \theta}}{\mathrm d \mathbb P^\pi_{\tilde \theta}}\big|_{\mathcal F_t}$ at all time $t\ge 0$, where $\mathcal F_t$ is the observation   $\sigma$-algebra up to time $t$.
We then establish a semi-martingale decomposition for $(\xi_t)_{t\geq 0}$ showing  
$
\textrm d \xi_t =\ell(\lambda_{\theta}^\pi, \lambda_{\tilde \theta}^\pi) (S_t) \textrm d t +\textrm d  M_t,
$ 
where $M_t $ is a zero-mean martingale  under $\mathbb P^\pi_{\theta}$, and 
$\ell(\lambda_{\theta}^\pi, \lambda_{\tilde \theta}^\pi) $
is a $\pi$-aggregated relative entropy rate between $\mathbb P^\pi_{\theta}$ and $\mathbb P^\pi_{\tilde \theta}$.  
We further control  this aggregated rate function    with the (pointwise) local relative entropy rate   
$\ell(\lambda^a_\theta,\lambda^a_{\tilde\theta})$, and   construct the robust set using  the duality of the KL divergence
at each state.
The details are    given in  Appendix~\ref{appendix: proof of theorem}.

 We emphasize that 
 constructing  
 $ \mathcal C_{\tau,\varepsilon}^\pi(\theta)$ 
 using the local relative entropy rate 
$\ell(\lambda^a_\theta,\lambda^a_{\tilde\theta})$ 
  is crucial   
for continuous-time models. 
In contrast, the robust set derived for discrete-time MDPs in \cite[Theorem 4.2]{eysenbach2022maximum} degenerates to an empty set  as the time step size tends to zero. We illustrate this phenomenon in the following simple example.

\begin{example}
[\textbf{Degeneracy of discrete-time robust sets}]
\label{example:robust set}
   Consider for simplifying an uncontrolled CTMDP, an unperturbed reward, and a finite-horizon  objective over the interval  $[0,1]$. For any given   time grid $\{t_i\}_{i=0}^N$  with $t_i=i\Delta t$ and $\Delta t=1/N$,
   \cite{eysenbach2022maximum} constructs the   robust set of   the corresponding discrete-time MDP as:
  \begin{equation*}
  %\label{eq:discrete_time_robust}
  \mathcal C_{\Delta t}(\theta)=\left\{\tilde {\theta}\in \Theta \, \bigg\vert \,\mathbb E_{p_{\theta}}\left[\sum_{i=0}^N\log \sum_{s_{t+1}\in \mathcal S} \frac{p_{\theta}(s_{t+1}|s_t)}{  {p}_{\tilde \theta}(s_{t+1}|s_t)} \right]\le \varepsilon\right\}\,,
   \end{equation*}
   where $p_{\theta}$ and $p_{\tilde \theta}$ are the one-step transition probabilities for the baseline and perturbed (uncontrolled) MDPs, respectively. As $\Delta t\to 0$, using the definition of the transition intensities,  
   $$
   \log\sum_{s'\in \mathcal S} \frac{p_{\theta}(s'|s)}{  {p}_{\tilde \theta}(s'|s)}
   =\log \left( 1+\sum_{s'\not =s}\frac{\lambda_{ss',\theta}}{\lambda_{ss',\tilde\theta}} +O(\Delta t)\right)
   =\log \left( 1+\sum_{s'\not =s}\frac{\lambda_{ss',\theta}}{\lambda_{ss',\tilde\theta}} \right)+O(\Delta t).
   $$
This implies that 
   the left-hand-side of the constraint in 
   $ \mathcal C_{\Delta t}(\theta)$ is of the order $\mathcal O(N)$ even when $\tilde \theta=\theta$.
   Consequently, for a fixed $\varepsilon>0$, $\mathcal C_{\Delta t}(\theta)$ degenerates to an empty set as $\Delta t\to 0$.
\end{example}

  Stronger entropy regularization yields a larger robust set  $\tau\mapsto \mathcal C_{\tau, \varepsilon}^\pi(\theta)$, 
analogous to Proposition
\ref{prop:monotonicity_tau_global}.

\begin{proposition}
\label{prop:monotonicity_tau}
Suppose  that Assumption~\ref{assump:support_match} holds, and let   $\pi \in \Pi$ and $\varepsilon \geq 0$. 
Then the set-valued map $\tau\mapsto \mathcal C_{\tau, \varepsilon}^\pi(\theta)$ is nondecreasing, i.e.,  
$ 
    \mathcal C_{\tau_1, \varepsilon}^\pi(\theta)
    \subseteq
    \mathcal C_{\tau_2, \varepsilon}^\pi(\theta)
    $ 
    for all
    $   0<\tau_1\le\tau_2$.
Moreover,
\begin{equation*}
    \mathcal C_{\infty, \varepsilon}^\pi(\theta) : =\lim_{\tau\to\infty}\mathcal C_{\tau, \varepsilon}^\pi(\theta)
    =
    \bigcup_{\tau>0}\mathcal C_{\tau, \varepsilon}^\pi(\theta)
    =
    \Big\{
        \tilde\theta\in\Theta
        \,\Big|\,
        \sum_{s\in\mathcal S}
        \bar{\mathrm d}_{\rho,\theta}^\pi(s)
        \sum_{a\in\mathcal A_s}
        \mu(a| s)\,c_{\theta,\tilde\theta}(s,a) 
        \le
        \varepsilon
    \Big\}\,,
\end{equation*}
and
\begin{equation*}
    \mathcal C_{0,\varepsilon}^\pi(\theta): = \lim_{\tau\rightarrow0}\mathcal C_{\tau, \varepsilon}^\pi(\theta)
    =
    \bigcap_{\tau>0}\mathcal C_{\tau, \varepsilon}^\pi(\theta)
    =
    \Big\{
        \tilde\theta\in\Theta
        \,\Big|\,
        \sum_{s\in\mathcal S}
        \bar{\mathrm d}_{\rho,\theta}^\pi(s)\,
        \max_{a\in\mathcal A_s} c_{\theta,\tilde\theta}(s,a) 
        \le
        \varepsilon
    \Big\}\,,
\end{equation*}
where $c_{\theta,\tilde\theta}(s,a)$ is defined in 
\eqref{eq:perturbation_cost}.
 
\end{proposition}

 Note that in the limit approaching the unregularized case $\tau \rightarrow 0  $,
 $ \mathcal C_{0,\varepsilon}^\pi(\theta)$ in Proposition 
\ref{prop:monotonicity_tau} 
measures model perturbations in an averaged sense over the state space, whereas  
 $\widetilde{  \mathcal C}_{0,\varepsilon}^\pi(\theta)$ in 
Proposition    \ref{prop:monotonicity_tau_global} measures  the model perturbation in a stronger pointwise manner over the state space.  
The proof of Proposition~\ref{prop:monotonicity_tau} is given in Appendix~\ref{appendix: proof of the monotonicity}.

\subsection{Illustration of Robust Sets for Market Making}
\label{sec:worked_example}

Theorems  
\ref{thm:ct_global_sa_robust} and 
\ref{thm:maxent_robust_ctmdp} provide two complementary ways to quantify the robustness of entropy-regularized objectives under state perturbations: one via a global perturbation of the state-occupancy measures, which is more intrinsic to the overall optimization objective, and the other via a local perturbation of the state transition intensities, which is a more interpretable model-level perturbation. In this section, we illustrate these two robust sets through a   CTMDP example   arising from    dynamic market making problems in finance 
\cite{cartea2015algorithmic}.

In this   problem,   an agent (market maker) 
continuously posts bid and ask quotes over time to provide liquidity of a risky asset, while optimally balancing expected trading profits against their own inventory risk.
The state $q \in \mathcal S=\{-10,-9,\ldots,10\}$
denotes the agent’s inventory level, and the action
$\delta =(\delta^-,\delta^+) \in \mathcal{A} = \Delta_{20} \times \Delta_{20}$ 
denotes the bid and ask quotes (relative to the mid-price), chosen from a 20-point uniform discretization of
$[0.06, 0.24]$. 
The state dynamics follows a controlled CTMC 
 in which, 
 given a state--action pair $(q,\delta)$, 
transitions occur only to adjacent inventory levels  $q \pm 1$ with    upward rate $\lambda^\delta_{q,q+1,\theta} = \mathbf{1}_{\{q<10\}}\Lambda_\theta(\delta^-)$ and   downward rate $\lambda^\delta_{q,q-1,\theta} = \mathbf{1}_{\{q>-10\}}\Lambda_\theta(\delta^+)$,
where $ \theta=(\alpha_\Lambda, \beta_\Lambda)$
and  
$\Lambda_\theta(\delta) = \tfrac{\lambda_R}{1+\exp(\alpha_\Lambda+\beta_\Lambda\delta)}$ for a fixed  $ \lambda_R > 0$.
The instantaneous  reward  is  
$ R_\theta(q,\delta)
    =\lambda^\delta_{q,q+1,\theta} \delta^-
    +
   \lambda^\delta_{q,q-1,\theta}  \delta^+
    -
    \frac{\gamma}{2}\sigma^2 q^2$,
  which combines the expected profit from the earned spread with a quadratic penalty on the inventory level.
We take the reference policy $\mu$ to be the uniform distribution over $\mathcal A$, 
the discount factor $r=1$,
and the initial state distribution $\rho =\delta_{\{q=0\}}$.

Note that perturbations in the parameter $\theta$ affect both the state transition dynamics and the reward function.
To illustrate the robustness guarantees, we fix a behavior policy  $\pi$ designed to mimic the inventory-stabilizing structure of the optimal market-making policy, and evaluate the robust sets 
$\widetilde{\mathcal C}^{\pi}_{\tau,\varepsilon}(\theta)$   in~\eqref{eq:robust_set_entropy}
and    $\mathcal C^{\pi}_{\tau,\varepsilon}(\theta)$   in~\eqref{eq:robust_set} for 
a fixed budget parameter $\varepsilon =0.01$
and   baseline parameters $\theta=(-1,10)$,   which are calibrated from real data in~\cite{barzykin2023algorithmic}.
{Further details are in Appendix~\ref{app: detailed worked example}.
}

\begin{figure}[!ht]
    \centering
    \includegraphics[width=1.0\linewidth]{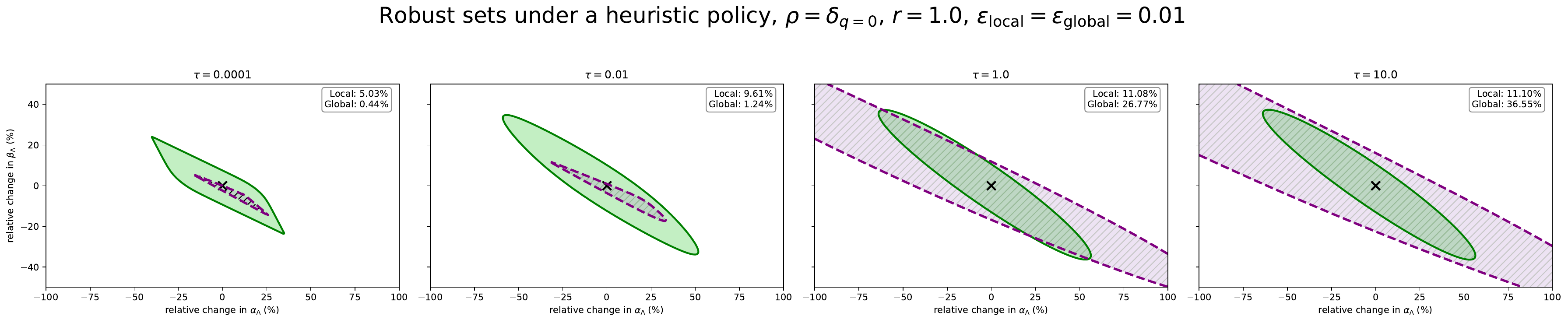}
    \caption{Robust regions for the market making example with $r=1$ and $\varepsilon=0.01$. Each panel shows simultaneous perturbations of $\alpha_\Lambda$ and $\beta_\Lambda$ for a fixed $\tau$. The green region is the local intensity-based robust set $\mathcal C^\pi_{\tau,\varepsilon}(\theta)$, and the purple region is the global occupancy-based robust set
$\widetilde{\mathcal C}^\pi_{\tau,\varepsilon}(\theta)$; the black cross marks the baseline parameter. }
    \label{fig:feasible-reference}
    \vspace{-4mm}
\end{figure}

Figure~\ref{fig:feasible-reference} plots certified robust regions under relative perturbations of $\theta=(\alpha_\Lambda,\beta_\Lambda)$, with axes
showing percentage changes from the baseline $\theta$. The green region is the local intensity-based certificate, the purple region is the global occupancy-based certificate.
Both regions expand with $\tau$, as predicted by
Propositions~\ref{prop:monotonicity_tau_global}
and~\ref{prop:monotonicity_tau}. 
Moreover, as shown in the propositions, 
when $\tau\downarrow0$, the log-sum-exp criteria become worst-case constraints.
The global set is controlled by the largest state--action occupancy/reward distortion, while the local set uses an occupancy-weighted sum of statewise worst action perturbations, making the global certificate more restrictive here.

As $\tau$ increases, the constraints become averaged perturbation budgets. The global set can then exceed the local set because it measures the realized effect after the policy and dynamics interact. Here, the policy is designed to
stabilize inventory near zero, so sizable changes in
$(\alpha_\Lambda,\beta_\Lambda)$ may still leave the discounted occupancy concentrated around central states. Such perturbations can have small global occupancy cost but large local relative-entropy cost.
Thus the two certificates are complementary: the local set gives an interpretable
model-level condition  and is relatively
stable across policy and initial-distribution choices, whereas the global set
measures the realized change in the controlled process and can be less
conservative, but is more sensitive to the policy and initial distribution. 

Appendix~\ref{app: detailed worked example} provides additional numerical comparisons of these robust sets. Appendix~\ref{app:certificate} presents an empirical analysis of certificate tightness, demonstrating that the certificate value acts as a continuous predictor of performance degradation and quantifying the conservatism of the certified region.

\vspace{-1mm}

% ======================================================================
\section{Experimental Results}\label{sec:experiments}
% ======================================================================

We validate the robustness theory on two event-driven control problems: a criss-cross queueing network (presented here) and single-asset market making (Appendix~\ref{app:robustness}). Both are implemented as continuous-time Markov jump processes matching the theoretical setting. 
Appendix~\ref{app:discretization_extended} illustrates the benefits of this event-driven implementation over a fixed–time-grid implementation. 

We evaluate on the criss-cross queueing network described in~\cite{dai2022queueing} with three job classes and
two servers.
We use the heavy-traffic instance: $\lambda = [1.5, 1.5, 1.5]$,
$\mu = [2.0, 3.0, 2.0]$, $h = [3.0, 1.0, 2.0]$,
$Q_{\max} = 50$, $T = 50$.  
The environment is simulated as a continuous-time Markov jump process. At each step, the next event time is sampled from the exponential clock with rate equal to the total arrival and effective service intensity, and the reward accumulated over the inter-event interval is computed exactly. 
Thus the implementation is \emph{event-driven}, unlike most contemporary works in the Deep RL literature which utilize a fixed time discretization (see Appendix~\ref{app:discretization_extended} for a comparison). 
We train three classes of neural network policies: (i) $\pi_{\rm std}$ trained via policy gradient \emph{without} entropy regularization, (ii) $\pi_{\tau}$ policies trained using various $\tau>0$, and (iii) $\pi_{\rm noisy}$ $\epsilon$-greedy policies derived from the standard policy. 
Full details on experimental setup are provided in Appendix~\ref{app:robustness}.

\para{Robustness to Dynamics Perturbations}
To test robustness, we evaluate each learned policy under perturbed arrival and service intensities $\lambda_i'=\alpha_\lambda \lambda_i$ and $\mu_i'=\alpha_\mu \mu_i$, where $(\alpha_\lambda,\alpha_\mu)$ ranges over a grid containing the baseline model $(1,1)$. 
We report the gain of entropy-regularized policy over the standard deterministic policy.
Figure~\ref{fig:queue_heatmaps} (left panel) shows that the selected policy with temperature $\tau=0.01$ improves over $\pi_{\rm std}$ on a large part of the perturbation grid. 
The gains are especially visible away from the baseline model, indicating that the entropy-regularized policy is less sensitive to misspecification of the continuous-time jump intensities.

\begin{figure}[!ht]
    \centering
    \includegraphics[width=0.35\linewidth]{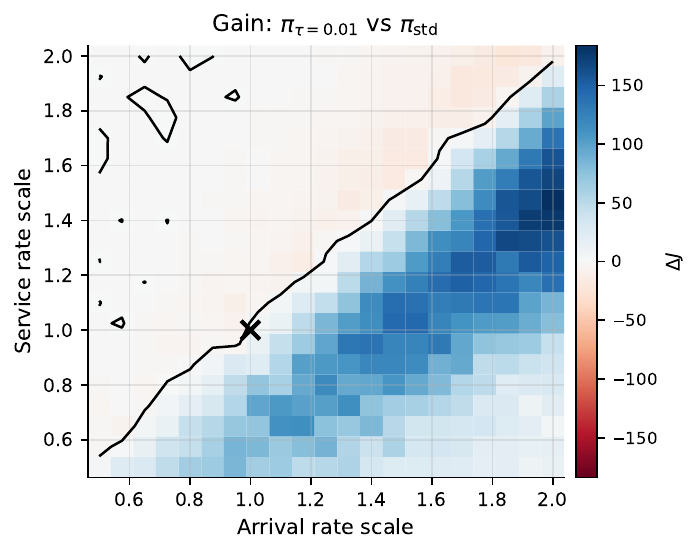}
    \includegraphics[width=0.35\linewidth]{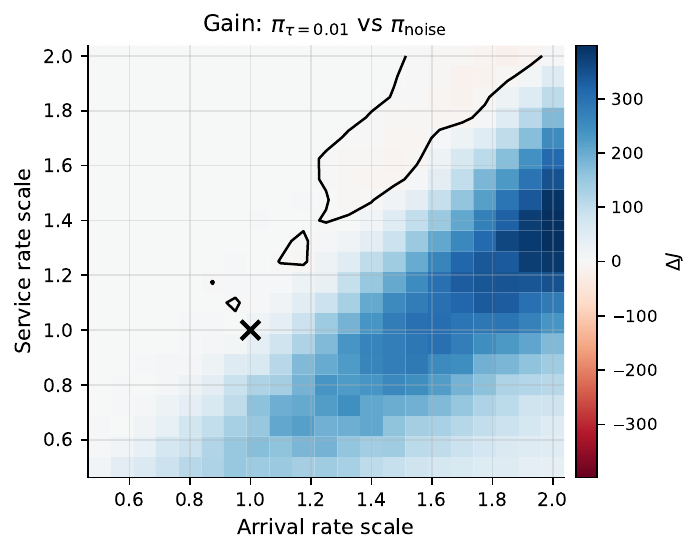}
    \caption{Robustness to arrival and service intensity perturbations in the criss-cross queueing network. Left: gain of the entropy-regularized policy over the standard deterministic policy. Right: gain over a noisy standard policy ($\epsilon$-greedy policy with $\epsilon = 0.1$). Black lines denote sign changes in $\Delta J$.}
    \label{fig:queue_heatmaps}
    \vspace{-1mm}
\end{figure}

A natural question is whether the improved robustness is simply due to added randomness. To test this, we compare against an uninformed noisy baseline $\pi_{\rm noise}  = 0.9\pi_{\rm std} + 0.1\,{\rm Unif}(\mathcal A)$  that injects   action noise. The right panel of Figure~\ref{fig:queue_heatmaps} shows that the entropy-regularized policy 
$\pi_\tau$ outperforms across a broad region of the perturbation grid. This indicates that the gains arise from strategic  value-aware randomization induced by entropy regularization, rather than from mere noise injection.

\begin{figure}[!ht]
    \centering
    \includegraphics[height=0.25\linewidth]{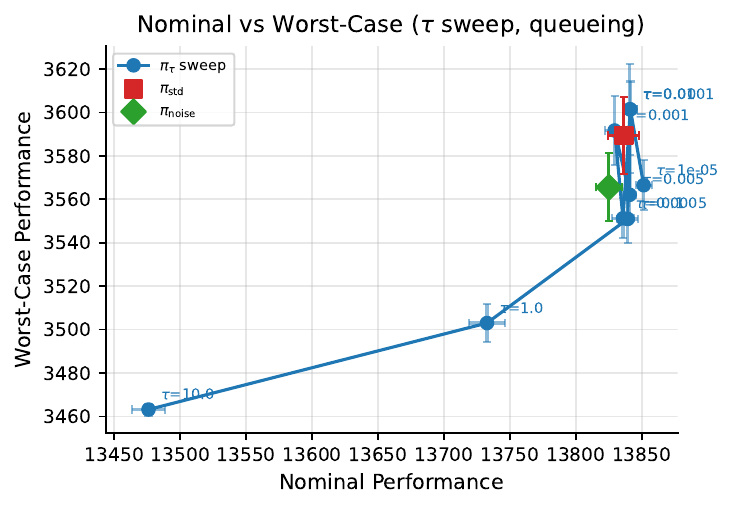}
    \includegraphics[height=0.25\linewidth]{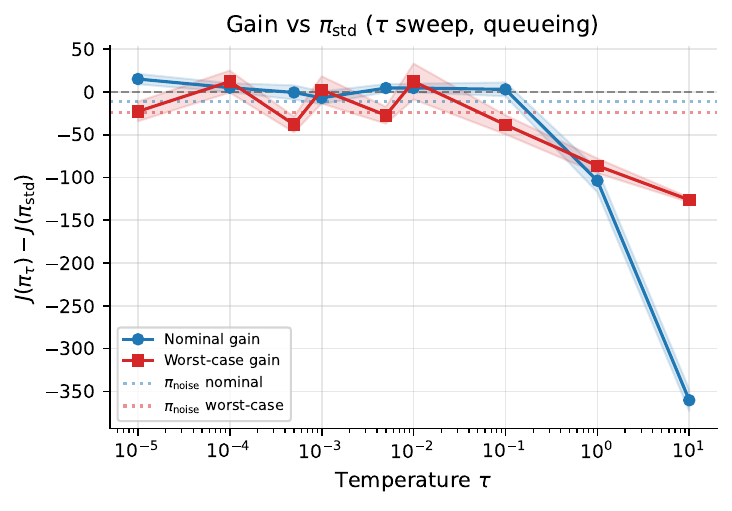}
    \caption{Effect of the regularization temperature in the criss-cross queueing network. Left: nominal and worst-case performance over the perturbation grid. Right: corresponding gains relative to $\pi_{\rm std}$.}
    \label{fig:queue_tau_sweep}
    \vspace{-2mm}
\end{figure}

Figure~\ref{fig:queue_tau_sweep} depicts a sweep over the regularization temperature. 
The left panel shows the nominal and worst-case performance of the learned
policies over the perturbation grid; small positive values of $\tau$ improve worst-case performance relative to $\pi_{\rm std}$, while large values of $\tau$ reduce both nominal and worst-case performance. 
This agrees with the theory: moderate entropy-regularization enlarges
the range of perturbations against which the policy is protected, whereas
excessive regularization makes the policy too diffuse to control congestion
effectively.
The right panel of Figure~\ref{fig:queue_tau_sweep} shows the same experiments in terms of gains relative to the standard policy.   
We observe that the  worst-case gain relative to $\pi_{\rm std}$ can exceed the corresponding nominal gain. This indicates that, under adverse intensity perturbations, $\pi_{\rm std}$ deteriorates more severely than the entropy-regularized policy.

{Analogous results for market making experiments are reported in Appendix~\ref{app:robustness}}, where formal statistical tests confirm significant worst-case improvements for both domains.

% ======================================================================
\section{Conclusion}\label{sec:conclusion}
% ======================================================================

We rigorously showed that entropy-regularized control provides robustness guarantees for continuous-time MDPs. 
For reward perturbations, the entropy-regularized objective admits an exact robust-control interpretation. 
For joint perturbations of rewards and transition intensities, it optimizes certified lower bounds on worst-case performance over policy-dependent uncertainty sets. 
We characterized these sets in two complementary ways: globally through discounted occupancy measures and locally through relative entropy rates of transition intensities. 
Unlike discrete-time formulations, the resulting continuous-time certificates avoid the additional unknown state-distribution entropy term and remain non-degenerate as the action frequency increases.

Our experiments support these findings. Both of the queueing-control and market making experiments show that moderate entropy regularization improves robustness to intensity misspecification over deterministic and noisy baselines. 
Entropy regularization is not necessarily the ideal robust RL method: it still requires choosing a temperature parameter, and the certified sets are policy-dependent. Nonetheless, our results suggest that entropy regularization is a simple and principled mechanism for obtaining robustness guarantees in continuous-time RL, without introducing an explicit adversary or a separate robust optimization procedure.

\para{Limitations and Future Work}
Several limitations remain in our work. 
First, our experiments validate the theory, but broader empirical studies on larger-scale benchmarks could study how the robustness gains scale with state and action complexity.
Second, our characterized robust sets provide worst-case certificates rather than exact descriptions of all perturbations under which entropy regularization improves performance, and may therefore be conservative. Developing tighter certificates is an important direction for future work. 
Finally, our analysis focuses on finite   spaces; extending the framework to continuous state and action spaces would broaden its applicability.

\newpage

\section*{Acknowledgments}
JC was supported by the EPSRC Centre for Doctoral Training in Mathematical Modelling, Analysis and Computation (MAC-MIGS) funded by the UK Engineering and Physical Sciences Research Council (grant EP/S023291/1).
YZ was funded in part by JPMorgan Chase \& Co.
Any views or opinions expressed herein are
solely those of the authors listed, and may differ from the views and opinions expressed by JPMorgan Chase
\& Co.~or its affiliates.
This material is not a product of the Research Department of J.P.~Morgan Securities
LLC.
This material does not constitute a solicitation or offer in any jurisdiction.

\section*{Disclaimer}

This paper was prepared for informational purposes in part by the Quantitative Trading \& Research Group of JPMorganChase \& Co. This paper is not a product of the Research Department of JPMorganChase \& Co. or its affiliates. Neither JPMorganChase \& Co. nor any of its affiliates makes any explicit or implied representation or warranty and none of them accept any liability in connection with this paper, including, without limitation, with respect to the completeness, accuracy, or reliability of the information contained herein and the potential legal, compliance, tax, or accounting effects thereof. This document is not intended as investment research or investment advice, or as a recommendation, offer, or solicitation for the purchase or sale of any security, financial instrument, financial product or service, or to be used in any way for evaluating the merits of participating in any transaction, and shall not constitute a solicitation under any jurisdiction or to any person, if such solicitation under such jurisdiction or to such person would be unlawful.

\bibliographystyle{plain}
\bibliography{Bibliography}

%%%%%%%%%%%%%%%%%%%%%%%%%%%%%%%%%%%%%%%%%%%%%%%%%%%%%%%%%%%%
\newpage
% ======================================================================
% APPENDIX
% ======================================================================
\appendix

\section{Extended Related Work}\label{app:extended_related}

\para{Continuous-time RL with entropy regularization}
Entropy regularization is widely used in continuous-time RL to encourage exploration, smooth  policy optimization, and mitigate  instabilities from time discretization. It underpins algorithm designs for controlled diffusions, including policy evaluation and TD 
 \cite{jia2022policy_eval}, policy‑gradient and actor–critic methods \cite{jia2022policy_grad,giegrich2024convergence}, Q‑learning \cite{jia2023q}, and TRPO/PPO \cite{zhao2023policy}. The framework has recently been extended beyond pure diffusions to jump‑diffusions\cite{gao2024reinforcement,bo2024continuous}, point processes\cite{meng2024reinforcement}, and regime‑switching systems\cite{huang2025switching,zhang2026regime_tsallis},
 and to risk-aware objective \cite{jia2026risk}. Parallel efforts integrate entropy regularization into mean‑field control and games \cite{guo2022entropy,frikha2025actor,wei2025continuous,ren2026common_noise_part1,ren2026common_noise_part2, dianetti2025entropy}.

Our work focuses on continuous-time pure jump processes and develops   efficient event-driven algorithms tailored to this setting. Indeed, when the state contains a diffusion component, the system evolves continuously, and existing algorithms typically require extremely fine time discretizations to sample actions accurately, resulting in substantial computational costs~\cite{szpruch2024optimal,gao2024reinforcement,jia2025accuracy}. In contrast, our continuous-time MDP framework evolves at random Poisson event times, which   provide an intrinsic event clock and eliminate the need for either a fixed sampling grid or an artificially imposed action frequency. This feature constitutes a key advantage of our approach. As demonstrated empirically in Appendix~\ref{app:discretization_extended}, the event-based formulation removes the sensitive time-step hyperparameter, avoids the fundamental bias--variance trade-off inherent in grid-based methods, and ensures that computational costs scale naturally with the underlying event rate. Similar observations have   been made for specific intensity-control models~\cite{meng2024reinforcement}.
 
 Moreover, despite the rapid development of RL algorithms, there is currently no theoretical work establishing that entropy regularization itself provides robustness guarantees in continuous-time RL. The closest related works are \cite{reisinger2021regularity,jia2026risk}. In particular, \cite{reisinger2021regularity} shows that optimal entropy-regularized policies depend continuously on model perturbations, while \cite{jia2026risk} introduces a risk-sensitive formulation to enforce robustness explicitly.
In contrast, our paper provides the first theoretical robustness guarantees arising directly from entropy regularization in continuous-time RL. We show that robustness emerges naturally by incorporating policy entropy into the objective function, as is commonly done in the RL literature, and we explicitly characterize the corresponding sets of admissible model perturbations.

\section{Proofs}\label{app:proofs}
This appendix contains the full proofs of all theoretical results.

\subsection{Equivalence of the arrival-driven model and CTMDP}
\label{appx: sampling methods difference}

We show that the state process generated by the arrival-driven sampling model has the same law as the aggregated CTMDP under the following structural assumption: for each state $s \in \mathcal S$, there exists a state-dependent decision rate $\lambda_s >0$ such that, for every $a\in\mathcal A_s$ and $s'\neq s$, $\lambda_{ss'}^a = \lambda_s p(s' \mid s,a)$,
where $p(\cdot \mid s,a)\in\mathcal P(\mathcal S)$. 
Accordingly, $\lambda^a_{ss}
    =
    -\sum_{s'\neq s}\lambda^a_{ss'}
    =
    -\lambda_s\bigl(1-p(s\mid s,a)\bigr)$.

Fix $\pi\in\mathcal P(\mathcal A\mid\mathcal S)$. In the arrival-driven sampling model, at a decision epoch $T_n$ with current state $S_{T_n}=s$, an action $A_n\sim\pi(\cdot\mid s)$ is sampled and then held fixed on $[T_n,T_{n+1})$. Since the event rate at state $s$ is $\lambda_s$ and does not depend on the chosen action, the holding time $\zeta_n := T_{n+1}-T_n$
satisfies $\zeta_n \mid S_{T_n}=s \sim \mathrm{Exp}(\lambda_s)$.
Moreover, conditional on $(S_{T_n},A_n)=(s,a)$, the post-event state satisfies
\begin{equation*}
    \mathbb P(S_{T_{n+1}}=s' \mid S_{T_n}=s, A_n=a)
    =
    p(s' \mid s,a),
    \quad s'\in\mathcal S\,.
\end{equation*}
Averaging over $A_n\sim\pi(\cdot\mid s)$ yields
\begin{equation*}
    \mathbb P(S_{T_{n+1}}=s' \mid S_{T_n}=s)
    =
    \sum_{a\in\mathcal A_s}\pi(a | s)\,p(s'\mid s,a)
    =: P^\pi(s,s')\,.
\end{equation*}

Therefore, conditional on the current state $s$, the next event time is exponential with rate $\lambda_s$, and the post-event state is drawn from the kernel $P^\pi(s,\cdot)$. Hence the state process $(S_t)_{t\ge0}$ generated by the arrival-driven sampling model is a CTMC with off-diagonal transition rates
\begin{equation*}
    \lambda_s P^\pi(s,s')
    =
    \lambda_s \sum_{a\in\mathcal A_s}\pi(a | s)\,p(s'\mid s,a)
    =
    \sum_{a\in\mathcal A_s}\pi(a | s)\lambda^a_{ss'}
    =
    \lambda^\pi_{ss'},
    \quad s'\neq s\,.
\end{equation*}
Its diagonal entries are therefore also given by $\lambda^\pi_{ss}=-\sum_{s'\neq s}\lambda^\pi_{ss'}$. Thus the state process has generator
$\lambda^\pi = (\lambda^\pi_{ss'})_{s,s'\in\mathcal S}$,
which is exactly the generator of the aggregated CTMDP under policy $\pi$. Since both processes also start from the same initial distribution $\rho$, they induce the same law on the state process. In particular, they yield the same objective function.

\subsection{Proof of Theorem~\ref{thm:exact_reward_only_robustness}} \label{appendix: proof of the reward -only}

\begin{proof}
For brevity, write $R:=R_\theta$, $\widetilde R:=R_{\tilde\theta}$ and
$\bar{\mathrm d}^\pi_\rho(s):=\bar{\mathrm d}_{\rho,\theta}^\pi(s)$.
We assume $\tilde\theta\in\Theta(\theta)$ only perturbs reward and leaves the transition dynamics unchanged, thus $\mathbb P^\pi_{\tilde\theta}
=\mathbb P^\pi_\theta$. 
Then, for every
$\tilde\theta\in\Theta(\theta)$,
\begin{equation*}
\begin{split}
    J(\pi;\tilde\theta)
    &:=
    \mathbb E_\rho^{\mathbb P^\pi_\theta}
    \bigg[
        \int_0^\infty e^{-rt}
        \sum_{a\in\mathcal A_{S_t}}
        \pi(a|S_t)\,\widetilde R(S_t,a)\,
        \mathrm dt
    \bigg]  =
    \frac{1}{r}
    \sum_{s\in\mathcal S}
    \bar{\mathrm d}^\pi_\rho(s)
    \sum_{a\in\mathcal A_s}
    \pi(a|s)\,\widetilde R(s,a).
\end{split}
\end{equation*}
Define the reward-only robust set by
\begin{equation} \label{eq:reward_only_set}
    \widehat{\mathcal R}_{\tau,\varepsilon}^\pi(\theta)
    :=
    \Big\{
        \widehat R:\mathcal S\times\mathcal A\to\mathbb R
        \,\Big|\,
        \tau\log\Big[
        \sum_{s\in\mathcal S}
        \bar{\mathrm d}^\pi_\rho(s)
        \sum_{a\in\mathcal A_s}
        \mu(a|s)
        \exp\Big(
            \frac{R(s,a)-\widehat R(s,a)}{\tau}
        \Big)
        \Big]
        \le \varepsilon
    \Big\},
\end{equation}
where $\varepsilon\in\mathbb R$, and consider the auxiliary optimization problem
\begin{equation}
\label{eq:def_reward_space_value}
    J_{\mathrm{rew}}
    :=
    \inf_{\widehat R\in \widehat{\mathcal R}_{\tau,\varepsilon}^\pi(\theta)}
    \frac{1}{r}
    \sum_{s\in\mathcal S}
    \bar{\mathrm d}^\pi_\rho(s)
    \sum_{a\in\mathcal A_s}
    \pi(a|s)\,\widehat R(s,a).
\end{equation}

The constraint is equivalent to
\begin{equation*}
    \sum_{s\in\mathcal S}
    \bar{\mathrm d}^\pi_\rho(s)
    \sum_{a\in\mathcal A_s}
    \mu(a|s)
    \exp\Big(
        \frac{R(s,a)-\widehat R(s,a)}{\tau}
    \Big)
    \le
    \exp\Big(\frac{\varepsilon}{\tau}\Big).
\end{equation*}
We compute $J_{\mathrm{rew}}$ by Lagrange duality. Let $\lambda\ge 0$ be
the Lagrange multiplier. The Lagrangian is
\small
\begin{equation*}
\begin{split}
    L(\widehat R,\lambda)
    =
    \frac{1}{r}
    \sum_{s\in\mathcal S}
    \bar{\mathrm d}^\pi_\rho(s)
    \sum_{a\in\mathcal A_s}
    \pi(a|s)\,\widehat R(s,a)
    +
    \lambda
    \bigg(
        \sum_{s\in\mathcal S}
        \bar{\mathrm d}^\pi_\rho(s)
        \sum_{a\in\mathcal A_s}
        \mu(a|s)
        \exp\Big(
            \frac{R(s,a)-\widehat R(s,a)}{\tau}
        \Big)
        -
        \exp\Big(\frac{\varepsilon}{\tau}\Big)
    \bigg).
\end{split}
\end{equation*}
\normalsize
For each fixed $\lambda>0$, the map $\widehat R\mapsto L(\widehat R,\lambda)$
is strictly convex. Moreover, Slater's condition holds: taking
$\widehat R=R+c$ with $c>0$ sufficiently large gives
\begin{equation*}
\begin{split}
    \sum_{s\in\mathcal S}
    \bar{\mathrm d}^\pi_\rho(s)
    \sum_{a\in\mathcal A_s}
    \mu(a|s)
    \exp\Big(
        \frac{R(s,a)-\widehat R(s,a)}{\tau}
    \Big)
    =
    e^{-c/\tau}
    <
    \exp\Big(\frac{\varepsilon}{\tau}\Big).
\end{split}
\end{equation*}
Hence strong duality applies. Moreover, any dual maximizer must satisfy
$\lambda^*>0$, since for $\lambda=0$ we have
$\inf_{\widehat R}L(\widehat R,0)=-\infty$.
Fix $\lambda>0$. Differentiating $L$ with respect to $\widehat R(s,a)$ gives
\begin{equation*}
    \frac{\partial L}{\partial \widehat R(s,a)}
    =
    \frac{1}{r}\,
    \bar{\mathrm d}^\pi_\rho(s)\,\pi(a|s)
    -
    \frac{\lambda}{\tau}\,
    \bar{\mathrm d}^\pi_\rho(s)\,\mu(a|s)
    \exp\Big(
        \frac{R(s,a)-\widehat R(s,a)}{\tau}
    \Big).
\end{equation*}
Solving this derivative to zero yields
\begin{equation*}
    \exp\Big(
        \frac{R(s,a)-\widehat R^*(s,a)}{\tau}
    \Big)
    =
    \frac{\tau}{r\lambda}
    \frac{\pi(a|s)}{\mu(a|s)}.
\end{equation*}
Therefore
\begin{equation*}
    \widehat R^*(s,a)
    =
    R(s,a)
    -
    \tau\log\frac{\pi(a|s)}{\mu(a|s)}
    -
    \tau\log\frac{\tau}{r\lambda}.
\end{equation*}

Since $\lambda^*>0$, complementary slackness implies that the constraint is
active at the optimum:
\begin{equation*}
    \sum_{s\in\mathcal S}
    \bar{\mathrm d}^\pi_\rho(s)
    \sum_{a\in\mathcal A_s}
    \mu(a|s)
    \exp\Big(
        \frac{R(s,a)-\widehat R^*(s,a)}{\tau}
    \Big)
    =
    \exp\Big(\frac{\varepsilon}{\tau}\Big).
\end{equation*}
Substituting the expression for $\widehat R^*$ gives
\begin{equation*}
\begin{split}
    \exp\Big(\frac{\varepsilon}{\tau}\Big)
    &=
    \frac{\tau}{r\lambda}
    \sum_{s\in\mathcal S}
    \bar{\mathrm d}^\pi_\rho(s)
    \sum_{a\in\mathcal A_s}
    \pi(a|s)
    =
    \frac{\tau}{r\lambda},
\end{split}
\end{equation*}
where we used $\sum_{a\in\mathcal A_s}\pi(a|s)=1$ and
$\sum_{s\in\mathcal S}\bar{\mathrm d}^\pi_\rho(s)=1$. Hence
\begin{equation*}
    \lambda^*
    =
    \frac{\tau}{r}
    \exp\Big(-\frac{\varepsilon}{\tau}\Big).
\end{equation*}
Consequently,
\begin{equation}
\label{eq:reward_space_optimizer_epsilon}
    \widehat R^*(s,a)
    =
    R(s,a)
    -
    \tau\log\frac{\pi(a|s)}{\mu(a|s)}
    -
    \varepsilon.
\end{equation}

Substituting~\eqref{eq:reward_space_optimizer_epsilon}
into~\eqref{eq:def_reward_space_value} gives
\begin{equation*}
\begin{split}
    J_{\mathrm{rew}}
    &=
    \frac{1}{r}
    \sum_{s\in\mathcal S}
    \bar{\mathrm d}^\pi_\rho(s)
    \bigg(
        \sum_{a\in\mathcal A_s}
        \pi(a|s)\,R(s,a)
        -
        \tau\,\mathrm{KL}(\pi\|\mu)(s)
        -
        \varepsilon
    \bigg)  =
    J_{\tau}(\pi;\theta)
    -
    \frac{\varepsilon}{r}.
\end{split}
\end{equation*}
Now define the parameterized robust set by
\begin{equation*}
    \widehat{\mathcal C}_{\tau,\varepsilon}^\pi(\theta)
    :=
    \Big\{
        \tilde\theta
        \,\Big|\,
        R_{\tilde\theta}\in
        \widehat{\mathcal R}_{\tau,\varepsilon}^\pi(\theta)
    \Big\}.
\end{equation*}
Then $ \Big\{
        R_{\tilde\theta}
        \,\Big|\,
        \tilde\theta\in\widehat{\mathcal C}_{\tau,\varepsilon}^\pi(\theta)
    \Big\}
    \subseteq
    \widehat{\mathcal R}_{\tau,\varepsilon}^\pi(\theta)$. 
Therefore,
\begin{equation*}
\begin{split}
    \inf_{\tilde\theta\in\widehat{\mathcal C}_{\tau,\varepsilon}^\pi(\theta)}
    J(\pi;\tilde\theta)
    &=
    \inf_{\tilde\theta\in\widehat{\mathcal C}_{\tau,\varepsilon}^\pi(\theta)}
    \frac{1}{r}
    \sum_{s\in\mathcal S}
    \bar{\mathrm d}^\pi_\rho(s)
    \sum_{a\in\mathcal A_s}
    \pi(a|s)\,R_{\tilde\theta}(s,a)  \\
    &\ge
    \inf_{\widehat R\in
    \widehat{\mathcal R}_{\tau,\varepsilon}^\pi(\theta)}
    \frac{1}{r}
    \sum_{s\in\mathcal S}
    \bar{\mathrm d}^\pi_\rho(s)
    \sum_{a\in\mathcal A_s}
    \pi(a|s)\,\widehat R(s,a)  =
    J_{\tau}(\pi;\theta)
    -
    \frac{\varepsilon}{r}.
\end{split}
\end{equation*}
This proves the reward-robust lower bound.
Moreover, suppose there exists $\tilde\theta^*\in\Theta(\theta)$
such that
\begin{equation*}
    R_{\tilde\theta^*}(s,a)
    =
    R_\theta(s,a)
    -
    \tau\log\frac{\pi(a|s)}{\mu(a|s)}
    -
    \varepsilon,
    \qquad s\in\mathcal S,\ a\in\mathcal A_s.
\end{equation*}
Then, by~\eqref{eq:reward_space_optimizer_epsilon},
\small
\begin{equation*}
\begin{split}
    \tau\log\Bigg[
    \sum_{s\in\mathcal S}
    \bar{\mathrm d}^\pi_\rho(s)
    \sum_{a\in\mathcal A_s}
    \mu(a|s)
    \exp\Big(
        \frac{R(s,a)-R_{\tilde\theta^*}(s,a)}{\tau}
    \Big)
    \Bigg]
    &=
    \tau\log\Bigg[
    \exp\Big(\frac{\varepsilon}{\tau}\Big)
    \sum_{s\in\mathcal S}
    \bar{\mathrm d}^\pi_\rho(s)
    \sum_{a\in\mathcal A_s}
    \pi(a|s)
    \Bigg] =
    \varepsilon.
\end{split}
\end{equation*}
\normalsize
Hence $\tilde\theta^*\in\widehat{\mathcal C}_{\tau,\varepsilon}^\pi(\theta)$ and
\begin{equation*}
    \inf_{\tilde\theta\in\widehat{\mathcal C}_{\tau,\varepsilon}^\pi(\theta)}
    J(\pi;\tilde\theta)
    =
    J_{\tau}(\pi;\theta)
    -
    \frac{\varepsilon}{r},
\end{equation*}
with $\tilde\theta^*$ a minimizer.
\end{proof}

\subsection{Proof of Theorem~\ref{thm:ct_global_sa_robust}} \label{appx: proof of proposition ct_global_sa}
\begin{proof}
Fix $\pi \in \Pi$. For $\vartheta\in\{\theta,\tilde\theta\}$, recall the discounted state occupancy measure and define the corresponding state--action occupancy measure by
\begin{equation*}
    \mathrm{d}^\pi_{\rho,\vartheta}(s)
    :=
    \mathbb E_{\rho}^{\mathbb P^\pi_\vartheta}
    \bigg[
        \int_0^\infty e^{-rt}\mathbf 1_{\{S_t=s\}}\,\mathrm dt
    \bigg] \, ,
    \qquad
    \mathrm{\nu}^\pi_{\rho,\vartheta}(s,a)
    :=
    \mathbb E_{\rho}^{\mathbb P^\pi_\vartheta}
    \bigg[
        \int_0^\infty e^{-rt}\mathbf 1_{\{S_t=s,A_t=a\}}\,\mathrm dt
    \bigg] \, ,
\end{equation*}
and the normalized occupancies by $\bar{\mathrm d}^\pi_{\rho, \vartheta}(s):=r\,\mathrm{d}^\pi_{\rho, \vartheta}(s)$ and 
$\bar{\mathrm \nu}^\pi_{\rho, \vartheta}(s,a):=r\,\mathrm{\nu}^\pi_{\rho, \vartheta}(s,a)$.
We have
$\bar{\mathrm \nu}^\pi_{\rho, \vartheta}(s,a)=\bar{\mathrm d}^\pi_{\rho, \vartheta}(s)\pi(a\mid s)$, and the objective can be rewritten as
\begin{equation*}
    J^\pi(\vartheta)
    :=
    \mathbb E_{\rho}^{\mathbb P^\pi_\vartheta}
    \bigg[
        \int_0^\infty e^{-rt}R_\vartheta(S_t,A_t)\,\mathrm dt
    \bigg]
    =
    \sum_{s,a}\mathrm{\nu}^\pi_{\rho, \vartheta}(s,a)R_\vartheta(s,a)
    =
    \frac{1}{r}\sum_{s,a}\bar{\mathrm \nu}^\pi_{\rho, \vartheta}(s,a)R_\vartheta(s,a) \, .
\end{equation*}
To evaluate the objective under a perturbed parameter $\tilde \theta \in \Theta$, we restrict the summation to the support of the baseline occupancy measure. 
Let $\mathcal{X}_\theta := \{(s,a) : \bar{\mathrm \nu}^\pi_{\rho, \theta}(s,a) > 0\}$.
Since the occupancies are non-negative and $R_{\tilde\theta}(s,a) > 0$, hence
\begin{equation*}
    J^\pi(\tilde\theta)
    \ge
    \frac{1}{r}
    \sum_{(s,a) \in \mathcal{X}_\theta}
    \bar{\mathrm \nu}^\pi_{\rho, \theta}(s,a)
    \frac{\bar{\mathrm \nu}^\pi_{\rho, \tilde\theta}(s,a)}{\bar{\mathrm \nu}^\pi_{\rho,\theta}(s,a)}
    R_{\tilde\theta}(s,a) \, .
\end{equation*}
Because $\bar{\mathrm \nu}^\pi_{\rho, \theta}(s,a) = 0$ for all $(s,a) \notin \mathcal{X}_\theta$, we have $\sum_{(s,a) \in \mathcal{X}_\theta} \bar{\mathrm \nu}^\pi_{\rho, \theta}(s,a) = 1$. Thus, applying Jensen's inequality to the concave logarithm yields
\begin{equation} \label{eq:jensen_log_reward_bound}
\begin{aligned}
    \log J^\pi(\tilde\theta)
    &\ge
    -\log r
    +
    \sum_{(s,a) \in \mathcal{X}_\theta}\bar{\mathrm \nu}^\pi_{\rho, \theta}(s,a)
    \log\Big(
        \tfrac{\bar{\mathrm \nu}^\pi_{\rho, \tilde\theta}(s,a)}{\bar{\mathrm \nu}^\pi_{\rho, \theta}(s,a)}
        R_{\tilde\theta}(s,a)
    \Big)  \, .
\end{aligned}
\end{equation}
We think of $R(s,a) := \log\big(R_\theta(s,a)\big)$ as the nominal reward.
For $(s,a) \in \mathcal{X}_\theta$, we define the adversarial reward as $\widehat R(s,a) 
:= \log\big( \frac{\bar{\mathrm \nu}^\pi_{\rho,\tilde\theta}(s,a)}{\bar{\mathrm \nu}^\pi_{\rho, \theta}(s,a)}
R_{\tilde\theta}(s,a) \big)$.
If $\bar{\mathrm \nu}^\pi_{\rho,\tilde\theta}(s,a) = 0$, we take $\widehat R(s,a) = -\infty$, in which case the subsequent lower bounds hold trivially.
For $(s,a) \in \mathcal{X}_\theta$, using $\bar{\mathrm \nu}^\pi_{\rho, \theta}(s,a) = \bar{\mathrm d}^\pi_{\rho, \theta}(s)\pi(a\mid s)$ we obtain
\begin{equation*}
R(s,a) - \widehat R(s,a) = \log\big(R_\theta(s,a)\big) - \log\Big(\tfrac{\bar{\mathrm d}^\pi_{\rho,\tilde\theta}(s)}{\bar{\mathrm d}^\pi_{\rho,\theta}(s)}R_{\tilde\theta}(s,a)\Big)  
=  \log\tfrac{\mathrm d^\pi_{\rho,\theta}(s)}{\mathrm d^\pi_{\rho,\tilde\theta}(s)}-\log\tfrac{R_{\tilde\theta}(s,a)}{R_\theta(s,a)} \, .
\end{equation*}
Therefore, by the definition of the robust set $\widetilde{\mathcal C}_{\tau, \varepsilon}^\pi$~\eqref{eq:robust_set_entropy}, $\tilde\theta\in\widetilde{\mathcal C}_{\tau,\varepsilon}^\pi(\theta)$ if and only if the perturbed reward$\widehat R(s,a)$ belongs to the reward-only robust set $\widehat{\mathcal R}_{\tau, \varepsilon}^\pi$~\eqref{eq:reward_only_set} centered at $R=\log R_\theta$.
Adopting the standard convention that $0 \cdot f(x) = 0$ for undefined or infinite $f(x)$, we can evaluate the summation over the entire state--action space $\mathcal{S} \times \mathcal{A}_s$. Applying Theorem~\ref{thm:exact_reward_only_robustness} to $R$ and $\widehat R$ gives:\begin{equation}\label{eq:transformed_reward_robust_bound}
\begin{split}
    &\inf_{\tilde\theta\in
    \widetilde{\mathcal C}_{\tau,\varepsilon}^\pi(\theta)}
    \sum_{s\in\mathcal S}
    \sum_{a\in\mathcal A_s}
    \bar{\mathrm \nu}_{\rho,\theta}^\pi(s,a)
    \widehat R(s,a) \ge
    \sum_{s\in\mathcal S}
    \sum_{a\in\mathcal A_s}
    \bar{\mathrm \nu}_{\rho,\theta}^\pi(s,a)
    \bigg(
        \log R_\theta(s,a)
        -
        \tau\log\frac{\pi(a\mid s)}{\mu(a\mid s)}
    \bigg)
    -
    \varepsilon .
\end{split}
\end{equation}
Combining~\eqref{eq:jensen_log_reward_bound} and
\eqref{eq:transformed_reward_robust_bound} yields
\begin{equation*}
\begin{split}
    \inf_{\tilde\theta\in
    \widetilde{\mathcal C}_{\tau,\varepsilon}^\pi(\theta)}
    \log J^\pi(\tilde\theta)
    &\ge
    -\log r
    +
    \sum_{s\in\mathcal S}
    \sum_{a\in\mathcal A_s}
    \bar{\mathrm \nu}_{\rho,\theta}^\pi(s,a)
    \bigg(
        \log R_\theta(s,a)
        -
        \tau\log\frac{\pi(a\mid s)}{\mu(a\mid s)}
    \bigg)
    -
    \varepsilon  \\
    &=
    -\log r
    +
    \sum_{s\in\mathcal S}
    \bar{\mathrm d}_{\rho,\theta}^\pi(s)
    \bigg(
        \sum_{a\in\mathcal A_s}
        \pi(a\mid s)\log R_\theta(s,a)
        -
        \tau\,\mathrm{KL}(\pi\|\mu)(s)
    \bigg)
    -
    \varepsilon .
\end{split}
\end{equation*}
Since the logarithm is increasing, and using $\bar{\mathrm d}_{\rho,\theta}^\pi(s) = r \mathrm d_{\rho,\theta}^\pi(s)$, we obtain 
\begin{equation*}
\begin{split}
    \inf_{\tilde\theta\in
    \widetilde{\mathcal C}_{\tau,\varepsilon}^\pi(\theta)}
    J^\pi(\tilde\theta)
    \ge
    \exp\bigg(
        r \sum_{s\in\mathcal S}
        \mathrm d_{\rho,\theta}^\pi(s)
        \bigg[
            \sum_{a\in\mathcal A_s}
            \pi(a\mid s)\log R_\theta(s,a)
            -
            \tau\,\mathrm{KL}(\pi\|\mu)(s)
        \bigg]
        - \log r - 
        \varepsilon
    \bigg),
\end{split}
\end{equation*}
which completes the proof.
\end{proof}

\subsection{Robustness Guarantees under Uniform Distributions}
\label{appendix:uniform}
\begin{proposition}
    
\label{corr: cover state-entropy case}
Fix $\pi\in\Pi$, $\tau>0$ and $\varepsilon\in\mathbb R$. Assume that $\mathcal A_s=\mathcal A$ for all $s$ and $\mu(\cdot\mid s)$ is uniform on $\mathcal A$, then
\small
\begin{equation*}
\begin{aligned}
    &\inf_{\tilde\theta\in\widetilde {\mathcal C}^{\pi, sa}_{\tau,\varepsilon}(\theta)}
    \mathbb E_{\rho}^{\mathbb P^\pi_{\tilde\theta}}
    \bigg[
        \int_0^\infty e^{-rt}R_{\tilde\theta}(S_t,A_t)\,\mathrm dt
    \bigg] \\
    &\ge
    \exp\bigg(
        r \, \mathbb E_\rho^{\mathbb P^\pi_\theta}
        \bigg[
            \int_0^\infty e^{-rt}
            \Big(
                \log R_\theta(S_t,A_t)
                - 
       \tau\,\mathrm{KL}(\bar{\mathrm d}^\pi_{\rho, \theta}\|\vartheta_{\mathcal S}) 
             -
               \tau\,\mathrm{KL}(\pi\|\mu)(S_t)
            \Big)\,\mathrm dt
        \bigg]
        -
        \varepsilon - \log r
    \bigg)\,,
\end{aligned}
\end{equation*}
\normalsize
where
$\vartheta_{\mathcal S}$
is the uniform distribution on $\mathcal S$, and 
 the robust set
 $ \widetilde {\mathcal C}^{\pi,sa}_{\tau,\varepsilon}(\theta)$ is defined as
\small
\begin{equation}
\label{eq:uniform_sa_set}
\begin{aligned}
    \widetilde{\mathcal C}^{\pi,\mathrm{sa}}_{\tau,\varepsilon}(\theta)
    :=
    \Bigg\{
        \tilde\theta\in\Theta
        \,\Bigg|\,
        \tau
        \log\bigg[
            \frac{1}{|\mathcal S| |\mathcal A|}
            \sum_{(s,a)\in\mathcal X}
            \exp\bigg(
                \frac{
                    \log
                    \tfrac{
                        \bar{\mathrm d}_{\rho,\theta}^\pi(s)
                    }{
                        \bar{\mathrm d}_{\rho,\tilde\theta}^\pi(s)
                    }
                    -
                    \log
                    \frac{
                        R_{\tilde\theta}(s,a)
                    }{
                        R_\theta(s,a)
                    }
                }{\tau}
            \bigg)
        \bigg]
        \le \varepsilon
    \Bigg\}.
\end{aligned}
\end{equation}
\normalsize
\end{proposition} 
 
\begin{proof}
Since $\mathcal A_s=\mathcal A$ for all $s$, write
$\mathcal X:=\mathcal S\times\mathcal A$. Recall that the normalized discounted state--action occupancy measure is given by $\bar \nu_{\rho,\theta}^\pi(s,a) = \bar{\mathrm d}_{\rho,\theta}^\pi(s)\pi(a | s)$. For later use, we denote this measure by $\bar \nu^\pi$ for simplicity.
We define
    $\bar \nu_\mu(s,a)
    :=
    \bar{\mathrm d}_{\rho,\theta}^\pi(s)\mu(a | s)$,
and let
\begin{equation*}
    u(s,a)
    :=
    \frac{1}{|\mathcal S||\mathcal A|}
    =
    \vartheta_{\mathcal S}(s)\mu(a | s),
    \qquad (s,a)\in\mathcal X,
\end{equation*}
where $\mu(\cdot\mid s)$ is uniform on $\mathcal A$ and $\vartheta_{\mathcal S}$
is uniform  on $\mathcal S$.
For each $\tilde\theta$, define
\begin{equation*}
    X_{\tilde\theta}(s,a)
    :=
    \log
    \frac{
        \bar{\mathrm d}_{\rho,\theta}^\pi(s)
    }{
        \bar{\mathrm d}_{\rho,\tilde\theta}^\pi(s)
    }
    -
    \log
    \frac{
        R_{\tilde\theta}(s,a)
    }{
        R_\theta(s,a)
    }.
\end{equation*}
The robust set in Theorem~\ref{thm:ct_global_sa_robust} uses the reference
measure $\bar \nu_\mu$. Changing the reference measure to any probability measure
$\bar \nu_r$ with compatible support amounts to
\begin{equation*}
\begin{split}
    \tau\log
    \bigg[
        \sum_{s,a}
        \bar \nu_r(s,a)
        \exp\bigg(
            \frac{X_{\tilde\theta}(s,a)}{\tau}
        \bigg)
    \bigg]
    &=
    \tau\log
    \bigg[
        \sum_{s,a}
        \bar \nu_\mu(s,a)
        \exp\bigg(
            \frac{
                X_{\tilde\theta}(s,a)
                +
                \tau\log\frac{\bar \nu_r(s,a)}{\bar \nu_\mu(s,a)}
            }{\tau}
        \bigg)
    \bigg].
\end{split}
\end{equation*}
Thus Theorem~\ref{thm:ct_global_sa_robust} applies with the reference
measure $\bar \nu_r$ by replacing the entropy penalty
$\mathrm{KL}(\bar \nu^\pi\| \bar \nu_\mu)$ by $\mathrm{KL}(\bar \nu^\pi\|\bar \nu_r)$. Taking
$\bar \nu_r=u$ gives the robust set
\begin{equation*}
    \widetilde{\mathcal C}^{\pi,\mathrm{sa}}_{\tau,\varepsilon}(\theta)
    =
    \Bigg\{
        \tilde\theta\in\Theta
        \,\Bigg|\,
        \tau
        \log\bigg[
            \frac{1}{|\mathcal S||\mathcal A|}
            \sum_{s\in\mathcal S}
            \sum_{a\in\mathcal A}
            \exp\bigg(
                \tfrac{
                    X_{\tilde\theta}(s,a)
                }{\tau}
            \bigg)
        \bigg]
        \le \varepsilon
    \Bigg\}.
\end{equation*}
Therefore, by Theorem~\ref{thm:ct_global_sa_robust} with the uniform
state--action reference measure $u$,
\begin{equation} \label{eq: lower bound with state entropy}
\begin{split}
    &\inf_{\tilde\theta\in
    \widetilde{\mathcal C}^{\pi,\mathrm{sa}}_{\tau,\varepsilon}(\theta)}
    \mathbb E_{\rho}^{\mathbb P^\pi_{\tilde\theta}}
    \bigg[
        \int_0^\infty e^{-rt} \sum_a \pi(a | S_t)
        R_{\tilde\theta}(S_t, a)\,\mathrm dt
    \bigg]
    \\
    &\quad\ge
    \exp\bigg(
        r\,
        \mathbb E_\rho^{\mathbb P^\pi_\theta}
        \bigg[
            \int_0^\infty e^{-rt} \sum_a \pi(a | S_t)
            \log R_\theta(S_t,a)\,\mathrm dt
        \bigg]
        -
        \tau\,\mathrm{KL}(\nu^\pi\|u)
        -
        \varepsilon
        -
        \log r
    \bigg).
\end{split}
\end{equation}
Now we decompose $\mathrm{KL}(\bar \nu^\pi\|u)$. Since
$\bar \nu^\pi(s,a)=\bar{\mathrm d}_{\rho,\theta}^\pi(s)\pi(a | s)$ and
$u(s,a)=\vartheta_{\mathcal S}(s)\mu(a | s)$, the chain rule for relative
entropy gives
\begin{equation} \label{eq: decompose KL}
\begin{split}
    \mathrm{KL}(\bar \nu^\pi\|u)
    &=
    \mathrm{KL}
    \big(
        \bar{\mathrm d}_{\rho,\theta}^\pi
        \,\big\|\,
        \vartheta_{\mathcal S}
    \big)
    +
    \sum_{s\in\mathcal S}
    \bar{\mathrm d}_{\rho,\theta}^\pi(s)
    \mathrm{KL}(\pi\|\mu)(s)  \\
    &=
    \mathrm{KL}
    \big(
        \bar{\mathrm d}_{\rho,\theta}^\pi
        \,\big\|\,
        \vartheta_{\mathcal S}
    \big)
    +
    r\,
    \mathbb E_\rho^{\mathbb P^\pi_\theta}
    \bigg[
        \int_0^\infty e^{-rt}
        \mathrm{KL}(\pi\|\mu)(S_t)\,\mathrm dt
    \bigg].
\end{split}
\end{equation}
Substituting~\eqref{eq: decompose KL} into~\eqref{eq: lower bound with state entropy} proves the claim.
\end{proof}

\subsection{Proof of Proposition~\ref{prop:monotonicity_tau_global}} \label{appx: proof of prop:monotonicity_tau_global}

\begin{proof}
Fix $\pi \in \Pi$ and $\tilde \theta \in \Theta$.
Recall that $\bar{\nu}_{\mu}(s,a):=
\bar{\mathrm d}_{\rho,\theta}^{\pi}(s)\mu(a | s)$ for any 
$s\in\mathcal S,\ a\in\mathcal A_s$.
For brevity we write
$f(s,a):=f_{\theta,\tilde\theta}(s,a)$ and $\Phi_\tau(\tilde\theta)
    :=
    \tau\log
    \sum_{s\in\mathcal S}
    \sum_{a\in\mathcal A_s}
    \bar{\nu}_{\mu}(s,a)
    \exp\left(\frac{f(s,a)}{\tau}\right)$.
Then, by definition~\eqref{eq:robust_set_entropy}, 
$\widetilde{\mathcal C}_{\tau,\varepsilon}^{\pi}(\theta)
    =
    \{\tilde\theta\in\Theta:\Phi_\tau(\tilde\theta)\le \varepsilon\}$.

We first prove monotonicity. 
Let $Z_\tau
    :=
    \sum_{s\in\mathcal S}
    \sum_{a\in\mathcal A_s}
    \bar{\nu}_{\mu}(s,a)e^{f(s,a)/\tau}$ and define the normalized exponential weights  $  q_\tau(s,a)
    :=
    \tfrac{
        \bar{\nu}_{\mu}(s,a)e^{f(s,a)/\tau}
    }{
        Z_\tau
    } $.
Differentiating $\Phi_\tau(\tilde\theta)=\tau\log Z_\tau$ gives
\begin{equation*}
\begin{split}
    \frac{\mathrm d}{\mathrm d\tau}\Phi_\tau(\tilde\theta)
    &=
    \log Z_\tau
    -
    \frac{1}{\tau}
    \sum_{s\in\mathcal S}
    \sum_{a\in\mathcal A_s}
    q_\tau(s,a)f(s,a).
\end{split}
\end{equation*}
Moreover,
\begin{equation*}
\begin{split}
    \mathrm{KL}(q_\tau\|\bar{\nu}_{\mu})
    &=
    \sum_{s\in\mathcal S}
    \sum_{a\in\mathcal A_s}
    q_\tau(s,a)
    \log
    \frac{q_\tau(s,a)}{\bar{\nu}_{\mu}(s,a)}
    =
    \frac{1}{\tau}
    \sum_{s\in\mathcal S}
    \sum_{a\in\mathcal A_s}
    q_\tau(s,a)f(s,a)
    -
    \log Z_\tau.
\end{split}
\end{equation*}
Hence $\frac{\mathrm d}{\mathrm d\tau}\Phi_\tau(\tilde\theta)
    =
    -\mathrm{KL}(q_\tau\|\bar{\nu}_{\mu})
    \le 0$, implying that
$\Phi_\tau(\tilde\theta)$ is nonincreasing in $\tau$. Therefore, for
$0<\tau_1\le\tau_2$, we have $\Phi_{\tau_2}(\tilde\theta)
    \le
    \Phi_{\tau_1}(\tilde\theta)$,
which proves
    $\widetilde{\mathcal C}_{\tau_1,\varepsilon}^{\pi}(\theta)
    \subseteq
    \widetilde{\mathcal C}_{\tau_2,\varepsilon}^{\pi}(\theta)$.

We next identify the limiting sets. Since $\mathcal S$ and $\mathcal A_s$ are
finite, the log-sum-exp limit gives
\begin{equation*}
    \lim_{\tau\rightarrow0}
    \Phi_\tau(\tilde\theta)
    =
    \max_{\substack{s\in\mathcal S,\ a\in\mathcal A_s\\
    \bar{\nu}_{\mu}(s,a)>0}}
    f_{\theta,\tilde\theta}(s,a).
\end{equation*}
Since $a\in\mathcal A_s=\operatorname{supp}\mu(\cdot\mid s)$, this is
equivalent to the condition
\begin{equation*}
    f_{\theta,\tilde\theta}(s,a)\le\varepsilon,
    \qquad
    \textnormal{for all $a\in\mathcal A_s$ and
    $\bar{\mathrm d}_{\rho,\theta}^{\pi}$-a.s. $s\in\mathcal S$.}
\end{equation*}
Therefore,
\begin{equation*}
    \bigcap_{\tau>0}
    \widetilde{\mathcal C}_{\tau,\varepsilon}^{\pi}(\theta)
    =
    \Big\{
        \tilde\theta\in\Theta
        \,\Big|\,
        f_{\theta,\tilde\theta}(s,a)\le\varepsilon,\;
        \textnormal{$\forall a\in\mathcal A_s$,
        $\bar{\mathrm d}_{\rho,\theta}^{\pi}$-a.s. $s\in\mathcal S$}
    \Big\}.
\end{equation*}

Similarly, as $\tau\to\infty$,
    $e^{f(s,a)/\tau}
    =
    1+\frac{f(s,a)}{\tau}+o\big(\tfrac1\tau\big)$,
uniformly over $(s,a)$. Hence
\begin{equation*}
\begin{split}
    \sum_{s,a}
    \bar{\nu}_{\mu}(s,a)e^{f(s,a)/\tau}
    &=
    1
    +
    \frac{1}{\tau}
    \sum_{s,a}
    \bar{\nu}_{\mu}(s,a)f(s,a)
    +
    o\big(\tfrac1\tau\big).
\end{split}
\end{equation*}
Using $\log(1+x)=x+o(x)$ gives
\begin{equation*}
    \lim_{\tau\to\infty}
    \Phi_\tau(\tilde\theta)
    =
    \sum_{s\in\mathcal S}
    \bar{\mathrm d}_{\rho,\theta}^{\pi}(s)
    \sum_{a\in\mathcal A_s}
    \mu(a | s)f_{\theta,\tilde\theta}(s,a).
\end{equation*}
Thus the limiting robust set as $\tau\to\infty$ is
\begin{equation*}
    \widetilde{\mathcal C}_{\infty,\varepsilon}^{\pi}(\theta)
    =
    \Big\{
        \tilde\theta\in\Theta
        \,\Big|\,
        \sum_{s\in\mathcal S}
        \bar{\mathrm d}_{\rho,\theta}^{\pi}(s)
        \sum_{a\in\mathcal A_s}
        \mu(a | s)f_{\theta,\tilde\theta}(s,a)
        \le\varepsilon
    \Big\}.
\end{equation*}

Since the family
$\{\widetilde{\mathcal C}_{\tau,\varepsilon}^{\pi}(\theta)\}_{\tau>0}$
is nondecreasing in $\tau$, its lower endpoint is the intersection over
$\tau>0$ and its upper endpoint is the limiting union as $\tau\to\infty$.
This completes the proof.

\end{proof}

\subsection{Proof of Theorem~\ref{thm:maxent_robust_ctmdp}} \label{appendix: proof of theorem}
Given a policy $\pi \in \Pi$ and state $s\in\mathcal{S}$, similarly to~\eqref{eq:action-wise local relative entropy}, we define the policy-wise local relative entropy by
\begin{equation} \label{eq: policy-wise local relative entropy}
\begin{split}
\ell(\lambda^\pi_{\theta}, \lambda^{\pi}_{\tilde \theta}) (s)
:=
\sum_{ s'\neq  s}
\Big[
\lambda^\pi_{ s s', \theta}   \log\frac{\lambda^\pi_{ s s', \theta}}{\lambda^\pi_{ s s', \tilde \theta}}
-\lambda^\pi_{ s s', \theta}  + \lambda^\pi_{ s s', \tilde \theta}  
\Big] \, ,
\end{split}
\end{equation}
The following lemma shows that the policy-wise local relative entropy is bounded above by the $\pi$-weighted average of the corresponding action-wise local relative entropy.
\begin{lemma} \label{lem:agg_divergence_bound}
Fix a state $s\in\mathcal S$ and a stationary Markov policy $\pi(\cdot| s)\in\mathcal P(\mathcal A)$.
Assume that for every $s'\neq s$ and every action $a\in\mathcal A$, 
$\lambda^a_{ss', \theta}>0$ implies $\lambda^a_{ss', \tilde \theta}>0$. 
Then
\begin{equation*}
    \ell(\lambda^\pi_{\theta}, \lambda^{\pi}_{\tilde \theta}) (s) \le
\sum_{a\in\mathcal A}\pi(a| s)\,\ell\big(\lambda^a_\theta, \lambda^a_{\tilde \theta} \big) (s)\,.
\end{equation*}
\end{lemma}

\begin{proof}
Define $f(x,y) :=x\log(x/y)-x+y$ with $x , y > 0$ having the same support. 
Notice that the Hessian matrix of $f$ is 
\begin{equation*}
    \boldsymbol H_f = \begin{bmatrix}
       \frac 1 x &  - \frac 1 y \\
        - \frac 1 y & \frac{x} {y^2}
    \end{bmatrix} \, .
\end{equation*}
The Hessian matrix is positive semidefinite on the domain, implying that $(x,y) \mapsto f(x,y)$ is jointly convex.  Hence given $ s'\neq  s$ and $a$, let $x = \lambda_{ s s', \theta}^a$ and $y=\lambda_{ ss', \tilde\theta}^a$. 
Apply Jensen's inequality under $\pi(\cdot| s)$:
\begin{equation*}
    f \Big(\sum_{a}\pi(a | s)\lambda_{ ss', \theta}^a,  \sum_{a}\pi(a| s)\lambda_{ s s', \tilde \theta}^a\Big)
\leq \sum_{a}\pi(a| s)\,
f\big(\lambda_{ ss', \theta}^a,\lambda_{ ss', \tilde \theta}^a\big) \, .
\end{equation*}
Summing over $ s'\neq  s$ yields the claim.
\end{proof}

Another lemma we need is the so-called Donsker–Varadhan variational formula, see~\cite{donsker1976asymptotic}.
We provide a proof for reader's convenience. 
\begin{lemma} \label{lem: Fenchel duality of KL divergence}
    Let $\mathcal{A}$ be a finite action set, $\mu \in \mathcal{P}(\mathcal{A})$ be a fixed reference probability measure and $\tau > 0$. 
    Define 
    \begin{equation*}
        \varphi (f) := \tau \log \Big( \sum_a \mu(a) e^{\tfrac{f(a)}{\tau}}\Big), \quad f: \mathcal A \to \mathbb R \,.
    \end{equation*}
    For any $\pi \in \mathcal{P}(\mathcal{A})$ with $\pi \ll \mu$,
    \begin{equation*}
        - \tau \mathrm{KL} (\pi \| \mu) = \inf_f \Big\{  - \sum_a f(a) \pi (a) + \varphi(f) \Big\} \, . 
    \end{equation*}
\end{lemma}
\begin{proof}
Fix $\tau>0$, $\mu\in\mathcal P(\mathcal A)$, and $\pi\in\mathcal P(\mathcal A)$ with $\pi\ll \mu$.
For any function $f:\mathcal A\to\mathbb R$, define the normalization constant $Z_f:=\sum_{a\in\mathcal A}\mu(a)\exp \big(\tfrac{f(a)}{\tau}\big)$,
and $\mu_f\in\mathcal P(\mathcal A)$ by
\begin{equation} \label{eq:mu_f}
    \mu_f(a):=\tfrac{\mu(a)\exp\big(\tfrac{f(a)}{\tau}\big)}{Z_f}, \quad a \in \mathcal A \,.
\end{equation}
Notice that whenever $\mu(a)=0$, $\mu_f(a)=0$ for all $f$, hence $\mu_f$ and $\mu$ have the same support.

We first rewrite the objective inside the infimum. 
Taking logarithm of~\eqref{eq:mu_f} and rearranging gives
\begin{equation*}
    f(a)=\tau\log\frac{\mu_f(a)}{\mu(a)}+\tau\log Z_f,
\quad a\in\mathrm{supp}(\mu) \, .
\end{equation*}
Multiplying by $\pi$ and summing over $a \in \mathrm{supp}(\mu)$ yields
\begin{align*}
-\sum_a f(a)\pi(a)
&= -\tau \sum_a \pi(a)\log\frac{\mu_f(a)}{\mu(a)} - \tau\log Z_f \, .
\end{align*}
Notice that 
\begin{align*}
-\tau \sum_a \pi(a)\log\frac{\mu_f(a)}{\mu(a)}
&= -\tau \sum_a \pi(a)\log\frac{\pi(a)}{\mu(a)}
   -\tau \sum_a \pi(a)\log\frac{\mu_f(a)}{\pi(a)} \\
&= -\tau\,\mathrm{KL}(\pi\|\mu) + \tau\,\mathrm{KL}(\pi\|\mu_f),
\end{align*}
where we used $\mathrm{KL}(\pi\|\mu_f)=\sum_a \pi(a)\log\frac{\pi(a)}{\mu_f(a)}$ and this is finite since $\mu_f$ and $\mu$ have the same support and  $\pi\ll \mu$.
Therefore, for every $f$,
\begin{equation*}
-\sum_a f(a)\pi(a)+\tau\log\big(\sum_a \mu(a)e^{\tfrac{f(a)}{\tau}}\big)
=
-\tau\,\mathrm{KL}(\pi\|\mu)+\tau\,\mathrm{KL}(\pi\|\mu_f)
\geq 
-\tau\,\mathrm{KL}(\pi\|\mu),
\end{equation*}
since $\mathrm{KL}(\pi\|\mu_f)\ge 0$.
Taking the infimum over $f$ yields
\begin{equation*}
    \inf_f \Big\{  - \sum_a f(a) \pi (a) + \tau \log \big( \sum_a \mu(a) e^{\tfrac{f(a)}{\tau}}\big) \Big\} \geq  - \tau \mathrm{KL} (\pi \| \mu) \, .
\end{equation*}

It remains to show that there exists $f^\ast$ s.t. $\mathrm{KL}(\pi\|\mu_{f^\ast})=0$, hence we attain the equality. 
Choose any constant $c\in\mathbb R$ and define, for $a\in\mathrm{supp}(\mu)$, $f^*(a) :=\tau\log\frac{\pi(a)}{\mu(a)}+c$, and define $f^*(a)$ arbitrarily for $a\notin\mathrm{supp}(\mu)$.
Then
\begin{equation*}
    \mu(a)e^{f^*(a)/\tau}=\mu(a)\frac{\pi(a)}{\mu(a)}e^{c/\tau}=\pi(a)e^{c/\tau}, \quad
Z_{f^\ast}=\sum_a \pi(a)e^{c/\tau}=e^{c/\tau},
\end{equation*}
implies that $\mu_{f^\ast}(a)=\pi(a)$ for all $a$ and hence $\mathrm{KL}(\pi\|\mu_{f^\ast})=0$, which proves the lemma.
\end{proof}

Now we are ready to prove Theorem~\ref{thm:maxent_robust_ctmdp}.
\begin{proof}
To simplify notations, we introduce the following shorthand.  
For any $s \in \mathcal{S}$, denote $\widetilde R^\pi(s) := \sum_{a\in\mathcal A} \pi(a | s) R_{\tilde \theta}(s,a)$ and write $\widetilde R = R_{\tilde \theta}$ and $R = R_{\theta}$.
Also we write $\mathbb P^\pi := \mathbb P^\pi_\theta$ for the law of the process under the baseline parameter $\theta$. 
For the transition intensities, denote $\lambda_{ss'} := \lambda_{ss',\theta}$ and $\widetilde{\lambda}_{ss'} := \lambda_{ss',\tilde\theta}$.
Moreover, we define the likelihood-ratio process $Z_t := \tfrac{\mathrm d \mathbb{P}^\pi_{\tilde \theta}}{\mathrm d \mathbb P^\pi_\theta}\big|_{\mathcal F_t}$ and
$Z_t' := \tfrac{1}{Z_t}$.

We start by taking a log-transform of the robust objective
\begin{equation*}
    \log  \inf_{\tilde\theta\in C_\tau^\pi} \bigg( \mathbb E_\rho^{\mathbb P^\pi_{\tilde \theta}} \bigg[\int_0^\infty e^{-rt} \widetilde R^\pi(S_t) \, \mathrm dt \bigg]\bigg) \, ,
\end{equation*}
since $R_{\tilde\theta}(s,a)>0$ for all $(s,a)$, the objective is strictly positive, hence the log-transform is well-defined. 
Moreover, since $x\mapsto \log x$ is strictly increasing on the domain, it can be moved inside the infimum.

Applying the change of measure and Jensen’s inequality to the concave function $\log$ and using that $r e^{-rt}dt$ is a probability measure on $[0,\infty)$, we have
\begin{equation} \label{eq: robust inter1}
\begin{split}
    \log  \inf_{\tilde\theta\in C_\tau^\pi} & \bigg( \mathbb E_\rho^{\mathbb P^\pi_{\tilde \theta}}  \bigg[\int_0^\infty e^{-rt} \widetilde R^\pi(S_t) \, \mathrm dt \bigg]\bigg)  \geq  \inf_{\tilde\theta\in C_\tau^\pi}   \mathbb E_\rho^{\mathbb P^\pi} \bigg[ \log  \bigg( \int_0^\infty r e^{-rt} Z_t \tfrac{\widetilde R^\pi(S_t)}{r} \, \mathrm dt \bigg) \bigg] \\
    & \geq \inf_{\tilde\theta\in C_\tau^\pi}   \mathbb E_\rho^{\mathbb P^\pi} \bigg[ \int_0^\infty r e^{-rt} \log \Big(  Z_t \tfrac{\widetilde R^\pi(S_t)}{r}  \Big)\, \mathrm dt  \bigg] \\
    & = \inf_{\tilde\theta\in C_\tau^\pi} \bigg\{  - \mathbb E_\rho^{\mathbb P^\pi} \bigg[ \int_0^\infty r e^{-rt} \log  Z_t' \, \mathrm dt  \bigg] +  \mathbb E_\rho^{\mathbb P^\pi} \bigg[ \int_0^\infty r e^{-rt} \log \Big( \tfrac{\widetilde R^\pi(S_t)}{r}  \Big)\, \mathrm dt  \bigg] \bigg\}\, .
\end{split}
\end{equation}
Applying Girsanov theorem~\cite[Thm. 1.35]{oksendal2007applied} for CTMCs, we have 
\begin{equation*}
    \log Z_t' = 
\int_0^t \sum_{s'\neq S_{u-}} \log  \frac{\lambda^\pi_{S_{u-}s'}} {\widetilde\lambda^\pi_{S_{u-}s'}} \, \mathrm dN_u^{S_{u-},s'} - \int_0^t
\sum_{s'\neq S_{u-}} \big( \lambda^\pi_{S_{u-}s'}-\widetilde\lambda^\pi_{S_{u-}s'} \big)\, \mathrm du \, ,
\end{equation*}
where 
\begin{equation*}
   N_t^{s,s'}
:=\sum_{0<u\le t}\mathbf 1_{\{S_{u-}=s,\;S_u=s'\}} \,.
\end{equation*}
Taking expectation under $\mathbb P^\pi$ and using the compensator of the
counting processes yields
\begin{equation*}
\begin{split}
\mathbb E^{\mathbb P^\pi}\big[\log Z_t'\big]
&= \mathbb E^{\mathbb P^\pi}\bigg[
\int_0^t \sum_{s'\neq S_{u-}} \Big( \lambda^\pi_{S_{u-}s'} \log  \frac{\lambda^\pi_{S_{u-}s'}} {\widetilde\lambda^\pi_{S_{u-}s'}} - 
 \big( \lambda^\pi_{S_{u-}s'}-\widetilde\lambda^\pi_{S_{u-}s'} \big) \Big)\, \mathrm du
\bigg] \\
&= \mathbb E^{\mathbb P^\pi}\bigg[\int_0^t \ell(\lambda^\pi, \widetilde\lambda^\pi) (S_u)\,\mathrm du\bigg].
\end{split}
\end{equation*}
since the set of jump times has Lebesgue measure $0$, we have $\int_0^t f(S_{u-}) \mathrm du = \int_0^t f(S_{u}) \mathrm du$ a.s. for any Borel measurable function $f$. 

Therefore, by the expression for $\mathbb E^{\mathbb P^\pi}\big[\log Z_t'\big]$, the first term in the last line of equation~\eqref{eq: robust inter1} is 
\begin{equation*}
    \begin{split}
         \inf_{\tilde{\theta} \in C_\tau^\pi}  
         - \mathbb E_\rho^{\mathbb{P}^\pi} \bigg[\int_0^\infty r e^{-rt} \, \log Z_t' \, \mathrm dt  \bigg]
          = \inf_{\tilde{\theta} \in C_\tau^\pi} -
         \mathbb E_\rho^{\mathbb{P}^\pi} \bigg[  \int_0^\infty r e^{-rt}  
         \Big( \int_0^t \ell(\lambda^\pi, \widetilde\lambda^\pi) (S_u)\,\mathrm du\Big) \,  \mathrm dt  \bigg] \,. 
    \end{split}
\end{equation*}
For any $s \in \mathcal{S}$, define $f(s) := \ell(\lambda^\pi, \widetilde\lambda^\pi) (s)$.
Applying Fubini's theorem path-wise, we have 
\begin{equation*}
    \int_0^\infty r e^{-rt}  \Big( \int_0^t f(S_u) \, \mathrm du \Big) \,  \mathrm dt  = \int_0^\infty f(S_u) \Big( \int_u^\infty r e^{-rt} \mathrm d t \Big) \mathrm du =  \int_0^\infty e^{-ru}f(S_u) \,  \mathrm du \, .
\end{equation*}
Therefore, 
\begin{equation*}
    \inf_{\tilde{\theta} \in C_\tau^\pi}  
    - \mathbb E_\rho^{\mathbb{P}^\pi} \bigg[\int_0^\infty r e^{-rt} \, \log Z_t' \, \mathrm dt  \bigg]
    =
    \inf_{\tilde{\theta} \in C_\tau^\pi}  
    - \mathbb E_\rho^{\mathbb{P}^\pi} \bigg[\int_0^\infty e^{-rt} \, \ell(\lambda^\pi, \widetilde\lambda^\pi)(S_t) \, \mathrm dt  \bigg] \, .
\end{equation*}
Hence the last line of equation~\eqref{eq: robust inter1} becomes 
\begin{equation} \label{eq: robust inter2}
    \inf_{\tilde\theta\in C_\tau^\pi} \bigg\{ - \mathbb E_\rho^{\mathbb{P}^\pi} \bigg[\int_0^\infty e^{-rt} \, \ell(\lambda^\pi, \widetilde\lambda^\pi)(S_t) \, \mathrm dt  \bigg] +  \mathbb E_\rho^{\mathbb P^\pi} \bigg[ \int_0^\infty r e^{-rt} \log \Big( \tfrac{\widetilde R^\pi(S_t)}{r}  \Big)\, \mathrm dt  \bigg] \bigg\} \, .
\end{equation}
Applying Lemma~\ref{lem:agg_divergence_bound}, for any $s \in \mathcal{S}$,
\begin{equation*}
   - \ell(\lambda^\pi,\widetilde\lambda^\pi)(s) \ge
- \sum_{a\in\mathcal A}\pi(a| s)\,\ell(\lambda^{a},\widetilde\lambda^{a})(s)\,. 
\end{equation*}
Moreover, by Jensen's inequality, for any $s\in\mathcal S$, 
\begin{equation*}
    \log \Big( \tfrac{\widetilde R^\pi(s)}{r} \Big) = \log \Big( \sum_{a} \pi(a |  s)\tfrac{\widetilde R^\pi(  s, a)}{r}  \Big) \geq \sum_{a} \pi(a |  s) \log \Big( \tfrac{\widetilde R(s, a) }{r} \Big) \, . 
\end{equation*}
Substituting these into~\eqref{eq: robust inter2} yields 
\begin{equation} \label{eq: robust inter3}
\begin{split}
    \log   \inf_{\tilde\theta\in C_\tau^\pi}  \bigg( \mathbb E_\rho^{\mathbb P^\pi_{\tilde \theta}}  \bigg[\int_0^\infty & e^{-rt} \widetilde R^\pi(S_t) \, \mathrm dt \bigg]\bigg)  \\ \geq
    \inf_{\tilde\theta\in C_\tau^\pi} &\bigg\{ 
    - \mathbb E_\rho^{\mathbb{P}^\pi} \bigg[\int_0^\infty e^{-rt} 
    \sum_{a}\pi(a| S_t)\,\ell(\lambda^{a},\widetilde\lambda^{a})(S_t) \, \mathrm dt  \bigg] \\
    &\quad +  \mathbb E_\rho^{\mathbb P^\pi} \bigg[ \int_0^\infty r e^{-rt} 
    \sum_{a}\pi(a |  S_t)\,
    \log \Big( \tfrac{\widetilde R(S_t, a) }{r} \Big)\, \mathrm dt  \bigg] 
    \bigg\} \, .
\end{split}
\end{equation}
For the second term in the right hand side of equation~\eqref{eq: robust inter3}, by adding and subtracting $\log \Big( \tfrac{R(S_t, a)}{r} \Big)$, we have
\begin{equation*}
\begin{split}
    \inf_{\tilde\theta\in C_\tau^\pi} \mathbb E_\rho^{\mathbb P^\pi} \bigg[ & \int_0^\infty r e^{-rt} 
    \sum_{a}\pi(a |  S_t)\,
    \log \Big( \tfrac{\widetilde R(S_t, a) }{r} \Big)\, \mathrm dt  \bigg]  \\
    & 
    = \inf_{\tilde\theta\in C_\tau^\pi}  \mathbb E_\rho^{\mathbb P^\pi} \bigg[ \int_0^\infty r e^{-rt} \sum_{a}\pi(a |  S_t) \log \Big( \tfrac{\widetilde R(S_t,a)}{R(S_t,a)}  \Big)\, \mathrm dt  \bigg] \\
    & \qquad + \mathbb E_\rho^{\mathbb P^\pi} \bigg[ \int_0^\infty r e^{-rt} \sum_{a}\pi(a |  S_t) \log \Big( \tfrac{R(S_t, a)}{r} \Big)\, \mathrm dt  \bigg] \, .
\end{split}
\end{equation*}
Therefore, 
\begin{equation*}
    \begin{split}
         \log   \inf_{\tilde\theta\in C_\tau^\pi}  \bigg( \mathbb E_\rho^{\mathbb P^\pi_{\tilde \theta}}  \bigg[\int_0^\infty & e^{-rt} \widetilde R^\pi(S_t) \, \mathrm dt \bigg]\bigg)  \\ 
         \ge \inf_{\tilde\theta\in C_\tau^\pi} & 
     \mathbb E_\rho^{\mathbb{P}^\pi} \bigg[\int_0^\infty e^{-rt} 
    \sum_{a}\pi(a| S_t) \Big( - \ell(\lambda^{a},\widetilde\lambda^{a})(S_t) + r \log \tfrac{\widetilde R(S_t,a)}{R(S_t,a)} \Big)\, \mathrm dt  \bigg]  \\
         & \qquad  + \mathbb E_\rho^{\mathbb{P}^\pi} \bigg[\int_0^\infty  e^{-rt} \sum_{a} \pi(a | S_t) \Big( r \log \Big(  \tfrac{R(S_t, a)}{r} \Big)\Big) \mathrm dt  \bigg]\,.
    \end{split}
\end{equation*}
By Lemma~\ref{lem: Fenchel duality of KL divergence}, the following inequality holds, for any $f: \mathcal{A} \to \mathbb R$,  
\begin{equation} \label{eq: coro of duality}
    \sum_a \pi(a) f(a) \leq \tau \mathrm{KL} (\pi \| \mu) + \varphi(f) \, ,
\end{equation}
where $\varphi (f) = \tau \log \Big( \sum_a \mu(a) e^{\tfrac{f(a)}{\tau}}\Big)$.
Applying~\eqref{eq: coro of duality} to $c_{\theta,\tilde\theta}(s,a)$~\eqref{eq:perturbation_cost}
state-wise yields
\begin{equation*}
    - \sum_{a} \pi(a |  s) c_{\theta,\tilde\theta}(s,a) \geq -\tau \, \mathrm{KL} (\pi(\cdot | s) \| \mu (\cdot | s)) - \varphi(c_{\theta,\tilde\theta}(s,\cdot)) \, .
\end{equation*}
Therefore,  
\begin{equation*}
    \begin{split}
         \log   \inf_{\tilde\theta\in C_\tau^\pi}  & \bigg( \mathbb E_\rho^{\mathbb P^\pi_{\tilde \theta}}  \bigg[\int_0^\infty  e^{-rt} \widetilde R^\pi(S_t) \, \mathrm dt \bigg]\bigg)  \\ 
         & \ge \inf_{\tilde\theta\in C_\tau^\pi} 
     \mathbb E_\rho^{\mathbb{P}^\pi} \bigg[\int_0^\infty e^{-rt} 
    \Big( - \tau \, \mathrm{KL} (\pi \| \mu ) (S_t) - \varphi(c_{\theta,\tilde\theta}(S_t, \cdot)) \Big)\, \mathrm dt  \bigg]  \\
         & \qquad  + \mathbb E_\rho^{\mathbb{P}^\pi} \bigg[\int_0^\infty  e^{-rt} \sum_{a} \pi(a | S_t) \Big( r \log \Big( \tfrac{R(S_t, a)}{r} \Big)\Big) \mathrm dt  \bigg] \\
         & = \mathbb E_\rho^{\mathbb{P}^\pi} \bigg[\int_0^\infty  e^{-rt} \Big( \sum_{a} \pi(a | S_t) \big( r \log R(S_t, a) \big) - \tau \, \mathrm{KL} (\pi\| \mu ) (S_t)\Big) \mathrm dt  \bigg]  \\
         & \qquad + \inf_{\tilde\theta\in C_\tau^\pi}  \mathbb E_\rho^{\mathbb{P}^\pi} \bigg[\int_0^\infty - e^{-rt} 
    \,  \varphi(c_{\theta,\tilde\theta}(S_t,\cdot)) \, \mathrm dt  \bigg] - \log r \, .
    \end{split}
\end{equation*}
By the definition of $\mathcal C_\tau^\pi$~\eqref{eq:robust_set}, for any $\tilde{\theta} \in C_\tau^\pi(\theta)$,
\begin{equation*}
\begin{split}
\mathbb E_\rho^{\mathbb{P}^\pi} \bigg[\int_0^\infty &- e^{-rt} 
    \,  \varphi(c_{\theta,\tilde\theta}(S_t, \cdot)) \, \mathrm dt  \bigg]  \\ & = 
    - \frac{\tau} {r} \sum_{s} \bar{\mathrm d}_{\rho, \theta}^\pi (s) \log \Big( \sum_a \mu(a | s) \exp \Big( \tfrac{c_{\theta,\tilde\theta}(s,a)}{\tau} \Big) \Big) \geq - \frac{\varepsilon}{r}   \, .
\end{split}
\end{equation*}
Consequently, 
\begin{equation*}
    \begin{split}
        \log   \inf_{\tilde\theta\in C_\tau^\pi}  & \bigg( \mathbb E_\rho^{\mathbb P^\pi_{\tilde \theta}}  \bigg[\int_0^\infty  e^{-rt} \widetilde R^\pi(S_t) \, \mathrm dt \bigg]\bigg)  \\ 
        & \geq \mathbb E_\rho^{\mathbb{P}^\pi} \bigg[\int_0^\infty  e^{-rt} \Big( \sum_{a} \pi(a | S_t) \big( r \log R(S_t, a) \big) - \tau \, \mathrm{KL} (\pi \| \mu )(S_t)\Big) \mathrm dt  \bigg] - \frac{\varepsilon}{r} - \log r\,,
    \end{split}
\end{equation*}
which completes the proof.
\end{proof}

\subsection{Proof of Proposition~\ref{prop:monotonicity_tau}}
\label{appendix: proof of the monotonicity}

\begin{proof}
Fix $\pi\in\Pi$ and $\tilde\theta\in\Theta$. For each $s\in\mathcal S$, define $g_s(\tau)
    :=
    \tau\log
    \big(
        \sum_{a\in\mathcal A_s}
        \mu(a | s)
        e^{c_{\theta,\tilde\theta}(s,a)/\tau}
    \big)$ with $ \tau>0$.
Then the defining criterion of
$\mathcal C_{\tau,\varepsilon}^{\pi}(\theta)$ is
$\Psi_\tau(\tilde\theta)
    :=
    \sum_{s\in\mathcal S}
    \bar{\mathrm d}_{\rho,\theta}^{\pi}(s)g_s(\tau)
    \le \varepsilon$.

The same log-sum-exp argument used in the proof of
Proposition~\ref{prop:monotonicity_tau_global}, applied statewise with
base measure $\mu(\cdot\mid s)$ and function
$c_{\theta,\tilde\theta}(s,\cdot)$, gives
\begin{equation*}
    \frac{\mathrm d}{\mathrm d\tau}g_s(\tau)
    =
    -\mathrm{KL}(q_\tau(\cdot\mid s)\|\mu(\cdot\mid s))
    \le 0,
\end{equation*}
where
\begin{equation*}
    q_\tau(a | s)
    :=
    \frac{
        \mu(a | s)e^{c_{\theta,\tilde\theta}(s,a)/\tau}
    }{
        \sum_{b\in\mathcal A_s}
        \mu(b\mid s)e^{c_{\theta,\tilde\theta}(s,b)/\tau}
    } .
\end{equation*}
Thus each $g_s(\tau)$ is nonincreasing in $\tau$, and therefore
$\Psi_\tau(\tilde\theta)$ is also nonincreasing. Hence, for
$0<\tau_1\le\tau_2$, we have $\Psi_{\tau_2}(\tilde\theta)
    \le
    \Psi_{\tau_1}(\tilde\theta)$,
which proves
    $\mathcal C_{\tau_1,\varepsilon}^{\pi}(\theta)
    \subseteq
    \mathcal C_{\tau_2,\varepsilon}^{\pi}(\theta)$.

It remains to identify the limiting sets. By the standard log-sum-exp limits as discussed in the proof of Proposition~\ref{prop:monotonicity_tau_global},
for each $s\in\mathcal S$,
\begin{equation*}
    \lim_{\tau\rightarrow0}g_s(\tau)
    =
    \max_{a\in\mathcal A_s}c_{\theta,\tilde\theta}(s,a),
    \qquad
    \lim_{\tau\to\infty}g_s(\tau)
    =
    \sum_{a\in\mathcal A_s}
    \mu(a | s)c_{\theta,\tilde\theta}(s,a).
\end{equation*}
Since $\mathcal S$ is finite, these limits pass through the finite weighted
sum defining $\Psi_\tau$. Therefore,
\begin{equation*}
    \mathcal C_{0,\varepsilon}^{\pi}(\theta)
    =
    \Big\{
        \tilde\theta\in\Theta
        \,\Big|\,
        \sum_{s\in\mathcal S}
        \bar{\mathrm d}_{\rho,\theta}^{\pi}(s)
        \max_{a\in\mathcal A_s}
        c_{\theta,\tilde\theta}(s,a)
        \le \varepsilon
    \Big\},
\end{equation*}
and
\begin{equation*}
    \mathcal C_{\infty,\varepsilon}^{\pi}(\theta)
    =
    \Big\{
        \tilde\theta\in\Theta
        \,\Big|\,
        \sum_{s\in\mathcal S}
        \bar{\mathrm d}_{\rho,\theta}^{\pi}(s)
        \sum_{a\in\mathcal A_s}
        \mu(a | s)c_{\theta,\tilde\theta}(s,a)
        \le \varepsilon
    \Big\}.
\end{equation*}
The monotonicity proved above then implies
\begin{equation*}
    \mathcal C_{0,\varepsilon}^{\pi}(\theta)
    \subseteq
    \mathcal C_{\tau,\varepsilon}^{\pi}(\theta)
    \subseteq
    \mathcal C_{\infty,\varepsilon}^{\pi}(\theta),
    \quad \forall \tau>0 .
\end{equation*}
\end{proof}

% ==================================================================
\section{Detailed Setup for the Market Making Example} \label{app: detailed worked example}
% ==================================================================

% \textbf{Single-asset market making.}
We use a market-making problem as the worked example.
In market making, a dealer continuously posts bid and ask quotes, earns the spread when market orders are filled, and controls inventory risk through the quoted distances. 
Since the inventory changes are only triggered by order arrivals, this is naturally an event-driven continuous-time control problem.

We now consider a simplified single-asset market making model. 
The state $q$ is the dealer's inventory taking values in $\mathcal S=\{-10,-9,\ldots,9,10\}$
and an action $\delta=(\delta^-,\delta^+)\in\mathcal A$ specifies the bid and ask quote distances.
Motivated by~\cite{barzykin2023algorithmic}, where optimal quotes remain in a narrow range near zero inventory, we discretise the quote interval $[0.06,0.24]$ bps into 20 evenly spaced points,
\begin{equation*}
    \Delta_{20}
    =
    \{\bar\delta^{(1)},\ldots,\bar\delta^{(20)}\},
    \qquad
    \bar\delta^{(i)}
    =
    \delta_{\min}+(i-1)\tfrac{\delta_{\max}-\delta_{\min}}{19},
    \quad i=1,\ldots,20 ,
\end{equation*}
with $\delta_{\min}=0.06$ and $\delta_{\max}=0.24$. 
The action space is thus $\mathcal A=\Delta_{20}\times\Delta_{20}$ with $|\mathcal A|=400$, 
and we take the reference policy to be uniform, $\mu(\delta | q)=1/|\mathcal A|$.

For a given state--action pair $(q,\delta)$, inventory can only move to $q+1$ or $q-1$, subject to the boundary constraints. 
These jumps occur when the dealer's bid or ask quote is filled by incoming market orders. The fill intensity is the instantaneous arrival rate of such fills: economically, a quote placed further from the mid-price is less attractive and hence is filled less often. Following~\cite{barzykin2023algorithmic}, we use the logistic specification
\begin{equation*}
    \Lambda_\theta(\delta^\pm)
    =
    \frac{\lambda_R}{1+\exp(\alpha_\Lambda+\beta_\Lambda\delta^\pm)},
    \quad
    \theta=(\alpha_\Lambda,\beta_\Lambda),
\end{equation*}
with fixed scale parameter $\lambda_R$. Since $\beta_\Lambda>0$, the fill intensity is decreasing in the quote distance.
The controlled transition rates are therefore
\begin{equation*}
    \lambda^\delta_{q,q+1,\theta}
    =
    \mathbf 1_{\{q<10\}}\Lambda_\theta(\delta^-),
    \qquad
    \lambda^\delta_{q,q-1,\theta}
    =
    \mathbf 1_{\{q>-10\}}\Lambda_\theta(\delta^+).
\end{equation*}
Here a bid fill increases inventory from $q$ to $q+1$, while an ask fill decreases inventory from $q$ to $q-1$.
The running reward is the expected spread capture minus a quadratic inventory penalty:
\begin{equation*}
    R_\theta(q,\delta)
    =
    \mathbf 1_{\{q<10\}}\Lambda_\theta(\delta^-)\delta^-
    +
    \mathbf 1_{\{q>-10\}}\Lambda_\theta(\delta^+)\delta^+
    -
    \frac{\gamma}{2}\sigma^2 q^2 .
\end{equation*}

We use the calibrated parameters from~\cite{barzykin2023algorithmic}, rescaled to minutes for time-dependent quantities. 
The baseline values are thus $\lambda_R=\frac{1000}{8\cdot 60}\ {\rm min}^{-1}$, $\sigma=\frac{50}{\sqrt{8\cdot 60}}\ {\rm bps}\cdot{\rm min}^{-1/2}$, $\alpha_\Lambda=-1$ and $\beta_\Lambda=10\ {\rm bps}^{-1}$. 
The quote-space parameters $\alpha_\Lambda$ and $\beta_\Lambda$ are left unchanged under this rescaling because $\alpha_\Lambda+\beta_\Lambda\delta$ is dimensionless. We set $\gamma=0.05\ {\rm bps}^{-1}\cdot{\rm M\$}^{-1}$.
Since our robustness certificate contains the term $\log(R_{\tilde\theta}/R_\theta)$, rewards must be strictly positive. We therefore use the shifted reward $R_\theta^+(q,\delta):=R_\theta(q,\delta)+C$, where $C:= - \min_{\tilde \theta} \min_{q,\delta} R_{\tilde \theta} (q,\delta) + 10^{-5}$ so that $R_\theta^+(q,\delta)\ge 10^{-5}$ uniformly over all $(q,\delta)$.

Finally, to compare the two policy-dependent robust sets
$\widetilde{\mathcal C}^{\pi}_{\tau,\varepsilon}(\theta)$
in~\eqref{eq:robust_set_entropy} and
$\mathcal C^{\pi}_{\tau,\varepsilon}(\theta)$
in~\eqref{eq:robust_set}, we fix a common behavior policy $\pi$.
The policy is chosen to mimic the inventory-stabilizing structure of the
optimal market-making policy: when the inventory is positive, the policy
favors actions with relatively tighter ask quotes and wider bid quotes,
thereby encouraging sell fills and pushing inventory back towards zero;
when the inventory is negative, it favors tighter bid quotes and wider ask
quotes, encouraging buy fills. Hence the policy induces mean-reverting
inventory dynamics and reflects the economic trade-off between spread
capture and inventory-risk control.

We fix the robustness budget at $\varepsilon=0.01$. Moreover, we set $r=1$.
With this normalization, the lower-bound constants in the two certificates
coincide; see Theorems~\ref{thm:ct_global_sa_robust}
and~\ref{thm:maxent_robust_ctmdp}. Therefore, the numerical comparison
mainly reflects the geometry and relative size of the robust sets, rather
than differences caused by the scale of the corresponding robust lower
bounds.

Under this setting, the local relative entropy between the nominal and
perturbed jump rates is explicit, 
\begin{equation*}
\begin{split}
    \ell(\lambda^\delta_\theta,\lambda^\delta_{\tilde\theta})(q)
    &=
    \mathbf 1_{\{q<10\}}
    \left[
        \Lambda_\theta(\delta^-)
        \log\frac{\Lambda_\theta(\delta^-)}
        {\Lambda_{\tilde\theta}(\delta^-)}
        -
        \Lambda_\theta(\delta^-)
        +
        \Lambda_{\tilde\theta}(\delta^-)
    \right]
    \\
    &\quad+
    \mathbf 1_{\{q>-10\}}
    \left[
        \Lambda_\theta(\delta^+)
        \log\frac{\Lambda_\theta(\delta^+)}
        {\Lambda_{\tilde\theta}(\delta^+)}
        -
        \Lambda_\theta(\delta^+)
        +
        \Lambda_{\tilde\theta}(\delta^+)
    \right].
\end{split}
\end{equation*}
Thus, for each candidate perturbation $\tilde\theta$, the local intensity-based robust constraint is evaluated as
\begin{equation*}
    \Phi_{\mathrm{loc}}(\tilde\theta;\tau)
    :=
    \sum_{q\in\mathcal S}
    \bar{\mathrm d}_{\rho,\theta}^{\pi}(q)
    \,
    \tau
    \log
    \bigg[
        \frac{1}{|\mathcal A|}
        \sum_{\delta\in\mathcal A}
        \exp
        \Big(
            \tfrac{
                \ell(\lambda^\delta_\theta,\lambda^\delta_{\tilde\theta})(q)
                -
                r\log
                \tfrac{R^+_{\tilde\theta}(q,\delta)}
                     {R^+_{\theta}(q,\delta)}
            }{\tau}
        \Big)
    \bigg],
\end{equation*}
and we include $\tilde\theta$ in
$\mathcal C^{\pi}_{\tau,\varepsilon}(\theta)$ whenever
$\Phi_{\mathrm{loc}}(\tilde\theta;\tau)\le \varepsilon$. 
To compute the normalized state occupancy measure, we use the formula
$\bar{\mathrm d}_{\rho,\tilde\theta}^{\pi}
=
r\rho(rI-Q_{\theta}^{\pi})^{-1}$, which follows from the standard resolvent representation of CTMDPs. Here $Q_{\theta}^{\pi}$ is the generator induced by the fixed policy $\pi$ under the parameter $\tilde\theta$, and $\rho$ is the initial distribution.

For the global occupancy-based robust set, the perturbation is measured
through the change in discounted state occupancy. For each $\tilde\theta$,
we first compute
$\bar{\mathrm d}_{\rho,\tilde\theta}^{\pi}
=
r\rho(rI-Q_{\tilde\theta}^{\pi})^{-1}$,
using the same fixed policy $\pi$ and initial distribution $\rho$. The
global constraint is then evaluated as
\begin{equation*}
    \Phi_{\mathrm{glob}}(\tilde\theta;\tau)
    :=
    \tau
    \log
    \bigg[
        \sum_{q\in\mathcal S}
        \bar{\mathrm d}_{\rho,\theta}^{\pi}(q)
        \frac{1}{|\mathcal A|}
        \sum_{\delta\in\mathcal A}
        \exp
        \Big(
            \tfrac{
                \log
                \tfrac{\bar{\mathrm d}_{\rho,\theta}^{\pi}(q)}
                     {\bar{\mathrm d}_{\rho,\tilde\theta}^{\pi}(q)}
                -
                \log
                \tfrac{R^+_{\tilde\theta}(q,\delta)}
                     {R^+_{\theta}(q,\delta)}
            }{\tau}
        \Big)
    \bigg].
\end{equation*}
We include $\tilde\theta$ in
$\widetilde{\mathcal C}^{\pi}_{\tau,\varepsilon}(\theta)$ whenever
$\Phi_{\mathrm{glob}}(\tilde\theta;\tau)\le \varepsilon$.

\begin{figure}[!ht]
    \centering
    \includegraphics[width=1.0\linewidth]{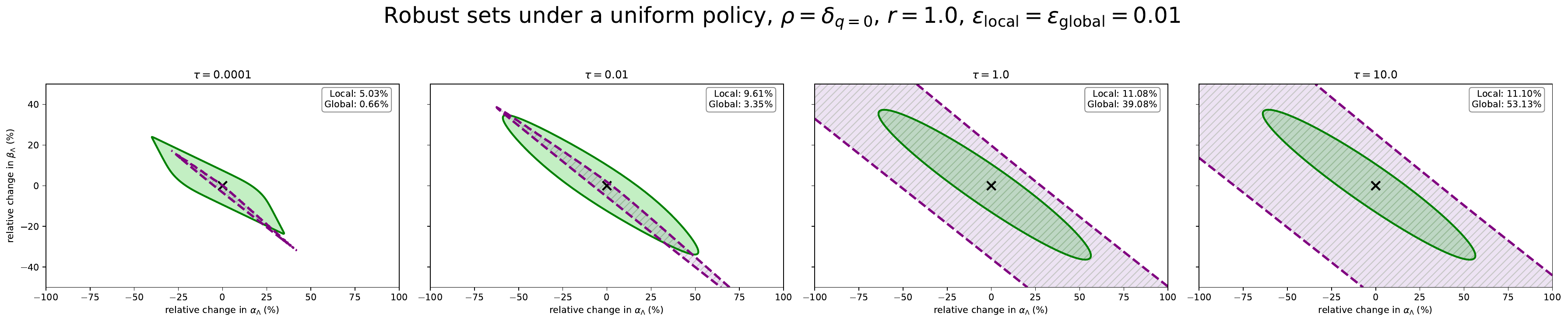}
    \caption{Robust regions for the market making example under the same setting as Figure~\ref{fig:feasible-reference} but with a uniform policy.}
    \label{fig:feasible2}
\end{figure}

Compared with Figure~\ref{fig:feasible-reference}, where the policy has an
inventory-stabilizing structure motivated by the optimal market-making
control, Figure~\ref{fig:feasible2} repeats the experiment under a less informative
benchmark policy, i.e. a uniform policy on $\mathcal A$. 
The local intensity-based certificate is almost unchanged, whereas the global occupancy-based certificate becomes larger. 
This reflects the different nature of the two constraints. 
The local certificate penalizes pointwise jump-rate perturbations averaged under the baseline occupancy, and is little affected here since the initial law is concentrated near zero and the baseline occupancy profile remains similar. 
By contrast, the global certificate depends on the realized occupancy distortion between the baseline and perturbed controlled processes, and is therefore more sensitive to the interaction between the policy and the
perturbed dynamics. 

\begin{figure}[!ht]
    \centering
    \includegraphics[width=1.0\linewidth]{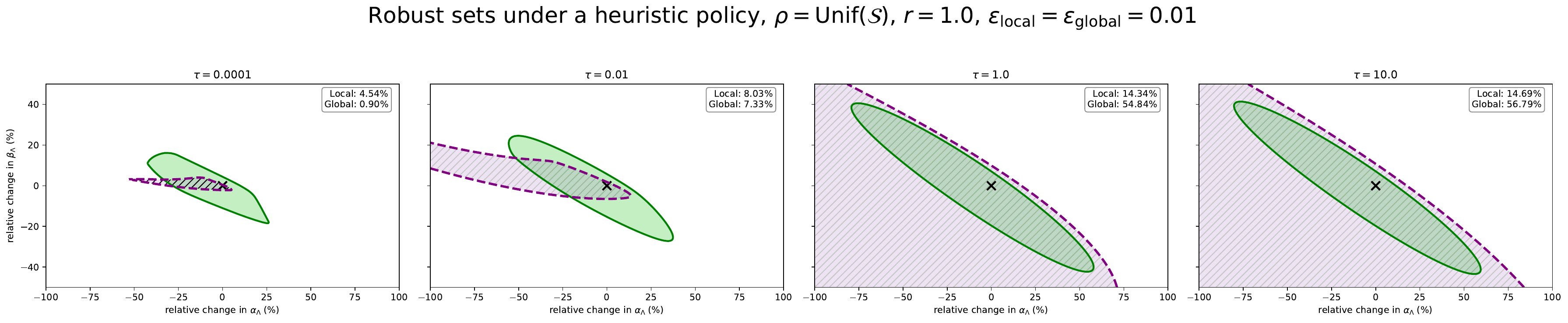}
    \caption{Robust regions for the market making example under the heuristic policy as Figure~\ref{fig:feasible-reference} but with a uniform initial state distribution.}
    \label{fig:feasible3}
\end{figure}

Next, to isolate the effect of the initial distribution, we replace the Dirac distribution at zero ($\rho = \delta_{\{q=0\}}$) with a uniform distribution ($\rho = \mathrm{Unif}(\mathcal{S})$). We maintain the inventory-stabilizing benchmark policy from Figure~\ref{fig:feasible-reference}. Because this policy mirrors the structure of the optimal control, keeping it fixed allows us to focus on robustness within a near-optimal regime.
As shown in Figure~\ref{fig:feasible3}, the local certificate changes only mildly:
changing $\rho$ only changes the baseline occupancy weights, while
the pointwise intensity are unchanged. 
Hence, under the stabilizing policy, the geometry of the local set is largely preserved.
By contrast, the global certificate changes substantially because its
constraint depends on the occupancy ratio. 
For small $\tau$, the log-sum-exp is close to a worst-case constraint, so a large
occupancy ratio at a few states can dominate the certificate.
Compared with $\rho=\delta_{\{q=0\}}$, the uniform initial law spreads mass across states and makes these worst-case ratios less extreme. 
For larger $\tau$, the constraint averages more across states, and the occupancy ratios are closer to one on the states with significant mass. 
This makes the global constraint less restrictive and gives a larger certified region. 

Hence the global certificate, although sometimes much larger, is more sensitive to the policy and initial distribution, while the local certificate is comparatively stable.

% ==================================================================
% ==================================================================
\section{Robustness of Entropy-Regularised Policies: Extended Experiments}\label{app:robustness}
% ==================================================================

This appendix presents empirical evidence that entropy-regularised
policies trained with the arrival-driven framework exhibit improved
worst-case performance under model perturbation.  We evaluate on two
domains---single-asset market making and criss-cross queueing
scheduling---using a common methodology: train at nominal parameters,
then evaluate across a two-dimensional perturbation grid and compare
the worst-case (minimum over the grid) discounted reward of
$\pi_{\mathrm{std}}$ ($\tau = 0$) against $\pi_\tau$ for several
values of the entropy temperature~$\tau$.

% ------------------------------------------------------------------
\subsection{Experimental Setup}\label{ssec:robust_setup}
% ------------------------------------------------------------------

\paragraph{General methodology.}
For each domain, a standard policy $\pi_{\mathrm{std}}$ is trained
with $\tau = 0$ (no entropy bonus), and a family of
entropy-regularised policies $\pi_\tau$ is trained for each $\tau$ in
a logarithmically spaced grid.  All policies are trained at nominal
environment parameters.  At evaluation time, each policy is deployed
without retraining across a $21 \times 21$ grid of perturbed
parameters.  For each grid cell, the mean discounted reward is
estimated over 10 evaluation episodes.  We also evaluate a noise
baseline $\pi_{\mathrm{noise}}$, which uses the $\pi_{\mathrm{std}}$
policy but replaces the chosen action with a uniform random action
with probability $\epsilon = 0.1$.

We report:
\begin{itemize}
  \item \emph{Nominal performance}: mean reward at the unperturbed
    (nominal) grid cell, averaged over seeds.
  \item \emph{Worst-case performance}: minimum over all grid cells of
    the seed-averaged reward.
  \item \emph{Brittleness}: relative drop from nominal to worst-case,
    $(1 - \mathrm{WC}/\mathrm{Nom}) \times 100\%$.
\end{itemize}

% ------------------------------------------------------------------
\subsubsection{Market-Making Environment}\label{sssec:robust_mm_setup}
% ------------------------------------------------------------------

We use a single-asset market-making environment similar to that discussed in the main body of this work (arithmetic Brownian motion
midprice, constant Poisson arrivals) with a logistic fill probability
model.  To ensure the entropy bonus is meaningful relative to the
action space, we reduce the action space to $3 \times 3 = 9$ discrete
actions (bid and ask spreads each drawn from
$\{0.05, 0.25, 0.45\}$), yielding $H_{\max} = \ln 9 \approx 2.20$.

\begin{table}[H]
\centering
\caption{Market-making robustness experiment: environment parameters.}
\label{tab:robust_mm_params}
\small
\begin{tabular}{lrl}
\toprule
Parameter & Value & Description \\
\midrule
Arrival rate $\lambda_R$ & 2.0833 & Order arrival rate \\
Fill intercept $\alpha$ & $-1.0$ & Logistic fill parameter \\
Fill slope $\beta$ & $10.0$ & Logistic fill parameter \\
Inventory penalty $\gamma$ & $0.05$ & Running inventory penalty \\
Volatility $\sigma$ & $2.2822$ & Midprice diffusion coefficient \\
Drift $\mu$ & $0.0$ & Midprice drift \\
Initial price $S_0$ & $100.0$ & Starting midprice \\
Max inventory $q_{\max}$ & $10$ & Inventory bound \\
Terminal time $T$ & $10.0$ & Episode length \\
Trajectories/epoch & $200$ & Batch size \\
\midrule
Bid/ask spread grid & $\{0.05, 0.25, 0.45\}$ & $3 \times 3 = 9$ actions \\
\midrule
Training epochs & $2{,}000$ & Per policy \\
Seeds & $0$--$9$ & 10 independent seeds \\
Network & MLP $[64, 64]$ & Hidden layer sizes \\
Learning rate & $0.001$ & Adam optimiser \\
\bottomrule
\end{tabular}
\end{table}

\paragraph{Perturbation axes.}
We perturb the volatility $\sigma$ and arrival rate $\lambda_R$
jointly via scale factors:
$\sigma' = \sigma \cdot s_\sigma$,
$\lambda_R' = \lambda_R \cdot s_\lambda$,
with $(s_\sigma, s_\lambda) \in [0.25, 4.0]^2$ on a $21 \times 21$
grid.  This range is chosen because volatility enters the optimal
spread quadratically due to the inventory penalty term ($\delta^* \propto \gamma \sigma^2 (T - t)$),
so a $4\times$ increase in $\sigma$ produces a $16\times$ increase
in the inventory penalty contribution, qualitatively changing which
spreads are optimal.

\paragraph{Entropy temperature values.}
$\tau \in \{10^{-4},\; 5 \times 10^{-4},\; 10^{-3},\; 5 \times 10^{-3},\; 10^{-2},\; 10^{-1},\; 1.0\}$.

% ------------------------------------------------------------------
\subsubsection{Queueing Scheduling Environment}\label{sssec:robust_queue_setup}
% ------------------------------------------------------------------

We use the criss-cross queueing network from
\S\ref{sec:experiments}, with two servers and three job classes.
Server~1 can serve classes~1 and~3; server~2 can serve classes~2
and~3.  The four discrete actions correspond to the joint priority
assignments for both servers:
(S1$\to$c1, S2$\to$c2),
(S1$\to$c1, S2$\to$c3),
(S1$\to$c3, S2$\to$c2),
(S1$\to$c3, S2$\to$c3).

\begin{table}[H]
\centering
\caption{Queueing robustness experiment: environment parameters.}
\label{tab:robust_queue_params}
\small
\begin{tabular}{lcccc}
\toprule
& Class 1 & Class 2 & Class 3 & \\
\midrule
Arrival rate $\lambda_i$ & 0.5 & 0.5 & 0.5 & \\
Service rate $\mu_i$ & 2.0 & 3.0 & 2.0 & \\
Holding cost $h_i$ & 3.0 & 1.0 & 2.0 & \\
\midrule
Queue capacity $Q_{\max}$ & \multicolumn{3}{c}{50} & \\
Terminal time $T$ & \multicolumn{3}{c}{50.0} & \\
Trajectories/epoch & \multicolumn{3}{c}{100} & \\
Training epochs & \multicolumn{3}{c}{2{,}000} & \\
Seeds & \multicolumn{3}{c}{$0$--$4$ (5 seeds)} & \\
Network & \multicolumn{3}{c}{MLP $[64, 64]$} & \\
Learning rate & \multicolumn{3}{c}{$0.001$} & \\
\bottomrule
\end{tabular}
\end{table}

\paragraph{Perturbation axes.}
We perturb the arrival rates of classes~1 and~3 independently via
scale factors:
$\lambda_1' = \lambda_1 \cdot s_1$,
$\lambda_3' = \lambda_3 \cdot s_3$,
with $(s_1, s_3) \in [0.5, 3.0]^2$ on a $21 \times 21$ grid, while
$\lambda_2$, all service rates, and all holding costs remain fixed.

This asymmetric perturbation changes which class creates the
bottleneck for server~1 (which must split capacity between classes~1
and~3), thereby changing the optimal scheduling priority.  Crucially,
the base arrival rates are set low enough
($\lambda_i = 0.5$, $\mu_i \geq 2.0$) that even at
$(s_1, s_3) = (3.0, 3.0)$, the system remains within its stability
region.  This ensures that any performance degradation under
perturbation arises from \emph{policy mismatch}---the trained policy
choosing suboptimal priorities---rather than from environmental
overload where no scheduling policy can prevent queue growth.

\paragraph{Entropy temperature values.}
$\tau \in \{10^{-4},\; 5 \times 10^{-4},\; 10^{-3},\; 5 \times 10^{-3},\; 10^{-2},\; 10^{-1},\; 1.0\}$.

% ------------------------------------------------------------------
\subsection{Statistical Methodology}\label{ssec:robust_stats}
% ------------------------------------------------------------------

Each seed produces an independent policy via a different random
initialisation and training trajectory batch.  The trained policy is
then evaluated on independent trajectory batches at each perturbation
grid cell.  For a given policy type (e.g., $\pi_\tau$ at a specific
$\tau$), the mean reward at each grid cell is first averaged over the
evaluation episodes, then averaged across seeds to obtain the
seed-averaged reward surface.  The worst-case performance is the
minimum of this surface over the perturbation grid.

\paragraph{Hypothesis test.}
For each $\tau > 0$, we test whether entropy regularisation improves
worst-case performance:
\[
  H_0\colon \mathrm{WC}(\pi_\tau) \leq \mathrm{WC}(\pi_{\mathrm{std}})
  \qquad \text{vs.} \qquad
  H_1\colon \mathrm{WC}(\pi_\tau) > \mathrm{WC}(\pi_{\mathrm{std}}).
\]
Let $(i^*_\tau, j^*_\tau) = \operatorname{arg\,min}_{i,j}\, \bar{R}_\tau(i,j)$ and
$(i^*_0, j^*_0) = \operatorname{arg\,min}_{i,j}\, \bar{R}_0(i,j)$ denote the
worst-case grid cells for $\pi_\tau$ and $\pi_{\mathrm{std}}$
respectively, where $\bar{R}$ is the seed-averaged reward.  The test
statistic is
\[
  t = \frac{\mathrm{WC}(\pi_\tau) - \mathrm{WC}(\pi_{\mathrm{std}})}
           {\sqrt{\mathrm{SE}_\tau^2 + \mathrm{SE}_0^2}},
\]
where $\mathrm{SE}_\tau$ and $\mathrm{SE}_0$ are the standard errors
(across seeds) at the respective worst-case grid cells.  Under $H_0$,
this follows an approximate $t$-distribution (Welch's $t$-test with
unequal variances).  We report one-sided $p$-values.

\paragraph{Justification.}
Welch's $t$-test is appropriate here because: (i)~each seed produces
a fully independent policy and evaluation, so seeds are the natural
unit of replication; (ii)~Welch's variant does not assume equal
variance across the $\pi_\tau$ and $\pi_{\mathrm{std}}$ conditions;
and (iii)~with 5--10 seeds, the $t$-distribution (rather than a
$z$-test) correctly accounts for small-sample uncertainty in the
variance estimate.  We do not apply multiple comparison corrections
across $\tau$ values, as the goal is to identify the best $\tau$
rather than to claim all values improve robustness.  The systematic
inverted-U pattern across $\tau$ (visible in
Figures~\ref{fig:robust_mm_tau} and~\ref{fig:robust_queue_tau})
provides additional evidence against cherry-picking.

% ------------------------------------------------------------------
\subsection{Market-Making Results}\label{ssec:robust_mm_results}
% ------------------------------------------------------------------

Table~\ref{tab:robust_mm_results} reports the full results for the
market-making experiment.  The standard policy exhibits $7.1\%$
brittleness under volatility and arrival rate perturbation.

\begin{table}[H]
\centering
\caption{Market-making robustness results (10 seeds, 2{,}000 epochs).
  Nominal and worst-case (WC) are seed-averaged discounted rewards.
  Gain is relative to $\pi_{\mathrm{std}}$ worst-case.  Statistically
  significant improvements ($p < 0.05$, one-sided Welch's $t$-test)
  are shown in \textbf{bold}.}
\label{tab:robust_mm_results}
\small
\begin{tabular}{rccccr}
\toprule
$\tau$ & Nominal (SE) & Worst-Case (SE) & WC Gain & $t$-stat & $p$ \\
\midrule
$0$ ($\pi_{\mathrm{std}}$)
  & $231.11\;(0.06)$ & $214.60\;(0.18)$ & --- & --- & --- \\
noise ($\epsilon{=}0.1$)
  & $231.07$ & $214.31$ & $-0.13\%$ & --- & --- \\
\midrule
$10^{-4}$
  & $231.14\;(0.06)$ & $214.80\;(0.30)$ & $+0.10\%$ & $0.60$ & $0.28$ \\
$\mathbf{5 \times 10^{-4}}$
  & $\mathbf{231.22\;(0.06)}$ & $\mathbf{215.27\;(0.21)}$
  & $\mathbf{+0.32\%}$ & $\mathbf{2.48}$ & $\mathbf{<0.01}$ \\
$\mathbf{10^{-3}}$
  & $\mathbf{231.25\;(0.09)}$ & $\mathbf{215.62\;(0.43)}$
  & $\mathbf{+0.48\%}$ & $\mathbf{2.19}$ & $\mathbf{<0.025}$ \\
$5 \times 10^{-3}$
  & $231.07\;(0.05)$ & $214.24\;(0.32)$ & $-0.17\%$ & $-0.96$ & --- \\
$10^{-2}$
  & $231.09\;(0.07)$ & $214.69\;(0.29)$ & $+0.04\%$ & $0.27$ & $0.40$ \\
$10^{-1}$
  & $230.89\;(0.07)$ & $213.77\;(0.23)$ & $-0.39\%$ & $-2.81$ & --- \\
$1.0$
  & $230.65\;(0.05)$ & $211.98\;(0.10)$ & $-1.22\%$ & $-12.81$ & --- \\
\bottomrule
\end{tabular}
\end{table}

Two entropy temperatures yield statistically significant worst-case
improvements: $\tau = 5 \times 10^{-4}$ ($+0.32\%$, $t = 2.48$,
$p < 0.01$) and $\tau = 10^{-3}$ ($+0.48\%$, $t = 2.19$,
$p < 0.025$).  Both also slightly improve nominal performance,
indicating no robustness--performance trade-off.  Larger $\tau$ values
degrade both nominal and worst-case performance, consistent with
excessive entropy overwhelming the reward signal.

\begin{figure}[H]
\centering
\begin{subfigure}[t]{0.48\textwidth}
  \centering
  \includegraphics[width=\textwidth]{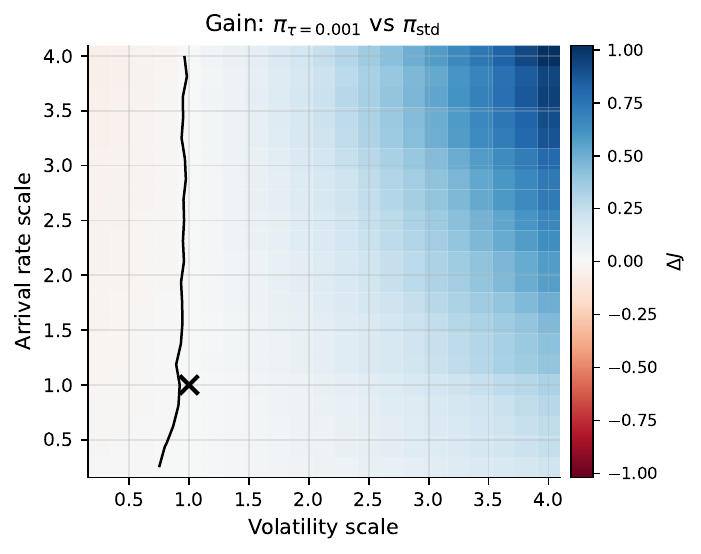}
  \caption{Gain of $\pi_{\tau=5\times10^{-4}}$ over $\pi_{\mathrm{std}}$.}
  \label{fig:robust_mm_gain_std}
\end{subfigure}
\hfill
\begin{subfigure}[t]{0.48\textwidth}
  \centering
  \includegraphics[width=\textwidth]{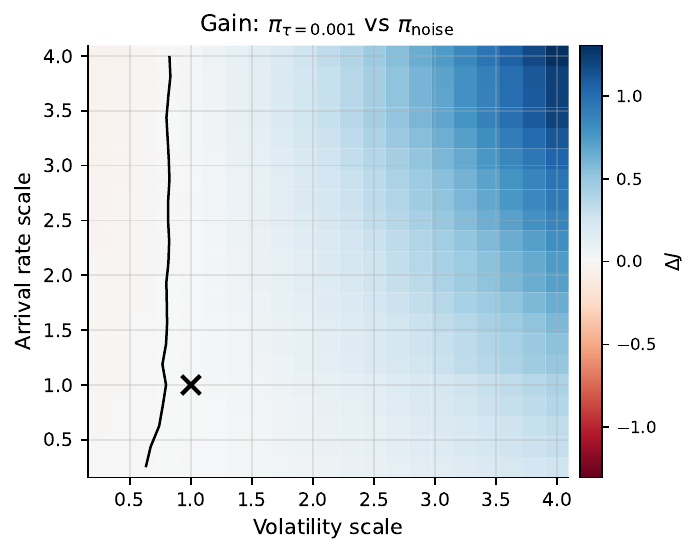}
  \caption{Gain of $\pi_{\tau=5\times10^{-4}}$ over $\pi_{\mathrm{noise}}$.}
  \label{fig:robust_mm_gain_noise}
\end{subfigure}
\caption{Market-making: reward gain heatmaps across the volatility
  scale $\times$ arrival rate scale perturbation grid.  Positive
  (blue) regions indicate where the entropy-regularised policy
  outperforms the baseline.  The entropy-regularised policy gains
  most in the high-volatility region (top of grid), where the optimal
  spread changes qualitatively from the nominal setting.}
\label{fig:robust_mm_heatmaps}
\end{figure}

\begin{figure}[H]
\centering
\includegraphics[width=0.65\textwidth]{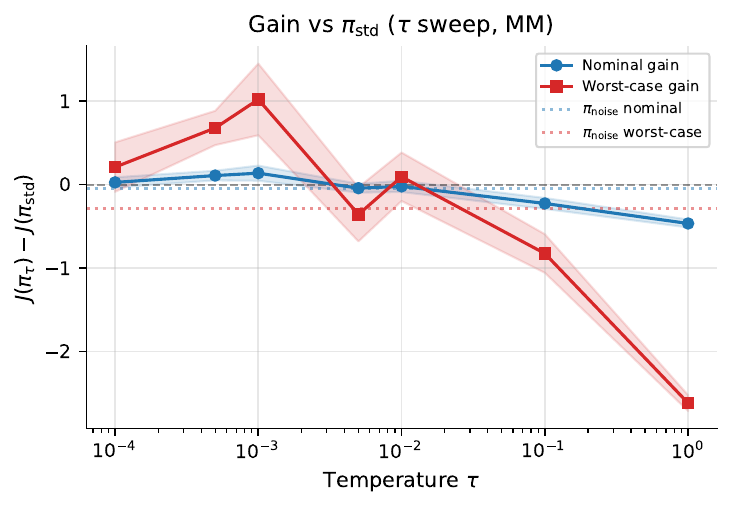}
\caption{Market-making: worst-case and nominal gain vs.\
  $\pi_{\mathrm{std}}$ as a function of $\tau$ (10 seeds, mean
  $\pm$ SE).  The inverted-U shape confirms that moderate entropy
  regularisation improves worst-case robustness, while excessive
  regularisation degrades performance.}
\label{fig:robust_mm_tau}
\end{figure}

% ------------------------------------------------------------------
\subsection{Queueing Scheduling Results}\label{ssec:robust_queue_results}
% ------------------------------------------------------------------

Table~\ref{tab:robust_queue_results} reports the full results for the
queueing scheduling experiment.  The standard policy exhibits $8.9\%$
brittleness under asymmetric arrival rate perturbation.

\begin{table}[H]
\centering
\caption{Queueing robustness results (5 seeds, 2{,}000 epochs).
  Format matches Table~\ref{tab:robust_mm_results}.}
\label{tab:robust_queue_results}
\small
\begin{tabular}{rccccr}
\toprule
$\tau$ & Nominal (SE) & Worst-Case (SE) & WC Gain & $t$-stat & $p$ \\
\midrule
$0$ ($\pi_{\mathrm{std}}$)
  & $14{,}912\;(1.0)$ & $13{,}589\;(105)$ & --- & --- & --- \\
noise ($\varepsilon{=}0.1$)
  & $14{,}912$ & $13{,}615$ & $+0.19\%$ & --- & --- \\
\midrule
$\mathbf{10^{-4}}$
  & $\mathbf{14{,}914\;(0.1)}$ & $\mathbf{13{,}871\;(52)}$
  & $\mathbf{+2.08\%}$ & $\mathbf{2.41}$ & $\mathbf{<0.01}$ \\
$5 \times 10^{-4}$
  & $14{,}913\;(1.2)$ & $13{,}744\;(113)$ & $+1.14\%$ & $1.00$ & $0.16$ \\
$10^{-3}$
  & $14{,}913\;(1.4)$ & $13{,}771\;(128)$ & $+1.34\%$ & $1.09$ & $0.14$ \\
$5 \times 10^{-3}$
  & $14{,}912\;(1.1)$ & $13{,}593\;(112)$ & $+0.03\%$ & $0.02$ & $0.49$ \\
$10^{-2}$
  & $14{,}913\;(1.1)$ & $13{,}666\;(114)$ & $+0.56\%$ & $0.49$ & $0.31$ \\
$10^{-1}$
  & $14{,}911\;(1.1)$ & $13{,}528\;(87)$ & $-0.45\%$ & $-0.45$ & --- \\
$1.0$
  & $14{,}914\;(0.4)$ & $13{,}799\;(32)$ & $+1.55\%$ & $1.91$ & $0.03$ \\
\bottomrule
\end{tabular}
\end{table}

The best entropy temperature is $\tau = 10^{-4}$, which achieves a
statistically significant $+2.08\%$ worst-case improvement ($t = 2.41$,
$p < 0.01$).  At $\tau = 1.0$, the improvement is $+1.55\%$
($t = 1.91$, $p = 0.03$).  As in market making, nominal performance
is virtually unchanged across all $\tau$ values, confirming that the
robustness benefit comes without a performance trade-off.

\begin{figure}[H]
\centering
\begin{subfigure}[t]{0.48\textwidth}
  \centering
  \includegraphics[width=\textwidth]{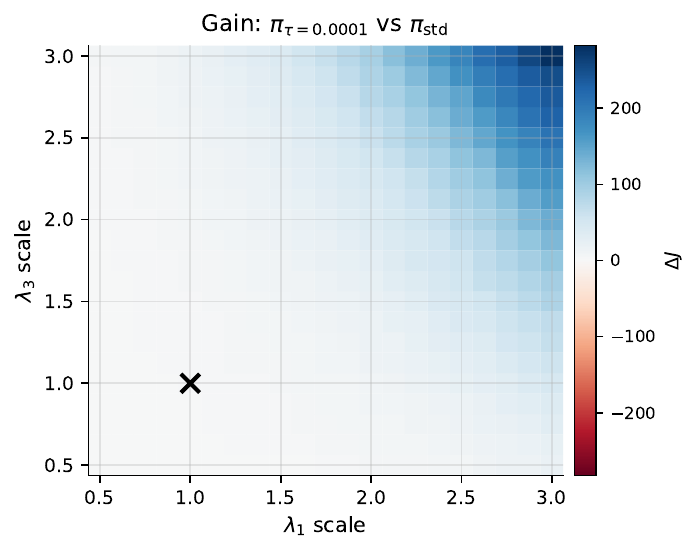}
  \caption{Gain of $\pi_{\tau=10^{-4}}$ over $\pi_{\mathrm{std}}$.}
  \label{fig:robust_queue_gain_std}
\end{subfigure}
\hfill
\begin{subfigure}[t]{0.48\textwidth}
  \centering
  \includegraphics[width=\textwidth]{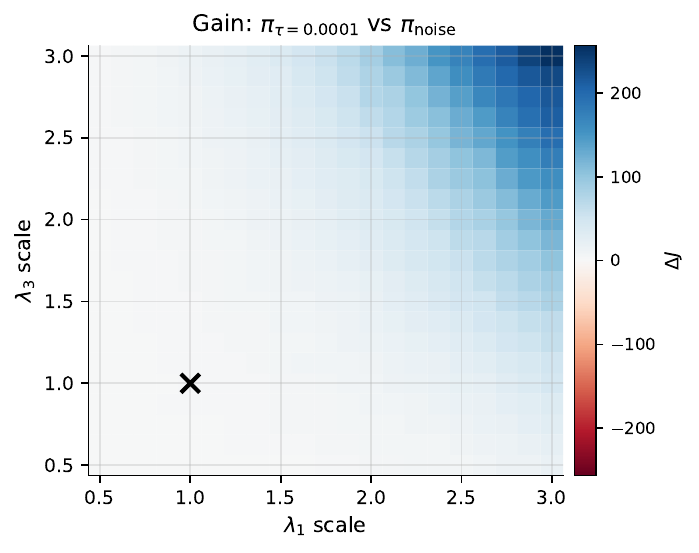}
  \caption{Gain of $\pi_{\tau=10^{-4}}$ over $\pi_{\mathrm{noise}}$.}
  \label{fig:robust_queue_gain_noise}
\end{subfigure}
\caption{Queueing: reward gain heatmaps across the
  $\lambda_1$ scale $\times$ $\lambda_3$ scale perturbation grid.
  The entropy-regularised policy gains most in the high-arrival
  corner ($s_1 = s_3 = 3.0$), where the load balance between
  classes~1 and~3 for server~1 shifts from the nominal setting.}
\label{fig:robust_queue_heatmaps}
\end{figure}

\begin{figure}[H]
\centering
\includegraphics[width=0.65\textwidth]{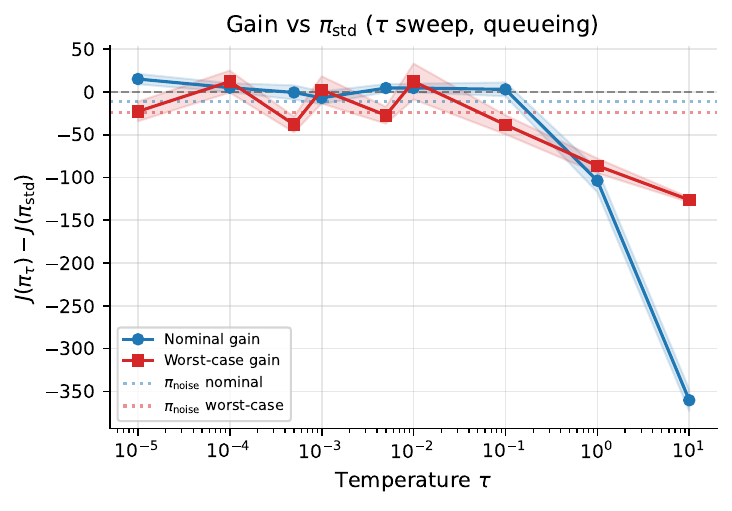}
\caption{Queueing: worst-case and nominal gain vs.\
  $\pi_{\mathrm{std}}$ as a function of $\tau$ (5 seeds, mean
  $\pm$ SE).  Small $\tau$ values (~$10^{-4}$) improve worst-case performance, while $\tau = 10^{-1}$ and greater slightly
  degrades it.}
\label{fig:robust_queue_tau}
\end{figure}

% ------------------------------------------------------------------
\subsection{Discussion}\label{ssec:robust_discussion}
% ------------------------------------------------------------------

Across both domains, entropy regularisation with appropriately chosen
$\tau$ significantly improves worst-case performance under model
perturbation without degrading nominal performance.  The optimal
$\tau$ is small in both cases ($5 \times 10^{-4}$ for market making,
$10^{-4}$ for queueing), indicating that light regularisation
suffices to hedge against parameter uncertainty.

The mechanism is consistent across domains: entropy regularisation
prevents the policy from concentrating all probability mass on a
single action that is optimal only under the nominal parameterisation.
By maintaining probability mass on alternative actions, the
entropy-regularised policy remains competitive when the optimal action
shifts under perturbation.  In market making, higher volatility makes
wider spreads optimal; in queueing, higher arrival rates for class~3
make prioritising class~3 on server~1 optimal.  In both cases,
$\pi_{\mathrm{std}}$ commits too strongly to the nominally optimal
action, while $\pi_\tau$ hedges.

% We note that an earlier version of the queueing experiment perturbed
% arrival and service rates uniformly (scaling all rates by a common
% factor).  This produced $73.9\%$ brittleness but no improvement from
% entropy regularisation.  The root cause was that uniform rate scaling
% pushes the system past its stability boundary ($\rho > 1$), where
% queues grow unboundedly regardless of the scheduling policy.  In this
% regime, brittleness arises from environmental overload rather than
% policy mismatch, and no amount of action-space smoothing can help.
% The asymmetric perturbation used here keeps the system stable while
% changing the optimal policy, isolating the regime where
% entropy regularisation is effective.

% ==================================================================
\section{Certificate Tightness Analysis}\label{app:certificate}
% ==================================================================

This section analyses the tightness of the theoretical robustness
certificate from Theorem~\ref{thm:maxent_robust_ctmdp}.  We use the same
market-making parameterization and perturbation axes as in
Section~\ref{sec:worked_example}, with $\theta = (\alpha_\Lambda, \beta_\Lambda)$.
Recall that the robust set $\mathcal C^\pi_{\tau,\varepsilon}(\theta)$
in~\eqref{eq:robust_set} is defined by the constraint
$\sum_s \bar{\mathrm d}^\pi_{\rho,\theta}(s)\, \tau \log\big[\sum_a \mu(a|s) \exp(c_{\theta,\tilde\theta}(s,a)/\tau)\big] \leq \varepsilon$.
For the market-making parameterization, note that the state-action pair is $(q,\delta)$, 
we define a pointwise certificate function by 
\begin{equation*}
    F(q;\tilde{\theta}) := \tau \log\big[\sum_\delta \mu(\delta|q) \exp(c_{\theta,\tilde\theta}(q,\delta)/\tau)\big], 
\end{equation*}
and the scalar
\begin{equation*}
    G(\tilde{\theta}) := \max_q F(q;\tilde{\theta}) - r\log(1/r).
\end{equation*}

Since $\max_q F \geq \sum_q \bar{\mathrm d}^\pi(q) F(q)$, the condition
$G(\tilde{\theta}) \leq 0$ implies
$\tilde{\theta} \in \mathcal C^\pi_{\tau,\varepsilon}(\theta)$
(with $\varepsilon = r\log(1/r)$).
Thus the certificate guarantees that
the entropy-regularized policy~$\pi_\tau$ maintains performance
whenever $G(\tilde{\theta}) \leq 0$.  We evaluate this on the
single-asset market-making environment with logistic fill
probabilities, perturbing the fill parameters
$(\alpha_\Lambda, \beta_\Lambda)$ and $(\lambda_R, \beta_\Lambda)$ on $21 \times 21$ grids.
The analytical softmax policy
$\pi_\tau(\delta| q) = \mathrm{softmax}(R_\theta(q,\delta)/\tau)$ is used
throughout, matching the theoretical assumptions exactly, along with learned PG and A2C policies.

% ------------------------------------------------------------------
\subsection{$G$ as a Continuous Robustness Predictor}\label{ssec:cert_scatter}
% ------------------------------------------------------------------

Figure~\ref{fig:cert_scatter} plots the certificate value
$G(\tilde{\theta})$ against the empirical performance ratio
$J(\pi, \tilde{\theta}) / J(\pi, \theta)$ for all
$441$ grid cells on the $\alpha_\Lambda$-$\beta_\Lambda$ perturbation plane,
evaluated under three policy types: the analytical softmax policy
$\pi_\tau(\delta| q) = \mathrm{softmax}(R_\theta(q,\delta)/\tau)$ that the
certificate theorem assumes, and two learned policies (Soft-PG and
Soft-AC) trained with entropy regularization.  All three policy
types cluster tightly together, confirming that the learned
policies converge to essentially the same behavior as the
theoretical optimum in this environment.
The certificate is valid across all policy types: every certified
point ($G \leq 0$, black squares) achieves
$J/J_{\mathrm{nom}} \geq 0.95$, with zero violations.  Moreover,
$G$ is monotonically predictive of degradation across the entire
grid---not just within the certified region.
Table~\ref{tab:cert_tightness} quantifies this: no cell
with $J/J_{\mathrm{nom}} < 0.90$ has $G < 1.03$, and no cell with
$J/J_{\mathrm{nom}} < 0.50$ has $G < 36.4$.

\begin{figure}[H]
\centering
\includegraphics[width=1.01\textwidth]{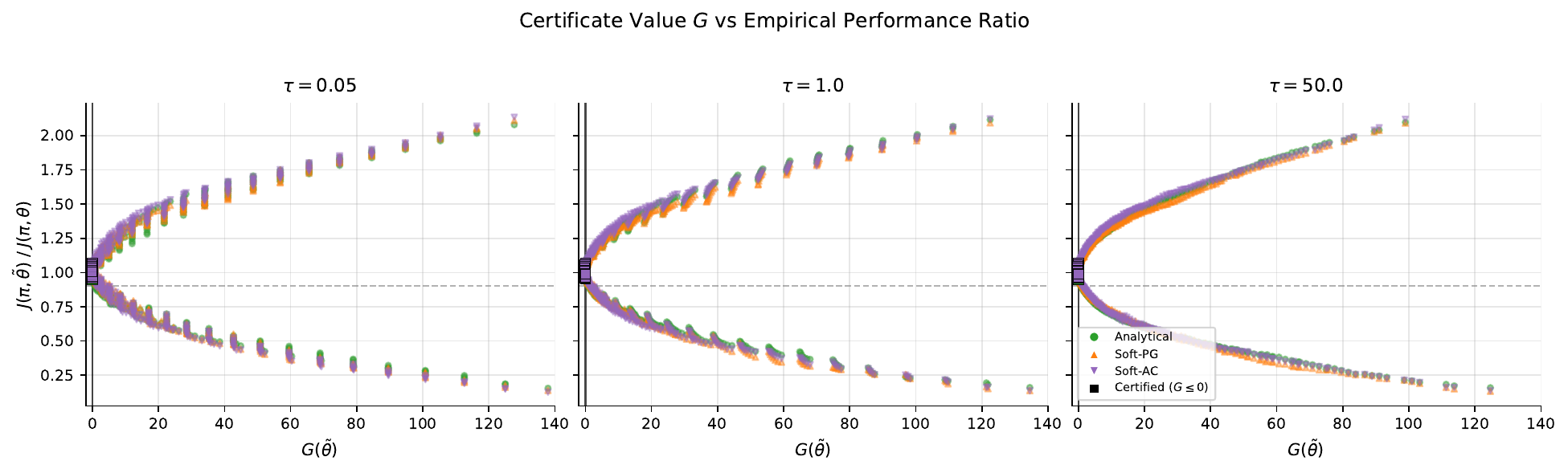}
\caption{Certificate value $G(\tilde{\theta})$ vs empirical
  performance ratio for three entropy temperatures and three
  policy types (analytical softmax, Soft-PG, Soft-AC).  All
  three cluster tightly, confirming that learned policies match
  the theoretical optimum.  Certified points ($G \leq 0$, black
  squares) all lie above the $0.9$ threshold (dashed grey).}
\label{fig:cert_scatter}
\end{figure}

\begin{table}[H]
\centering
\caption{Tightness of the certificate ($\tau = 1.0$,
  $\alpha_\Lambda$-$\beta_\Lambda$ plane).  For each $G$ threshold, the table
  reports the number of included grid cells and the minimum
  empirical performance ratio within that set.}
\label{tab:cert_tightness}
\small
\begin{tabular}{rccc}
\toprule
$G$ threshold & Cells included & Min $J/J_{\mathrm{nom}}$ & \\
\midrule
$G \leq 0$ (certificate) & 24 / 441 & 0.953 & \\
$G \leq 1$   & 69  & 0.903 & \\
$G \leq 5$   & 144 & 0.806 & \\
$G \leq 10$  & 202 & 0.726 & \\
$G \leq 50$  & 367 & 0.449 & \\
\bottomrule
\end{tabular}
\end{table}

The certificate boundary $G = 0$ is conservative: the empirical
$90\%$-robustness region contains $257$ cells while the certificate
guarantees only $24$.  However, the conservatism is bounded---the
``true'' $G$ threshold for $90\%$ robustness is approximately
$G \approx 1$, so the certificate is off by roughly one $G$-unit
rather than by orders of magnitude.

% ------------------------------------------------------------------
\subsection{Certified Region Size vs $\tau$}\label{ssec:cert_tau}
% ------------------------------------------------------------------

Figure~\ref{fig:cert_tau} plots the size of the certified region
$|\{\tilde{\theta} : G(\tilde{\theta}) \leq 0\}|$ as a function
of the entropy temperature~$\tau$, computed on a dense logarithmic
grid from $\tau = 10^{-3}$ to $\tau = 10^2$.  The theoretical
prediction is confirmed: the certified region grows monotonically
with~$\tau$ on both the $\alpha_\Lambda$-$\beta_\Lambda$ and
$\lambda_R$-$\beta_\Lambda$ perturbation planes.  The growth
saturates around $\tau \approx 1$, reaching $24$ cells
($\alpha_\Lambda$-$\beta_\Lambda$) and $26$ cells ($\lambda_R$-$\beta_\Lambda$).
This plateau reflects the geometric structure of the certificate:
as $\tau \to \infty$, the soft-max in the $G$ formula converges
to a hard maximum, and the certified set approaches its
$\tau$-independent upper bound.

For comparison, the empirical robust region (defined as
$J/J_{\mathrm{nom}} \geq 0.9$) contains $\sim\!260$ cells
and is approximately constant across~$\tau$.  The ${\sim}10\times$
gap between certified and empirical region sizes reflects the
conservatism identified in \S\ref{ssec:cert_scatter}.

\begin{figure}[H]
\centering
\includegraphics[width=0.6\textwidth]{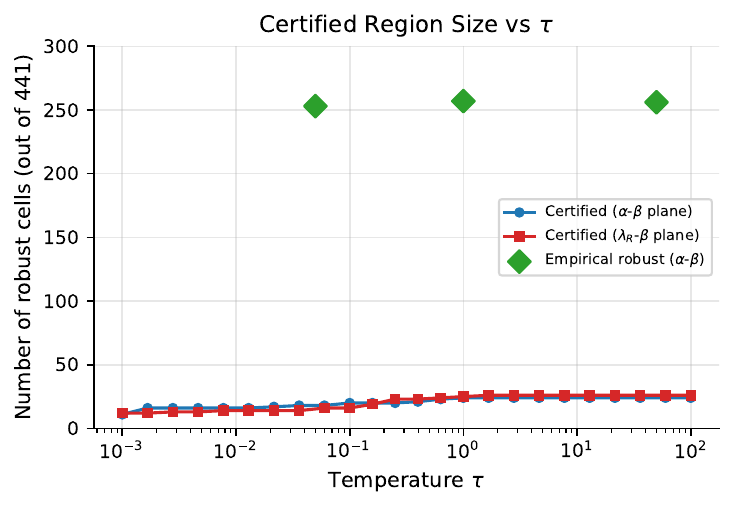}
\caption{Certified region size (number of grid cells with
  $G \leq 0$) as a function of~$\tau$.  Both perturbation planes
  show monotone growth, confirming the theoretical prediction.
  Green diamonds show the empirical robust region size for
  reference.}
\label{fig:cert_tau}
\end{figure}

% ------------------------------------------------------------------
\subsection{Transition vs Reward Decomposition}\label{ssec:cert_decomp}
% ------------------------------------------------------------------

The certificate function decomposes as
$X_\delta(q) = \ell_\delta(q) - r \cdot \log(R_{\tilde\theta}(q,\delta) / R_\theta(q,\delta))$,
where $\ell_a$ is the local relative entropy (KL divergence)
between baseline and perturbed fill intensities
(\emph{transition perturbation}), and the log-ratio term captures
the \emph{reward perturbation}.  Figure~\ref{fig:cert_decomp}
decomposes these two contributions at representative perturbation
points.

\begin{figure}[H]
\centering
\includegraphics[width=0.75\textwidth]{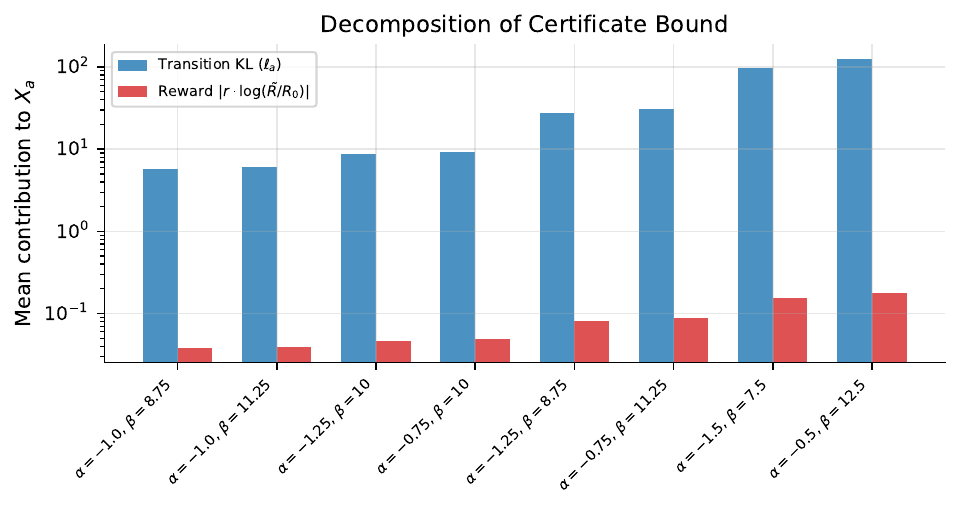}
\caption{Decomposition of the certificate bound into transition
  KL ($\ell_a$) and reward perturbation
  ($|r \cdot \log(R_{\tilde\theta}/R_\theta)|$) contributions, shown on a
  log scale.  The transition KL dominates by $1.5$--$3$ orders
  of magnitude at all perturbation points.}
\label{fig:cert_decomp}
\end{figure}

The transition KL term dominates overwhelmingly: it exceeds the
reward term by a factor of $150$--$800\times$ across all grid
points.  This has two implications.  First, the policy is
inherently robust to reward misspecification---a reward-only
certificate (setting $\ell_\delta = 0$) would cover nearly the entire
perturbation grid.  Second, the conservatism of the joint
certificate is driven entirely by the transition KL term.  Since
the local relative entropy of logistic fill intensities grows
quadratically with parameter perturbation, even moderate changes
in $(\alpha_\Lambda, \beta_\Lambda)$ produce large KL values that push $G$ above
zero.  Tightening the transition bound (e.g., via action-dependent
KL budgets or state-dependent certificates) would directly enlarge
the certified region.

% ------------------------------------------------------------------
\subsection{Worst-Case Adversary Within the Certified Set}\label{ssec:cert_adversary}
% ------------------------------------------------------------------

For each $\tau$, we identify the worst-case perturbation
\emph{within} the certified set: \\
\begin{center}
$\tilde{\theta}^* = \arg\min_{\{G(\tilde{\theta}) \leq 0\}}
J(\pi_\tau, \tilde{\theta})$.
\end{center}

\begin{table}[H]
\centering
\caption{Worst-case adversary within the certified set
  ($\alpha_\Lambda$-$\beta_\Lambda$ plane) for all three policy types.
  The certificate guarantees robustness at these points;
  the empirical performance ratio confirms it.}
\label{tab:cert_adversary}
\small
\begin{tabular}{rlcccc}
\toprule
$\tau$ & Policy & Worst-case $(\alpha_\Lambda, \beta_\Lambda)$ & $G(\tilde{\theta}^*)$
  & $J$ & $J / J_{\mathrm{nom}}$ \\
\midrule
\multirow{3}{*}{$0.05$}
  & Analytical & $(-1.00, 10.25)$ & $-0.088$ & $58.66$ & $0.946$ \\
  & Soft-PG    & $(-1.00, 10.25)$ & $-0.088$ & $58.16$ & $0.969$ \\
  & Soft-AC    & $(-1.00, 10.25)$ & $-0.088$ & $57.37$ & $0.961$ \\
\midrule
\multirow{3}{*}{$1.0$}
  & Analytical & $(-0.80, 9.00)$  & $-0.019$ & $58.16$ & $0.953$ \\
  & Soft-PG    & $(-1.00, 10.25)$ & $-0.114$ & $58.49$ & $0.967$ \\
  & Soft-AC    & $(-1.00, 10.25)$ & $-0.114$ & $57.77$ & $0.960$ \\
\midrule
\multirow{3}{*}{$50.0$}
  & Analytical & $(-1.00, 10.25)$ & $-0.116$ & $57.63$ & $0.946$ \\
  & Soft-PG    & $(-1.00, 10.25)$ & $-0.116$ & $58.17$ & $0.960$ \\
  & Soft-AC    & $(-1.00, 10.25)$ & $-0.116$ & $57.78$ & $0.959$ \\
\bottomrule
\end{tabular}
\end{table}

Across all policy types and all $\tau$ values, the worst-case
certified point achieves $J/J_{\mathrm{nom}} \geq 0.946$, well
above the $0.9$ robustness threshold.  The learned policies
(Soft-PG, Soft-AC) achieve slightly \emph{higher} worst-case
ratios than the analytical softmax ($0.96$--$0.97$ vs $0.95$),
likely because the narrow action space makes all policies nearly
equivalent.  The worst-case perturbations are consistently at
$(\alpha_\Lambda, \beta_\Lambda) = (-1.00, 10.25)$---a $\beta_\Lambda$ shift of only
$0.25$ from nominal---with $G$ values close to zero (boundary of
the certified set).  This confirms that the certificate boundary
coincides with the ``hardest'' perturbations it can guarantee,
and that the guarantee holds equally for learned and analytical
policies.

% ------------------------------------------------------------------
\subsection{Perturbation Budget Curves}\label{ssec:cert_budget}
% ------------------------------------------------------------------

Figure~\ref{fig:cert_budget} provides one-dimensional
cross-sections through the certificate surface and empirical
performance, fixing one parameter at its nominal value and sweeping
the other.  These curves directly visualise the safety margin
between the certificate boundary ($G = 0$, blue dashed line) and
the onset of significant performance degradation
($J/J_{\mathrm{nom}} = 0.9$, red dashed line).

\begin{figure}[H]
\centering
\includegraphics[width=\textwidth]{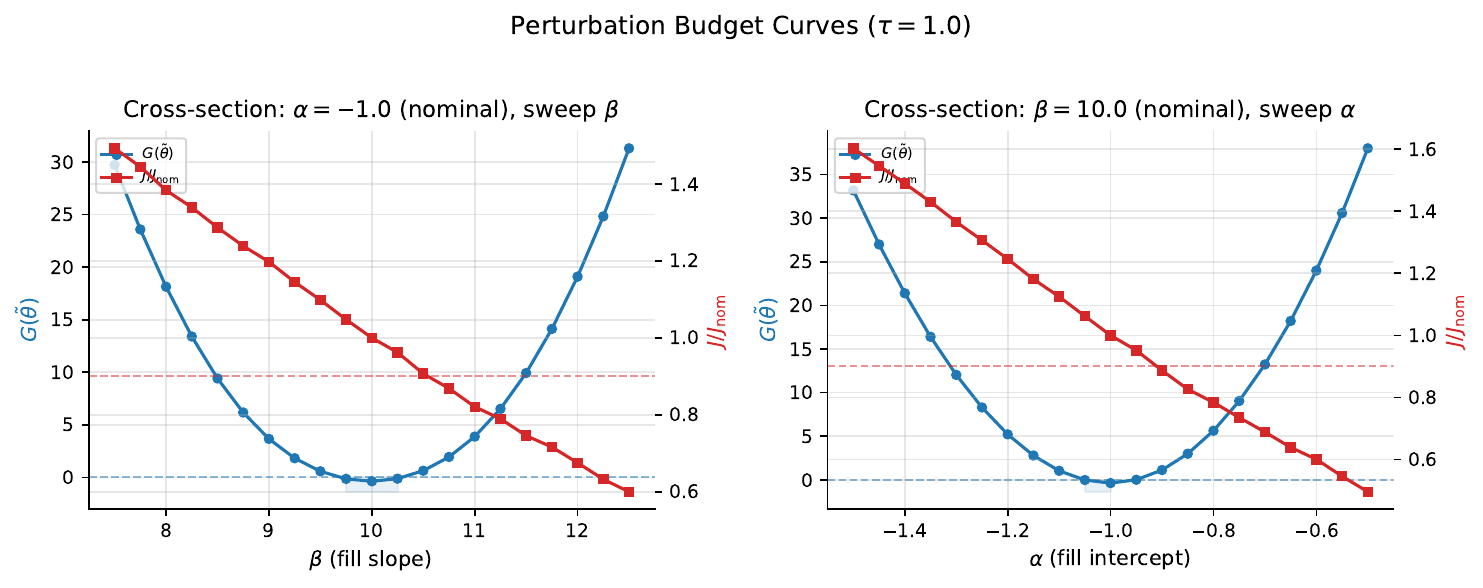}
\caption{One-dimensional cross-sections at $\tau = 1.0$.
  \textbf{Left:} sweep $\beta_\Lambda$ with $\alpha_\Lambda = -1.0$ fixed.
  \textbf{Right:} sweep $\alpha_\Lambda$ with $\beta_\Lambda = 10.0$ fixed.
  The certificate boundary ($G = 0$, blue) is crossed well before
  the empirical $90\%$ performance boundary (red), quantifying the
  conservatism along each perturbation direction.}
\label{fig:cert_budget}
\end{figure}

Along the $\beta_\Lambda$-axis (left panel), the certificate allows
perturbations of roughly $\pm 0.5$ around the nominal $\beta_\Lambda = 10$,
while the empirical $90\%$ robustness extends to approximately
$\beta_\Lambda \in [8.5, 12.0]$.  Along the $\alpha_\Lambda$-axis (right panel),
the certificate permits $\alpha_\Lambda \in [-1.1, -0.9]$ while the
empirical region extends to $\alpha_\Lambda \in [-1.3, -0.7]$.  In both
directions, the certificate is conservative by a factor of roughly
$2$--$3\times$ in perturbation magnitude.  Importantly, performance
degrades \emph{gradually} beyond the certificate boundary (there
is no cliff edge), suggesting that the certificate provides a
useful (if conservative) early warning even beyond its formal
guarantee.

% ------------------------------------------------------------------
\subsection{Reward-Only Certificate}\label{ssec:cert_reward_only}
% ------------------------------------------------------------------

Theorem~\ref{thm:exact_reward_only_robustness} provides an exact equivalence for
reward-only perturbations (setting $\ell_\delta = 0$). Since
\S\ref{ssec:cert_decomp} established that the transition KL term
$\ell_\delta$ dominates the certificate by $150$--$800\times$ over the
reward term, we expect that dropping $\ell_\delta$ should certify nearly
the entire perturbation grid.

Figure~\ref{fig:cert_reward_scatter} confirms this prediction.
The reward-only certificate $G_{\mathrm{rw}}$ certifies
$401$--$409$ out of $441$ grid cells across all three
$\tau$ values, compared to $18$--$24$ for the joint
certificate~$G$.

\begin{figure}[H]
\centering
\includegraphics[width=\textwidth]{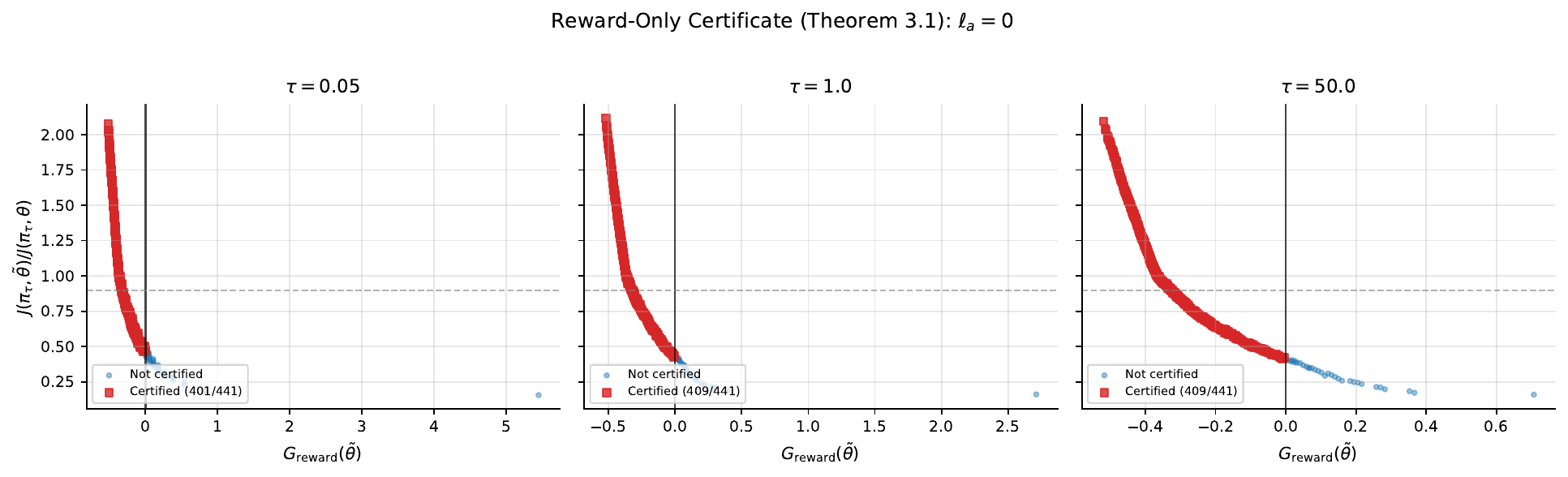}
\caption{Reward-only certificate (Theorem~\ref{thm:exact_reward_only_robustness},
  $\ell_\delta = 0$) vs empirical performance ratio. The reward-only
  certificate covers ${\sim}93\%$ of the perturbation grid
  with zero violations, confirming that the policy is inherently
  robust to reward misspecification.}
\label{fig:cert_reward_scatter}
\end{figure}

Figure~\ref{fig:cert_reward_tau} shows the reward-only certified
region size across the full $\tau$ sweep. While the joint
certificate~$G$ covers only $5$--$24$ cells, the reward-only
certificate covers ${\sim}400$ cells across the entire
temperature range, with the gap to the empirical robust region
(${\sim}255$ cells) closing to a factor of ${\sim}1.1$.

\begin{figure}[H]
\centering
\includegraphics[width=0.6\textwidth]{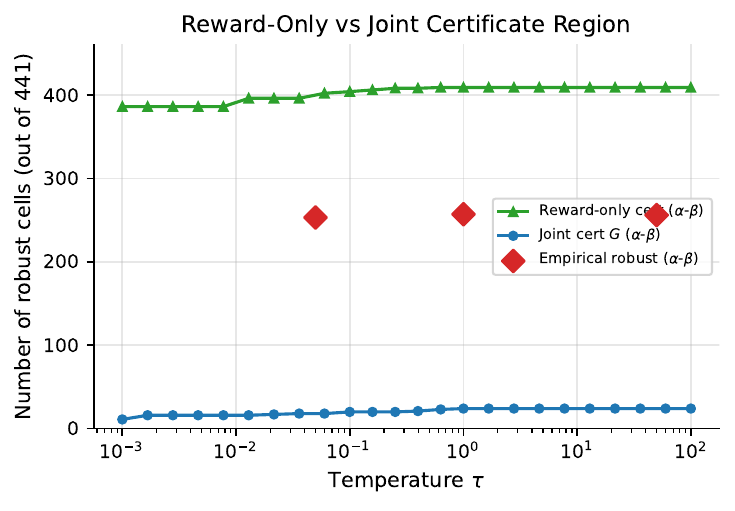}
\caption{Certified region size vs $\tau$: reward-only certificate
  (green) vs joint $G$ (blue). The large gap confirms that
  certificate conservatism is entirely driven by the transition
  KL term.}
\label{fig:cert_reward_tau}
\end{figure}

This analysis directly identifies the source of conservatism:
the local relative entropy of logistic fill intensities grows
quadratically with parameter deviation, producing large KL values
even for moderate changes in $(\alpha_\Lambda, \beta_\Lambda)$. Tightening the
transition bound---e.g., via action-dependent KL budgets---would
directly enlarge the certified region.

% ------------------------------------------------------------------
\subsection{Global Certificate Tightness}\label{ssec:cert_global}
% ------------------------------------------------------------------

Theorem~\ref{thm:ct_global_sa_robust} defines a global
occupancy-based certificate $\Phi_{\mathrm{glob}}$ that weights
by the occupancy measure under perturbed dynamics, potentially
yielding a less conservative bound than the local $G$ which takes
a maximum over states. We also consider the occupancy-weighted
local certificate $G_w = \sum_q \mu_0(q) F(q) - r\log(1/r)$,
which uses the \emph{nominal} occupancy as weights (so that
$G_w \leq G$ by construction).

Table~\ref{tab:cert_comparison} compares all certificate variants.
The occupancy-weighted certificates certify modestly more cells
than the state-max $G$: at $\tau = 1.0$,
$G_w$ certifies $25$ cells and
$\Phi_{\mathrm{glob}}$ certifies $25$ cells, compared to
$24$ for~$G$.

\begin{table}[H]
\centering
\caption{Certified region sizes (number of grid cells) on the
  $\alpha_\Lambda$-$\beta_\Lambda$ plane for different certificate variants.
  The empirical robust region (defined as $J/J_{\mathrm{nom}} \geq 0.9$)
  is shown for reference.}
\label{tab:cert_comparison}
\small
\begin{tabular}{lccc}
\toprule
Certificate & $\tau = 0.05$ & $\tau = 1.0$ & $\tau = 50.0$ \\
\midrule
$G$ (state-max, Thm~\ref{thm:maxent_robust_ctmdp})
  & 18 & 24 & 24 \\
$G_w$ (occupancy-weighted)
  & 19 & 25 & 26 \\
$\Phi_{\mathrm{glob}}$ (Thm~\ref{thm:ct_global_sa_robust})
  & 18 & 25 & 26 \\
$G_{\mathrm{rw}}$ (reward-only, Thm~\ref{thm:exact_reward_only_robustness})
  & 401 & 409 & 409 \\
Empirical ($J/J_{\mathrm{nom}} \geq 0.9$)
  & 253 & 257 & 256 \\
\bottomrule
\end{tabular}
\end{table}

\begin{figure}[H]
\centering
\includegraphics[width=\textwidth]{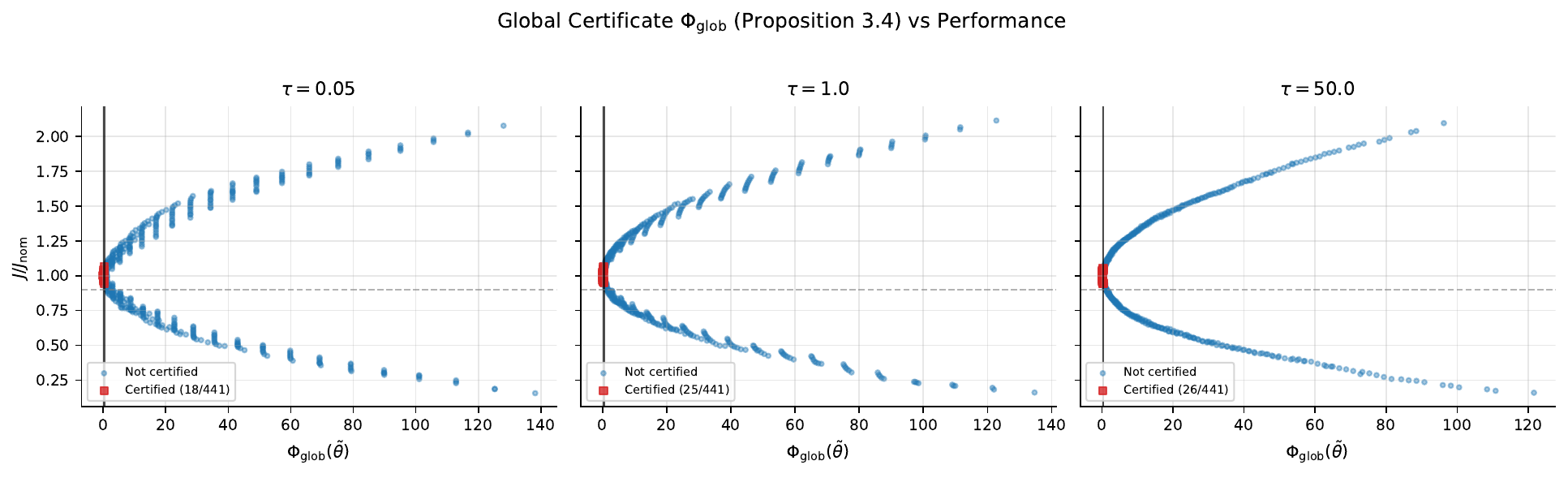}
\caption{Global certificate $\Phi_{\mathrm{glob}}$ vs empirical
  performance ratio. The vertical line marks the certificate
  threshold $r \log(1/r) \approx 0.37$. All certified points
  ($\Phi_{\mathrm{glob}} \leq r\log(1/r)$, red squares) achieve
  $J/J_{\mathrm{nom}} > 0.94$.}
\label{fig:cert_global_scatter}
\end{figure}

\begin{figure}[H]
\centering
\includegraphics[width=0.6\textwidth]{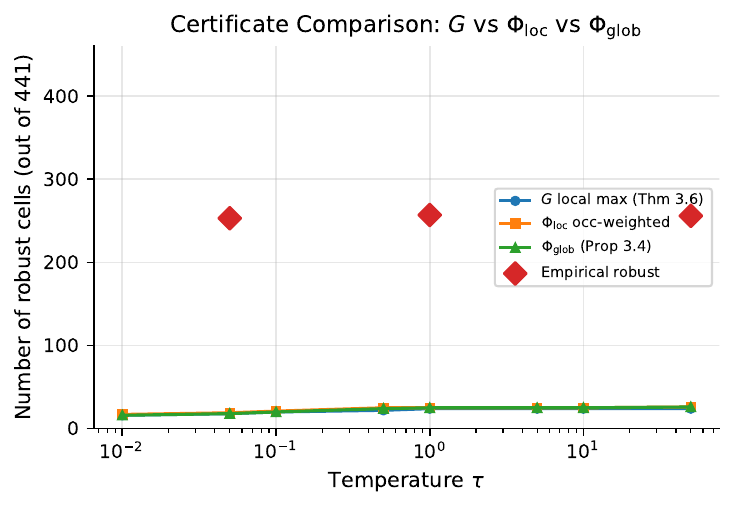}
\caption{Certified region size vs $\tau$ for three certificate
  variants: $G$ (state-max, blue), $G_w$
  (occupancy-weighted, orange), and $\Phi_{\mathrm{glob}}$
  (global, green). Green diamonds show the empirical robust
  region. All three certificates show monotone growth, with
  the occupancy-weighted variants certifying slightly more cells.}
\label{fig:cert_global_tau}
\end{figure}

The modest improvement from occupancy weighting reflects the
structure of the inventory-stabilising policy: because the
softmax policy penalises extreme inventories, the occupancy
measure concentrates near $q = 0$, where the $F$-values are
similar to their maximum. In environments where occupancy
is more dispersed (e.g., with weaker inventory penalties),
the global certificate would provide a larger improvement.

\begin{figure}[H]
\centering
\includegraphics[width=0.55\textwidth]{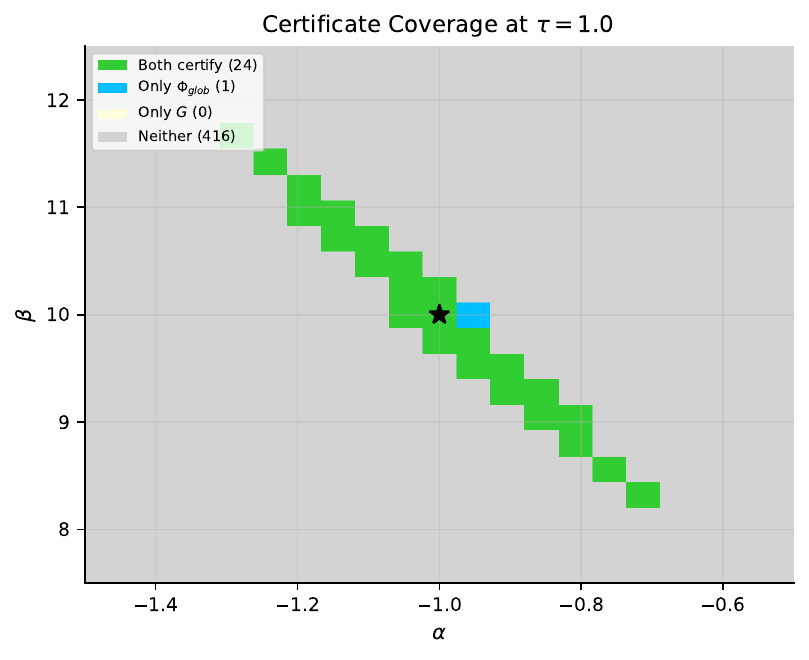}
\caption{Certificate coverage heatmap at $\tau = 1.0$.
  Green cells are certified by both $G$ and
  $\Phi_{\mathrm{glob}}$; blue cells are certified only by
  $\Phi_{\mathrm{glob}}$. The black star marks the nominal
  operating point.}
\label{fig:cert_global_heatmap}
\end{figure}

Figure~\ref{fig:cert_state_scatter} further confirms that
$G_w \leq G$ pointwise across all 441 grid cells, with most
points clustering near the diagonal---indicating that the
worst-case state dominates the average. That the improvement
from occupancy-weighting is small (only $1$--$2$ additional cells)
confirms that the conservatism of~$G$ is \emph{not} driven by the
max-over-states aggregation. Rather, the bottleneck lies in the
per-state bound $F(q)$, which aggregates the transition KL
uniformly over all actions.

\begin{figure}[H]
\centering
\includegraphics[width=0.5\textwidth]{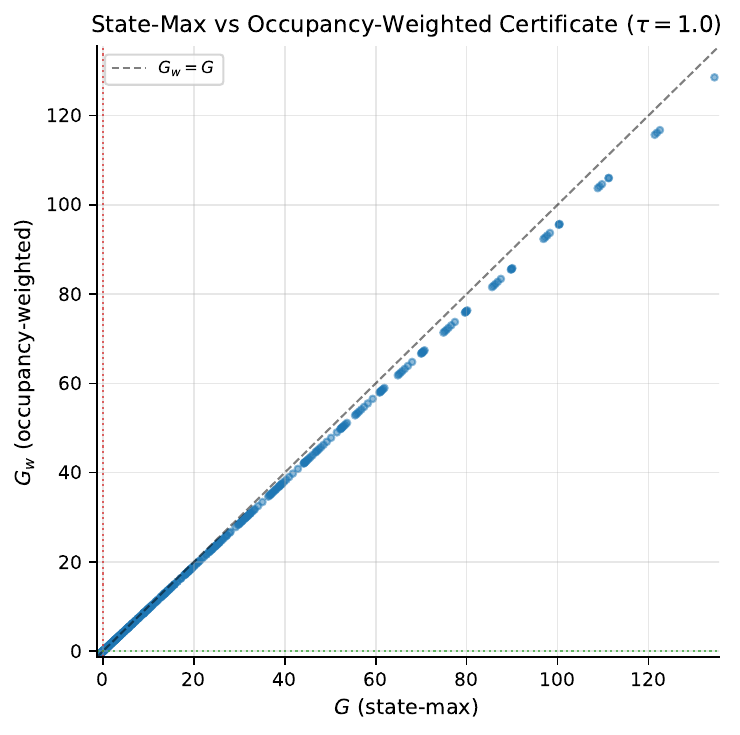}
\caption{$G$ (state-max) vs $G_w$ (occupancy-weighted) for all
  $441$ grid cells at $\tau = 1.0$. All points lie below the
  diagonal ($G_w \leq G$), confirming the theoretical ordering.
  Most points cluster near the diagonal, indicating that the
  worst-case state dominates the average.}
\label{fig:cert_state_scatter}
\end{figure}

\begin{figure}[H]
\centering
\includegraphics[width=0.55\textwidth]{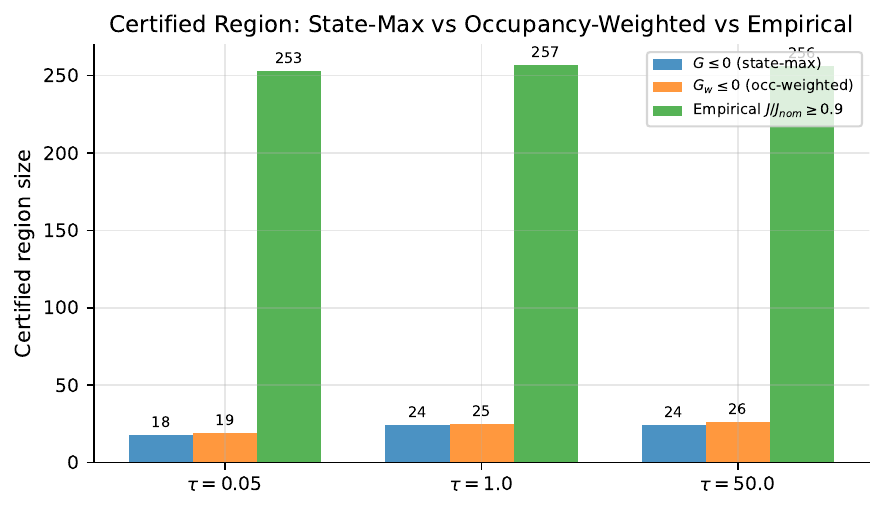}
\caption{Certified region sizes for $G$ (state-max), $G_w$
  (occupancy-weighted), and the empirical robust region.
  Occupancy-weighting provides a small improvement
  ($1$--$2$ additional cells), but the ${\sim}10\times$ gap
  to the empirical region remains.}
\label{fig:cert_state_bar}
\end{figure}

% ------------------------------------------------------------------
\subsection{Policy-Weighted KL Decomposition}\label{ssec:cert_policy_kl}
% ------------------------------------------------------------------

The certificate computes an effective per-state KL contribution
via the soft-max aggregation
$\ell_{\mathrm{cert}}(q) = \tau \bigl(\mathrm{lse}(\ell/\tau) -
\log|\mathcal{A}|\bigr)$, which aggregates over all actions.
The actual policy $\pi_\tau(\cdot \mid q) =
\mathrm{softmax}(R_\theta(q, \cdot)/\tau)$ concentrates mass on
high-reward (often low-KL) actions, so the effective KL
experienced by the policy is
$\ell_\pi(q) = \sum_\delta \pi(\delta \mid q) \, \ell_\delta(q)$.

\begin{figure}[H]
\centering
\includegraphics[width=0.7\textwidth]{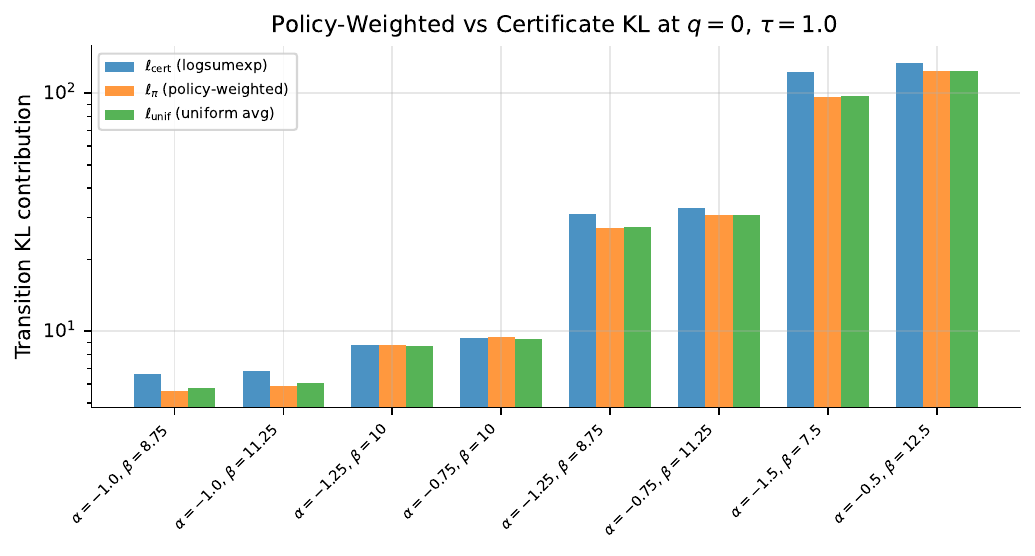}
\caption{Comparison of three KL aggregation methods at $q = 0$
  ($\tau = 1.0$) across $8$ representative perturbation points.
  The certificate's logsumexp aggregation (blue) exceeds the
  policy-weighted KL (orange) and the uniform average (green),
  but the amplification is modest ($1$--$2.5\times$).}
\label{fig:cert_kl_comparison}
\end{figure}

\begin{figure}[H]
\centering
\includegraphics[width=0.55\textwidth]{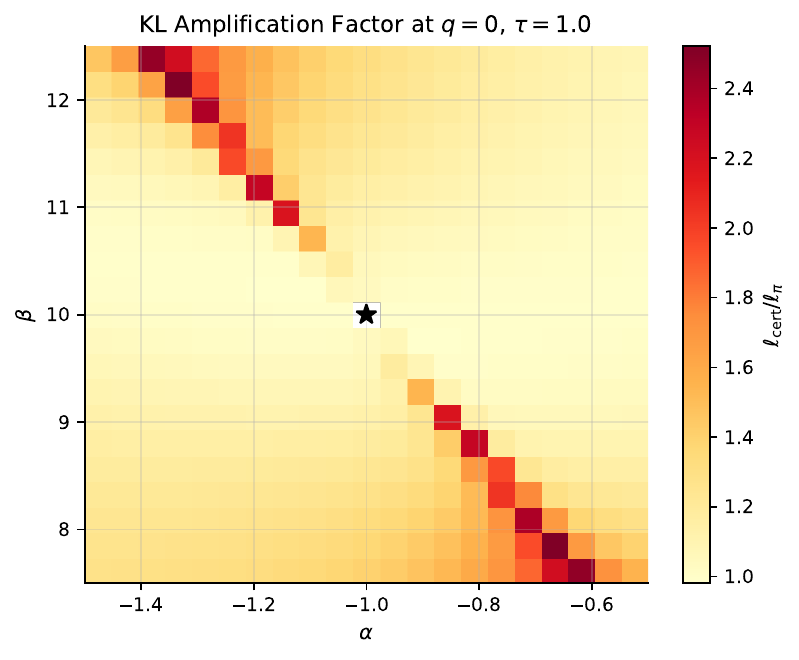}
\caption{KL amplification ratio
  $\ell_{\mathrm{cert}} / \ell_\pi$ across the
  $\alpha_\Lambda$-$\beta_\Lambda$ grid at $q = 0$, $\tau = 1.0$.
  The ratio ranges from $1.0$ (at small perturbations)
  to $2.5$ (at grid extremes). The black star marks the
  nominal point.}
\label{fig:cert_kl_heatmap}
\end{figure}

The amplification factor is modest ($1.0$--$2.5\times$),
indicating that at $\tau = 1.0$ the softmax policy distributes
weight broadly enough that the logsumexp aggregation closely
approximates the policy-weighted KL. At smaller $\tau$
(more concentrated policies), the amplification would be larger.
Combined with the finding that state-max and occupancy-weighted
certificates differ minimally (\S\ref{ssec:cert_global}),
this localises the certificate's conservatism to the inherent
gap between the local relative entropy of logistic intensities
and the actual performance degradation---a gap that could be
closed with distribution-dependent KL bounds.

% ------------------------------------------------------------------
\subsection{Action Space Convergence}\label{ssec:cert_action_space}
% ------------------------------------------------------------------

Example~\ref{example:robust set} shows that for discrete-time
robust sets, the certificate degenerates as the action space
grows: the discrete-time KL penalty scales with $|\mathcal{A}|$,
eventually making the robust set empty. In continuous time,
the normalisation $\log|\mathcal{A}|$ in the certificate formula
compensates for the action space size. We verify this
numerically by computing the certified region for action grids
ranging from $|\mathcal{A}| = 25$ ($5 \times 5$) to
$|\mathcal{A}| = 10{,}000$ ($100 \times 100$).

\begin{figure}[H]
\centering
\includegraphics[width=0.55\textwidth]{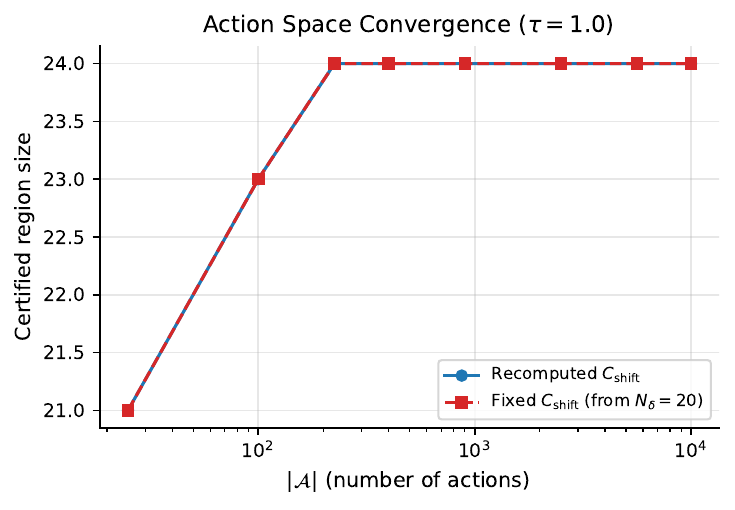}
\caption{Certified region size vs action space cardinality at
  $\tau = 1.0$. Both the recomputed $C_{\mathrm{shift}}$
  (blue) and fixed $C_{\mathrm{shift}}$ (red) variants converge
  by $|\mathcal{A}| \approx 225$, confirming that the
  continuous-time certificate remains non-degenerate.}
\label{fig:cert_action_conv}
\end{figure}

\begin{figure}[H]
\centering
\includegraphics[width=0.55\textwidth]{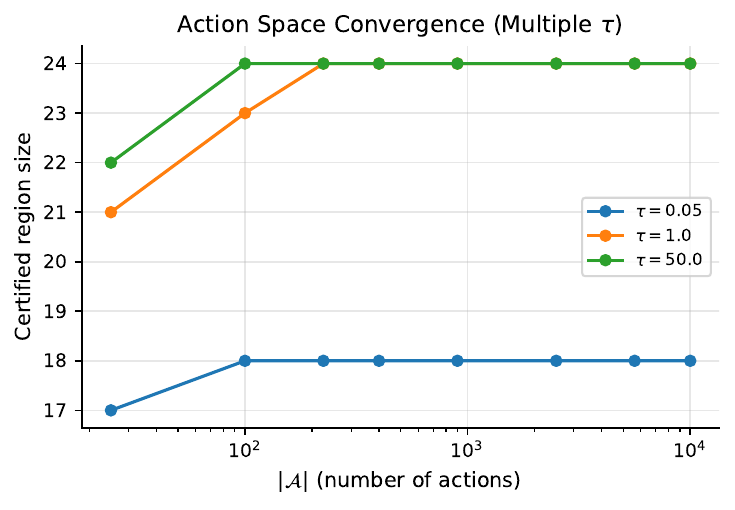}
\caption{Action space convergence for three entropy temperatures.
  All three converge rapidly and plateau at their respective
  certified region sizes ($18$--$24$ cells). The certificate
  is non-degenerate across $400\times$ variation in
  $|\mathcal{A}|$.}
\label{fig:cert_action_multi_tau}
\end{figure}

The reward shift $C_{\mathrm{shift}}$ is constant at $38.66$
across all action grid sizes because the minimum reward always
occurs at the grid boundary $\delta = \delta_{\min} = 0.12$
with extreme inventory $q = \pm 10$.

\begin{figure}[H]
\centering
\includegraphics[width=0.55\textwidth]{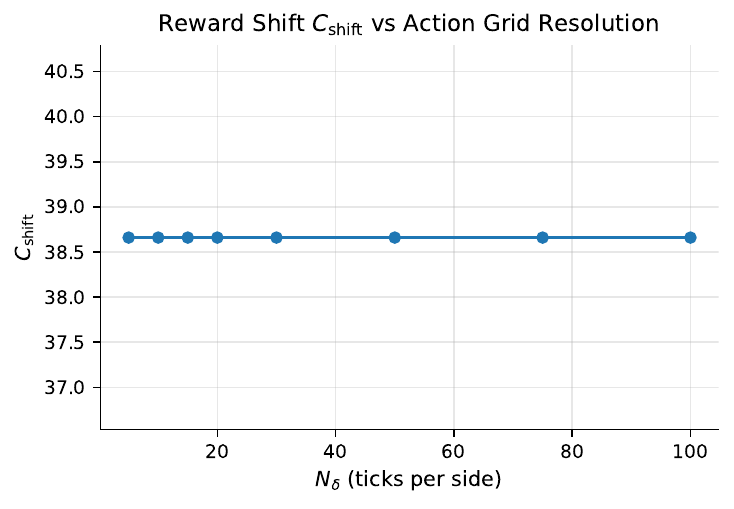}
\caption{Reward shift $C_{\mathrm{shift}}$ vs action grid
  resolution $N_\delta$. The shift is constant because the
  minimum reward is determined by the grid boundary, not the
  interior resolution.}
\label{fig:cert_action_cshift}
\end{figure}

The convergence is rapid: the certified region stabilises
by $N_\delta = 15$ ($|\mathcal{A}| = 225$) for all $\tau$
values. This confirms the theoretical advantage of the
continuous-time certificate formulation over its discrete-time
counterpart: the $\log|\mathcal{A}|$ normalisation ensures
that refining the action discretisation does not degrade the
certificate, in stark contrast to the degeneration described
in Example~\ref{example:robust set}.

% ==================================================================
\section{Event-Based vs Fixed-Grid Time discretization}\label{app:discretization_extended}
% ==================================================================

This section presents empirical evidence comparing two approaches for
implementing continuous-time environments as RL training
simulators: \emph{fixed-grid} discretization, which evaluates the
policy at uniformly spaced time steps $t_k = k \Delta t$, and
\emph{event-based} (arrival-driven) discretization, which advances
the simulation to the next stochastic event (e.g., order arrival,
queue transition) and evaluates the policy only at event times.  We
argue that event-based discretization is preferable on three grounds:
(i)~it eliminates a sensitive hyperparameter ($\Delta t$),
(ii)~it avoids a fundamental bias--variance trade-off that makes
fixed-grid performance non-monotone in $\Delta t$, and
(iii)~its computational cost scales with the event rate of the
underlying process rather than an artificial grid resolution.

% ------------------------------------------------------------------
\subsection{The discretization Trade-Off}\label{ssec:disc_tradeoff}
% ------------------------------------------------------------------

In a fixed-grid implementation with step size $\Delta t$, the agent
observes the state and selects an action at each time
$t_k = k \Delta t$, $k = 0, 1, \ldots, \lfloor T / \Delta t \rfloor$.
The choice of $\Delta t$ introduces a trade-off:
\begin{itemize}
  \item \textbf{Large $\Delta t$ (coarse grid):} The agent makes
    fewer decisions per episode, reducing the action space temporally
    but introducing \emph{discretization error}: the policy cannot
    react to events that occur between grid points.
  \item \textbf{Small $\Delta t$ (fine grid):} discretization error
    decreases, but the number of steps per episode grows as
    $T / \Delta t$, increasing the trajectory length.  This
    exacerbates the \emph{credit assignment problem}: each individual
    reward signal becomes smaller (scaling as $\Delta t$), and the
    agent must propagate information over more steps to learn
    long-horizon behaviour.  Variance in policy gradient estimates
    also increases with trajectory length.
\end{itemize}

This trade-off implies that there exists an intermediate $\Delta t$
at which fixed-grid RL performance peaks---and this optimum is
\emph{a priori} unknown to the practitioner.  Event-based
discretization sidesteps this trade-off entirely by operating at the
natural time scale of the environment: the agent acts only when
something happens.

% ------------------------------------------------------------------
\subsection{Market-Making Experiment}\label{ssec:disc_mm}
% ------------------------------------------------------------------

We train policy gradient (PG) and advantage actor-critic (A2C)
agents on the single-asset market-making environment
(cf.~\S\ref{sssec:robust_mm_setup}) under both event-based and
fixed-grid discretizations.  For fixed-grid, we vary the number of
steps per episode: $n \in \{50, 100, 200, 500, 1000\}$,
corresponding to $\Delta t = T/n$.  All agents use the same network
architecture (MLP $[64, 64]$), learning rate ($10^{-3}$), and are
trained for 4{,}000 epochs with 5 seeds each.  The evaluation metric
is the per-episode total reward (undiscounted), which is comparable
across discretizations.

Table~\ref{tab:disc_mm} reports converged per-episode rewards. The
Asmussen--Scanlan (AS) analytical optimal provides a reference upper
bound ($J^* = 68.13$).

\begin{table}[H]
\centering
\caption{Market-making: per-episode total reward (5 seeds, last 200
  training epochs).  $\%J^*$ is the fraction of the AS optimal
  achieved.  Event-driven results are from separate policy evaluation
  since per-step training rewards require knowledge of the
  (stochastic) number of events per episode.}
\label{tab:disc_mm}
\small
\begin{tabular}{llccr}
\toprule
discretization & Algorithm & $J$ (mean $\pm$ SE) & $\%J^*$ & $n$ steps \\
\midrule
Fixed-grid ($n{=}200$)  & PG  & $66.41 \pm 0.37$ & $97.5\%$ & 200 \\
Event-driven            & PG  & $64.88 \pm 0.36$ & $95.2\%$ & variable \\
Fixed-grid ($n{=}500$)  & PG  & $62.30 \pm 0.62$ & $91.5\%$ & 500 \\
Fixed-grid ($n{=}1000$) & PG  & $52.54 \pm 0.38$ & $77.1\%$ & 1000 \\
Fixed-grid ($n{=}100$)  & PG  & $47.48 \pm 0.31$ & $69.7\%$ & 100 \\
Fixed-grid ($n{=}50$)   & PG  & $24.04 \pm 0.06$ & $35.3\%$ & 50 \\
\midrule
Fixed-grid ($n{=}200$)  & A2C & $60.53 \pm 0.06$ & $88.9\%$ & 200 \\
Event-driven            & A2C & $57.04 \pm 0.22$ & $83.7\%$ & variable \\
Fixed-grid ($n{=}500$)  & A2C & $54.62 \pm 0.14$ & $80.2\%$ & 500 \\
Fixed-grid ($n{=}1000$) & A2C & $50.42 \pm 0.07$ & $74.0\%$ & 1000 \\
Fixed-grid ($n{=}100$)  & A2C & $44.86 \pm 0.09$ & $65.8\%$ & 100 \\
Fixed-grid ($n{=}50$)   & A2C & $22.47 \pm 0.07$ & $33.0\%$ & 50 \\
\bottomrule
\end{tabular}
\end{table}

Several observations are immediate.  First, performance is strongly
non-monotone in $n$ for fixed-grid agents: both PG and A2C peak at
$n = 200$ ($\Delta t = 0.05$), with substantial degradation for both
coarser and finer grids.  At $n = 1000$ ($\Delta t = 0.01$), PG
achieves only $77.1\%$ of $J^*$, a $20$-percentage-point drop from
the peak---despite having five times finer temporal resolution.
Second, event-driven agents achieve competitive performance without
requiring any grid resolution tuning: event-driven PG attains
$95.2\%$ of $J^*$, which places it between the best ($n = 200$) and
second-best ($n = 500$) fixed-grid configurations.  An RL
practitioner who does not know the optimal $n$ in advance risks
choosing $n = 500$ or $n = 1000$, both of which significantly
underperform event-driven PG.

% ------------------------------------------------------------------
\subsection{discretization Sweep: Wall-Clock Cost}\label{ssec:disc_wallclock}
% ------------------------------------------------------------------

To quantify the computational trade-off, we conduct a dedicated
sweep of A2C agents across six $\Delta t$ values with identical
training budgets (2{,}000 epochs, 5 seeds).  We record wall-clock
training time on the same hardware.

\begin{table}[H]
\centering
\caption{A2C discretization sweep: per-episode total reward and
  wall-clock training time (5 seeds, 2{,}000 epochs).}
\label{tab:disc_sweep}
\small
\begin{tabular}{lcccc}
\toprule
Method & $\Delta t$ & $J$ (mean $\pm$ SE) & $\%J^*$ & Time (s) \\
\midrule
FG A2C & $0.02$   & $22.46 \pm 0.06$ & $33.0\%$ & $132$ \\
FG A2C & $0.01$   & $44.68 \pm 0.11$ & $65.6\%$ & $239$ \\
FG A2C & $0.005$  & $60.46 \pm 0.14$ & $88.7\%$ & $527$ \\
FG A2C & $0.002$  & $54.52 \pm 0.10$ & $80.0\%$ & $1{,}296$ \\
FG A2C & $0.001$  & $50.07 \pm 0.06$ & $73.5\%$ & $2{,}763$ \\
FG A2C & $0.0005$ & $46.51 \pm 0.05$ & $68.3\%$ & $6{,}039$ \\
\midrule
Arrival-driven A2C & --- & $57.04 \pm 0.22$ & $83.7\%$ & $876$ \\
\bottomrule
\end{tabular}
\end{table}

The non-monotone pattern is stark: performance peaks at
$\Delta t = 0.005$ ($88.7\%$ of $J^*$) and declines for both
coarser and finer grids.  The finest grid ($\Delta t = 0.0005$)
achieves only $68.3\%$ of $J^*$ while requiring $6{,}039$ seconds
of training---$11.5\times$ slower than the peak and
$6.9\times$ slower than arrival-driven.  The arrival-driven agent
achieves $83.7\%$ of $J^*$ at $876$ seconds, placing it favourably
on the performance--cost Pareto frontier (see
Figure~\ref{fig:disc_pareto}).

\begin{figure}[H]
\centering
\begin{subfigure}[t]{0.48\textwidth}
  \centering
  \includegraphics[width=\textwidth]{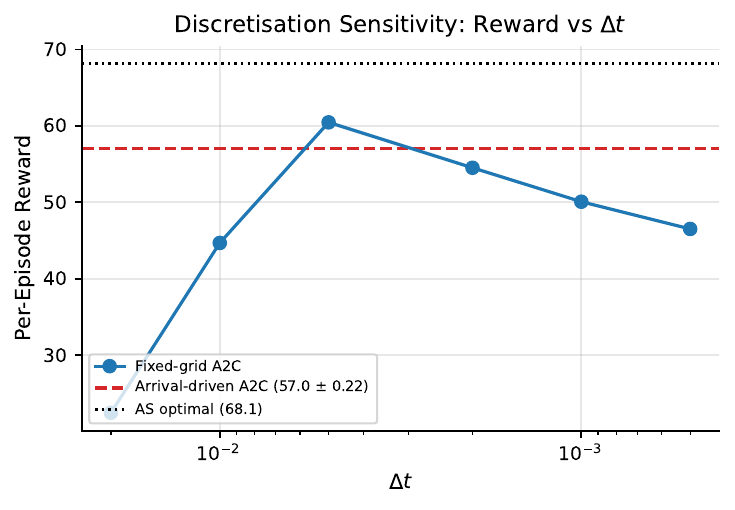}
  \caption{Per-episode reward vs $\Delta t$. Performance is
    non-monotone, peaking at an intermediate $\Delta t$ that is
    unknown \emph{a priori}.}
  \label{fig:disc_reward_dt}
\end{subfigure}
\hfill
\begin{subfigure}[t]{0.48\textwidth}
  \centering
  \includegraphics[width=\textwidth]{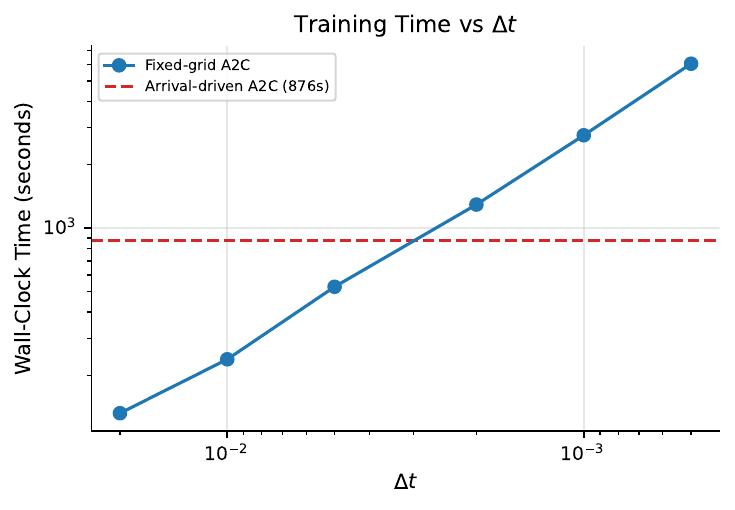}
  \caption{Wall-clock training time increases monotonically as
    $\Delta t$ decreases. Finer grids are both slower and worse.}
  \label{fig:disc_wallclock_dt}
\end{subfigure}
\caption{Fixed-grid A2C discretization sweep on the market-making
  environment (5 seeds, 2{,}000 epochs).  The dashed red line shows
  the arrival-driven A2C baseline.  Finer temporal resolution
  ($\Delta t \to 0$) does not improve RL performance; instead, the
  credit assignment problem dominates, and the agent fails to learn
  effective policies despite much higher computational cost.}
\label{fig:disc_sweep}
\end{figure}

\begin{figure}[H]
\centering
\includegraphics[width=0.55\textwidth]{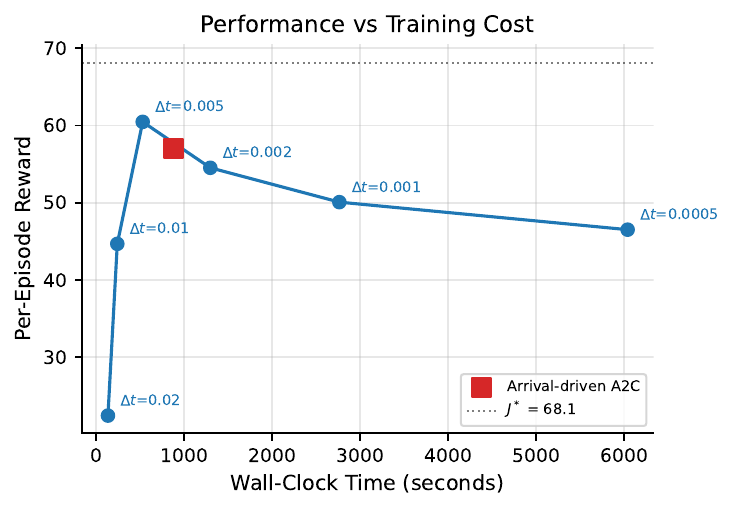}
\caption{Performance vs training cost for fixed-grid A2C at
  various $\Delta t$ (blue circles) and arrival-driven A2C (red
  square).  The arrival-driven agent achieves strong performance
  without requiring grid resolution tuning.  No fixed-grid
  configuration simultaneously dominates the arrival-driven agent
  in both reward and training time by a large margin.}
\label{fig:disc_pareto}
\end{figure}

% ------------------------------------------------------------------
\subsection{Queueing Scheduling Experiment}\label{ssec:disc_queue}
% ------------------------------------------------------------------

We also compare discretizations on the criss-cross queueing network
(cf.~\S\ref{sssec:robust_queue_setup}), with arrival rates
$\lambda = (1.5, 1.5, 1.5)$ and service rates
$\mu = (2.0, 3.0, 2.0)$.  The holding costs are
$h = (3.0, 1.0, 2.0)$.  Two analytical baselines provide reference
points: the $c\mu$-rule ($J_{c\mu} = -1163.3$) and the MaxPressure
policy ($J_{\mathrm{MP}} = -1276.4$), where less negative is better.

\begin{table}[H]
\centering
\caption{Queueing scheduling: per-episode total reward (5 seeds,
  2{,}000 epochs).  Analytical baselines are shown for reference.}
\label{tab:disc_queue}
\small
\begin{tabular}{lcr}
\toprule
Method & $J$ (mean $\pm$ SE) & \\
\midrule
FG A2C ($\Delta t = 0.1$)  & $-1{,}105 \pm 3$ & \\
Event-driven PG            & $-1{,}139 \pm 5$ & \\
$c\mu$-rule (analytical)   & $-1{,}163$ & \\
Event-driven A2C           & $-1{,}251 \pm 10$ & \\
MaxPressure (analytical)   & $-1{,}276$ & \\
FG A2C ($\Delta t = 0.01$) & $-1{,}337 \pm 4$ & \\
\bottomrule
\end{tabular}
\end{table}

The same non-monotone pattern appears: a $10\times$ refinement
of the grid from $\Delta t = 0.1$ to $\Delta t = 0.01$ degrades
performance by $21\%$ (from $-1{,}105$ to $-1{,}337$).  The
event-driven PG agent achieves $-1{,}139$, outperforming the
$c\mu$-rule analytical baseline and falling between the two
fixed-grid A2C configurations.

% ------------------------------------------------------------------
\subsection{Discussion}\label{ssec:disc_discussion}
% ------------------------------------------------------------------

The experiments above support three conclusions about time
discretization for continuous-time RL environments:

\paragraph{1. Fixed-grid performance is non-monotone in $\Delta t$.}
In both market making and queueing, performance peaks at an
intermediate $\Delta t$ and degrades for both coarser and finer
grids.  This is not a training budget artefact: all configurations
receive the same number of epochs.  The root cause is a trade-off
between \emph{discretization bias} (large $\Delta t$, too few
decisions) and \emph{credit assignment difficulty} (small $\Delta t$,
too many low-signal steps per episode).  The optimal $\Delta t$
depends on the environment dynamics and is unknown a priori, making
it a difficult hyperparameter to tune.

\paragraph{2. Event-based discretization eliminates $\Delta t$.}
The arrival-driven implementation has no grid resolution
hyperparameter.  It automatically operates at the natural time scale
of the stochastic process, making decisions only when events occur.
This yields competitive performance without tuning: event-driven PG
achieves $95.2\%$ of $J^*$ in market making (vs.\ $97.5\%$ for the
best-tuned fixed grid) and outperforms the $c\mu$-rule in queueing.
An RL practitioner using event-driven discretization avoids the risk
of accidentally choosing a suboptimal $\Delta t$ (e.g., $n = 1000$
in market making achieves only $77\%$ of $J^*$).

\paragraph{3. Computational cost scales with physics.}
In the event-based approach, the number of steps per episode
is determined by the event rate of the environment (e.g., order
arrival intensity in market making, transition rates in queueing),
not by an artificial grid.  This avoids wasting computation on
time intervals where no events occur.  In the market-making
discretization sweep, the finest grid ($\Delta t = 0.0005$) requires
$6{,}039$s of training and achieves only $68.3\%$ of $J^*$, while
arrival-driven A2C uses $876$s and achieves $83.7\%$.  More
generally, event-based cost grows linearly with the total event rate,
providing a natural computational scaling.

%%%%%%%%%%%%%%%%%%%%%%%%%%%%%%%%%%%%%%%%%%%%%%%%%%%%%%%%%%%%

\end{document}